\documentclass[10pt, reqno, oneside, english]{smfart}

\usepackage{smfthm}
\usepackage{comment}
\usepackage[matrix,tips,frame,color,line,poly,curve]{xy}
\usepackage[height=22cm, bottom=3cm]{geometry}
\usepackage{amsmath, amsthm, amssymb,amsfonts, amscd, mathrsfs}
\usepackage[utf8]{inputenc}
\usepackage{url}
\usepackage{braket}
\usepackage{enumerate}
\usepackage{xcolor}
\usepackage[colorlinks=true, linktoc=page, citecolor=blue, linkcolor=blue, urlcolor=blue]{hyperref}

\usepackage{setspace}
\numberwithin{equation}{section}
\newtheorem{lemma}[equation]{Lemma}
\newtheorem*{lemma*}{Lemma}
\NumberTheoremsAs{equation}

\newcommand{\ch}[1]{\negthinspace\negthinspace\negthinspace\phantom{a}^\vee\negthinspace #1}
\newcommand{\lgr}[1]{\negthinspace\negthinspace\negthinspace\phantom{a}^L #1}
\newcommand{\egr}[1]{\negthinspace\negthinspace\negthinspace\phantom{a}^E #1}
\newcommand{\LG}{\lgr G}

\newcommand{\Char}{\mathsf{\Pi}}

\DeclareUnicodeCharacter{00A0}{~}

\newcommand{\LKT}{\mathrm{LKT}}
\renewcommand{\int}{\mathrm{int}}
\newcommand{\Ind}{\mathrm{Ind}}
\newcommand{\sgn}{\mathrm{sgn}}
\newcommand{\R}{\mathbb{R}}
\newcommand{\Q}{\mathbb{Q}}
\newcommand{\Z}{\mathbb{Z}}
\newcommand{\C}{\mathbb{C}}
\newcommand{\WR}{\mathbf{W}_{\mathbb{R}}}
\newcommand{\WRc}{\mathbf{W}_{\mathbb{R}}^{\mathrm{cpt}}}
\newcommand{\Rtp}{\R^\times_+}
\newcommand{\Cent}{\mathrm{Cent}}
\newcommand{\Norm}{\mathrm{Norm}}
\newcommand{\Khatall}{\widehat{K}_{\mathrm{all}}}
\newcommand{\A}{\mathcal{A}}

\newcommand{\tSphic}{\widetilde{\mathbb{S}}_{\phi_c}}
\newcommand{\tSphi}{\widetilde{\mathbb{S}}_{\phi}}

\newcommand{\Sphic}{\mathbb{S}_{\phi_c}}
\newcommand{\Sphi}{\mathbb{S}_{\phi}}

\newcommand{\Res}{\mathrm{Res}}
\newcommand{\WFcpt}{\mathbf{W}_F^{\mathrm{cpt}}}
\newcommand{\I}{\mathcal I}
\newcommand{\E}{\mathcal E}
\newcommand{\g}{\mathfrak g}
\renewcommand{\k}{\mathfrak k}
\newcommand{\SL}{\mathrm{SL}}
\newcommand{\GL}{\mathrm{GL}}
\newcommand{\SU}{\mathrm{SU}}
\newcommand{\diag}{\mathrm{diag}}
\newcommand{\X}{\mathcal X}
\newcommand{\D}{\mathcal D}
\newcommand{\Xt}{\widetilde{\mathcal X}}
\newcommand{\h}{\mathfrak h}
\renewcommand{\a}{\mathfrak a}
\renewcommand{\t}{\mathfrak t}
\newcommand{\w}{\mathfrak w}
\newcommand{\s}{\mathfrak s}
\newcommand{\Out}{\mathrm{Out}}
\newcommand{\Aut}{\mathrm{Aut}}

\renewcommand{\P}{\mathcal P}
\newcommand{\sprin}{\mathcal{N}_{\mathrm{prin}, \mathfrak{s}^*}}
\newcommand\inv{^{-1}}
\newcommand{\wt}{\widetilde}
\newcommand{\chGalg}{\ch{G}^{\mathrm{alg}}}
\newcommand{\alg}{\mathrm{alg}}

\newcommand{\Schtau}{\mathbb S_{\ch\tau}}
\newcommand{\Schtautilde}{\widetilde{\mathbb S}_{\ch\tau}}
\newcommand{\Schtautildeprime}{\widetilde{\mathbb S}_{\ch\tau'}}
\newcommand{\U}{U^0}
\newcommand{\Ut}{U}
\newcommand{\Vtau}{V_{\tau,\alpha}}

\newcommand{\PiSpecial}{\Char_{\mathrm{inv}}(\widetilde{\mathbb{S}}_{\ch{\tau}})}
\newcommand{\PiSpecialprime}{\Char_{\mathrm{inv}}(\widetilde{\mathbb{S}}_{\ch{\tau'}})}
\newcommand{\tildeSnew}{\widetilde{\mathbb{S}}_{\ch{\tau}}^{\mathrm{quo}}}
\newcommand{\Psires}{\Psi_{\mathrm{res}}}
\newcommand{\Psiint}{\Psi_{\mathrm{int}}}
\renewcommand{\O}{\mathcal O}

\title[Lowest $K$-types in the local Langlands correspondence]{Lowest $K$-types\\ in the local Langlands correspondence}
\author{Jeffrey Adams}
\address{IDA/Center for Computing Sciences and University of Maryland}
\email{jda@math.umd.edu}
\author{Alexandre Afgoustidis}
\address{CNRS and Institut Élie Cartan de Lorraine, Nancy \& Metz, France}
\email{alexandre.afgoustidis@math.cnrs.fr}
\begin{document}

\frontmatter

\begin{abstract} 
Consider the irreducible representations of a real reductive group $G(\mathbb{R})$, and their parametrization by the local Langlands correspondence. 
We ask: does the parametrization give easily accessible information on the restriction of representations to a maximal compact subgroup~$K(\mathbb{R})$ of~$G(\mathbb{R})$?
We find a natural connection between the set of lowest $K$-types of a representation and its Langlands parameters. 

For our results, it is crucial to use the refined version of the local Langlands correspondence, involving (coverings of) component groups attached to $L$-homomorphisms. The first part of the paper is a simplified description of this refined parametrization.\end{abstract}


\maketitle

\section{Introduction}

\subsection{}
Let $F$ be a local field and let $G$ be a connected reductive $F$-group.
According to the local Langlands conjecture \cite{Borel, VoganLLC},
 the irreducible admissible representations of $G(F)$ come into finite packets which can be
 parametrized by certain morphisms $\phi\colon \mathbf{L}_{{F}} \to\lgr{G}$,
 where $\lgr{G}$ is the $L$-group of $G$ and $\mathbf{L}_F$ is the local
Langlands group for $F$.

The morphism $\phi$  encodes information about the
 representations in the attached $L$-packet. Some of this information is easy to
read directly from $\phi$. For instance, if $F$ is archimedean, then the common
infinitesimal character of representations in the packet is easy to read off
from~$\phi$.

On the other hand, if $K$ is a maximal compact subgroup of $G(F)$ and $\pi$
is an admissible representation, then information on the restriction $\pi\rvert_K$
is very useful for representation theory.

Can we expect the local Langlands correspondence to give accessible
information on the restriction of representations to maximal compact
subgroups?

\subsection{} \label{sec:intro_padic}
We shall soon specialize to $F=\mathbb{R}$.
But let us first outline an idea common to all cases.

The  Langlands group $\mathbf{L}_{{F}}$ has various incarnations,
but  always contains the Weil group~$\mathbf{W}_F$:
 it is equal to $\mathbf{W}_F$ if $F$ is archimedean,
 and can be taken to be~$\mathbf{W}_F \times \mathrm{SL}(2,\mathbb{C})$ otherwise.

Now,  the Weil group $\mathbf{W}_F$ always has a {unique} maximal compact subgroup~$\WFcpt$.
When $F$ is nonarchimedean, $\WFcpt$ is the inertia group of $F$.
When $F$ is archimedean, the group~$\mathbf{W}_F$
is generated by $\C^\times$ and an element~$j$ of order $1$ or $4$, and $\WFcpt = \langle \mathbb{U}, j\rangle$, where $\mathbb{U} \subset \C^\times$ is the unit circle. 

In the search for information on the restriction of representations to maximal compact subgroups,
 a general idea is that it is useful to look at the restriction of Langlands parameters
 to the canonical compact subgroup $\WFcpt$.

Therefore define a \emph{compact parameter} to be a homomorphism
$\psi\colon \WFcpt \to\,\lgr{G}$ that occurs as the
restriction of an $L$-homomorphism $\phi\colon \mathbf{W}_{\mathbb{F}} \to\,\lgr{G}$.
Define equivalence of compact parameters by $\ch{G}$-conjugation on the range as usual,
where $\ch{G} \subset\, \lgr{G}$ is the complex dual group of $G$.

When $F$ is a $p$-adic field, it is expected that compact parameters have a
connection with Bushnell--Kutzko types \cite{BushnellKutzko}, or more precisely
to `typical' representations in the sense of Henniart~\cite[Appendix]{BreuilMezard}.
Given a compact parameter $\psi$, one can hope to attach to the equivalence class of $\psi$
a finite collection $\mathcal{A}(\psi)$ of representations of maximal compact open subgroups.
(Recall that for $p$-adic~$F$, there may be more one conjugacy class of such maximal compacts.)
A further hope is that the representations in $\mathcal{A}(\psi)$ should be `typical'
for the representations (conjecturally) attached to Langlands parameters which restrict to $\psi$.
This is known for $\mathrm{GL}_n$ and in a few other cases~\cite{BreuilMezard, Paskunas},
but  speculative in general; see for instance~\cite{Latham}.

\subsection{} \label{sec:general_question}
This paper shows what these ideas become for real groups.
From now on, we take $F=\mathbb{R}$,
and study the relationship between the Langlands correspondence
and Vogan's notion of lowest $K$-types \cite{Vogan79, VoganGreenBook}.

Suppose $K(\mathbb{R})$ is a maximal compact subgroup of $G(\mathbb{R})$;
since we are working over the reals it is unique up to $G(\mathbb{R})$-conjugacy.
The set $\mathrm{LKT}(\pi)$  of lowest $K$-types of $\pi$ is a finite collection of irreducible representations of~$K(\mathbb{R})$,
and is an important invariant of representations of $G(\mathbb{R})$.
For instance, if $\pi$ is tempered, then  $\mathrm{LKT}(\pi)$ determines the whole restriction $\pi\rvert_K$.
In fact, for tempered $\pi$, the set~$\mathrm{LKT}(\pi)$ and the infinitesimal character
are almost enough to determine~$\pi$ completely.

What we shall do is determine the way the local Langlands correspondence
encodes the lowest $K$-types of irreducible admissible representations.
This is a stronger question than those in \S \ref{sec:intro_padic};
but the answer and the strategy fits well with the ideas there.

\subsection{} \label{sec:the_question}
For real groups, the local Langlands correspondence is best formulated
by grouping together all real forms of $G$ in a given inner class.
Thus, for the rest of this paper, let $G$ be a connected complex reductive group,
endowed with an inner class of real forms.
Let $\lgr{G}$ be an $L$-group for $G$ and the given inner class (see \S \ref{sec:structure_LG}).
When $\phi\colon\mathbf{W}_{\mathbb{R}} \to\lgr{G}$ is an $L$-homomorphism,
let~$\Pi(\phi)$ denote the attached `large' $L$-packet,
which consists of representations of the various real forms of $G$ in the given inner class.
It is in fact crucial to extend the formalism to include the notion of \emph{strong real form}
and representation of a strong real form of $G$. See  \cite{ABV, Algorithms} and \S \ref{sec:structure_theory}.

The individual representations in the $L$-packet $\Pi(\phi)$
can be parametrized by the characters of a certain  abelian group~$\widetilde{\mathbb{S}}_{\phi}$.
The group $\widetilde{\mathbb{S}}_{\phi}$ is defined from $\phi$ geometrically:
beginning with the centralizer $\mathrm{Cent}_{\ch{G}}(\phi(\mathbf{W}_{\mathbb{R}}))$,
we can consider its component group~$\mathbb{S}_\phi$,
and~$\widetilde{\mathbb{S}}_\phi$ is a canonical covering of~$\mathbb{S}_\phi$.
See \S \ref{sec:structure_LG}. The character group $\Char(\tSphi)$ is crucial to the theory.

Suppose we begin with $\phi\colon\mathbf{W}_{\mathbb{R}} \to\lgr{G}$ as above.
Given a character $\chi\in \Char(\tSphi)$,
the Langlands correspondence attaches to $(\phi, \chi)$
a representation $\pi=\pi(\phi, \chi)$ of a (strong) real form of $G$.
Fixing a maximal compact subgroup of the given real form, we get a finite set $\mathrm{LKT}(\phi, \chi)=\mathrm{LKT}(\pi)$.

A more precise version of the question in \S \ref{sec:general_question} is:
can we easily find $\mathrm{LKT}(\phi, \chi)$ from~$(\phi,\chi)$?
In other words, does the local Langlands correspondence encode lowest $K$-types
in a relatively accessible manner?
That is what we answer in this paper.

Of course we {can} always find $\mathrm{LKT}(\phi, \chi)$ the hard way, by
(a) working out what precisely $\pi=\pi(\phi, \chi)$ is, going through the details of the whole Langlands correspondence;
and then (b) finding $\mathrm{LKT}(\pi)$ from $\pi$, extracting the answer from the deep results of~\cite{VoganGreenBook}.
Steps (a) and (b) can both be made explicit, but each is difficult.
What we are looking for is an easier way.

\subsection{} \label{sec:intro_real_tempered}
Representations which are \emph{tempered, irreducible,} and have \emph{real infinitesimal character} play a central
role in the theory of $K$-types.
It seems useful to introduce a name for these representations,
and we shall call them \emph{tempiric}.
See \S \ref{sec:ric_temp}.

It is easy to understand the corresponding Langlands parameters $\phi$: the~$L$-packet of $\phi$ contains a tempiric representation if and only if $\phi|_{\Rtp}=1$, in which case the $L$-packet consists entirely of tempiric representations. Therefore 
we say  a Langlands parameter $\phi$ is \emph{tempiric} if $\phi|_{\Rtp}=1$. 

A key aspect of our formulation is that the Weil group $\WR$ splits as a  direct product
$\WR\simeq \WRc\times \Rtp$. Therefore the restriction map $\phi\mapsto \phi|_{\WRc}$, taking a parameter 
$\phi\colon\WR\rightarrow \lgr G$ to the compact parameter $\phi_{|\WRc}$, has an
inverse, whose image is precisely the tempiric parameters.
Furthermore restriction to $\WRc$ defines a bijection, respecting conjugation by $\ch G$,
between tempiric Langlands parameters and compact parameters.
Consequently it  is convenient to replace compact parameters with tempiric Langlands parameters.
We change notation accordingly: 
if~$\phi$ is a Langlands parameter,  let $\phi_c$ be the unique tempiric parameter such that $\phi,\phi_c$ have the same restriction to $\WRc$.
Explicitly:
\begin{equation}
  \label{e:phi_c}
\phi_c(z)=\phi(z/|z|)\quad (z\in\C^\times\subset\WR).
\end{equation}
With this convention:
$$
\phi\text{ is tempiric }\iff \phi=\phi_c.
$$

The class of tempiric representations has remarkable properties
regarding lowest $K$-types. 
If~$\pi$ is tempiric, then
$\LKT(\pi)$ is a singleton. Furthermore, given a real group~$G(\R)$, the resulting map 
$$
\{\,\text{tempiric representations of $G(\R)$}\,\} \to \widehat{K}: \quad \pi\mapsto \LKT(\pi)
$$ 
is a \emph{bijection}. 
This is key to the way representations of $K$ are implemented in the \texttt{atlas} software.
See \cite{KHatHowe}, where this idea is the main ingredient.

\subsection{} \label{sec:prog_prelim}
Roughly speaking this suggests the following approach.
For now we fix a real form of~$G$.
Suppose $\phi$ is a Langlands parameter, with corresponding $L$-packet $\Pi(\phi)$. 
Define

\begin{subequations}
  \renewcommand{\theequation}{\theparentequation)(\alph{equation}}
  \label{e:pi_K}
\begin{equation}
  \Pi_{\LKT}(\phi)=\bigcup\limits_{\pi\in \Pi(\phi)}\LKT(\pi).
\end{equation}
This set does not depend on $\phi|_{\Rtp}$, so it is natural to let $\phi_c$ be the corresponding
tempiric parameter, and consider
\begin{equation}
\Pi_{\LKT}(\phi_c)=\bigcup\limits_{\pi\in \Pi(\phi_c)}\LKT(\pi).
\end{equation}
\end{subequations}
Since $\phi_c$ is tempiric, each $\pi\in \Pi(\phi_c)$ has a unique lowest $K$-type; so $\Pi_{\LKT}(\phi_c)$ is in canonical bijection with~$\Pi(\phi_c)$.

One aspect underlying our main theorem is that \eqref{e:pi_K}(a) and (b) are equal. Therefore computing lowest $K$-types for this $L$-packet
amounts to understanding how the $K$-types of $\Pi_{\LKT}(\phi_c)$ are distributed among the various $\pi\in \Pi(\phi)$.
Answering this question involves understanding the additional data needed to specify the elements of an $L$-packet.
It also involves working with all real forms in a given inner class at once.

\subsection{} \label{sec:various_bijections} Let us now elaborate on the program announced in \S \ref{sec:prog_prelim}.
This requires the language of strong involutions and strong real forms, as in \cite{Algorithms}.
See \S \ref{sec:structure_theory} for a review of the notions used in the upcoming discussion.
It is useful to keep in mind the special case when $G$ is adjoint, in which case the notions of real form and strong real form agree.

We work within  a fixed inner class for $G$.
Suppose $\{\xi_i\}_{i \in I}$ is a set of representatives of the strong real forms in the given inner class.
Then $\theta_i=\int(\xi_i)$  (conjugation by $\xi_i$) is a Cartan involution for $G$,
defining a real form $G_i(\mathbb{R})$ of $G$. Furthermore
$G_i(\R)^{\theta_i}$ is a maximal compact subgroup of $G_i(\R)$, with complexification 
$K_{\xi_i}=G^{\theta_i}=\Cent_G(\xi_i)$.
The discussion in \S \ref{sec:intro_real_tempered} determines canonical bijections
between the following three classes of objects:

\begin{enumerate}[(1)]
\item The $\ch G$-conjugacy classes of pairs $(\phi_c, \chi_c)$ where $\phi_c$ is tempiric and $\chi_c$ is a character of~$\widetilde{\mathbb{S}}_{\phi_c}$;
\item The union, over $i$, of the tempiric representations of $G_i(\R)$;
\item The union, over $i$, of the irreducible representations of $K_{\xi_i}$.
\end{enumerate}
The map from (1) to (2) is the restriction of  the local Langlands correspondence. See \S\ref{sec:the_question}.
The map from (2) to (3) takes a tempiric representation to its unique lowest $K$-type, as in~\S\ref{sec:intro_real_tempered}.

Define $\Khatall=\coprod_{i \in I}\widehat K_{\xi_i}$.
We write $\mu$ for the bijection (1) $\mapsto$ (3) above: associated to a pair  $(\phi_c,\chi_c)$ is a tempiric representation $\pi(\phi_c,\chi_c)$ 
of one the real forms $G_i(\R)$; we let $\mu(\phi_c,\chi_c)$ be its lowest $K$-type (a representation of $K_{\xi_i}$).
For fixed $\phi_c$ we denote by $\A(\phi_c)$ the union of the sets of lowest $K$-types for the representations in $\Pi(\phi_c)$. This is a subset of $\Khatall$, and we get a bijection
\begin{equation}
\label{vogan_bijection}
\mu(\phi_c,*)\colon\Char(\tSphic)\overset{\sim}\longrightarrow \A(\phi_c),   \quad \chi_c\mapsto \mu(\phi_c,\chi_c).
\end{equation}

\subsection{} \label{sec:program_rgroups}
We come to the statement of our main results
on the interplay between the Langlands correspondence and lowest $K$-types.

Suppose $\phi\colon\WR\to\lgr{G}$ is a Langlands parameter, and $\chi$ is a character of 
$\widetilde{\mathbb{S}}_\phi$. This defines a representation $\pi=\pi(\phi,\chi)$ of
one of our strong real forms. 
We want to find the set $\mathrm{LKT}(\phi, \chi)$ of lowest $K$-types  of $\pi(\phi,\chi)$.
According to the discussion in Section \ref{sec:prog_prelim}, this is a subset of~$\A(\phi_c)$.
We use the bijection \eqref{vogan_bijection},  and specify a subset of the characters of $\tSphic$. 

There is an obvious inclusion $\phi_c(\mathbf{W}_\mathbb{R}) \subset  \phi(\mathbf{W}_\mathbb{R})$,
and therefore
$\mathrm{Cent}_{\ch{G}}(\phi(\mathbf{W}_{\mathbb{R}})) \subset \mathrm{Cent}_{\ch{G}}(\phi_c(\mathbf{W}_{\mathbb{R}}))$.
This induces a group homomorphism $\iota\colon \mathbb{S}_\phi \to \mathbb{S}_{\phi_c}$.
The same argument applied to coverings gives a homomorphism of abelian groups:
\begin{equation} \label{inclusion_morphism}
\widetilde{\mathbb{S}}_\phi \to \widetilde{\mathbb{S}}_{\phi_c}.
\end{equation}
A crucial point is:
\begin{prop}\label{prop:injectivity}
The morphism \eqref{inclusion_morphism} is injective.
\end{prop}

Therefore the dual of  \eqref{inclusion_morphism} provides a
canonical surjection
\begin{equation} \label{restriction_morphism}
\mathrm{Res}\colon \Char(\widetilde{\mathbb{S}}_{\phi_c}) \twoheadrightarrow\Char(\widetilde{\mathbb{S}}_{\phi})
\end{equation}
of character groups.
This leads to the main theorem.

\begin{theo} \label{main_theorem_rgroups}
Suppose $(\phi,\chi)$ is a pair consisting of a Langlands parameter $\phi:\WR\to\lgr G$ and a character $\chi$ of
$\tSphi$. Let $\pi(\phi,\chi)$ be the corresponding representation of one of the  strong real forms of $G$. 
Let $\Omega = \mathrm{Res}^{-1}(\chi)\subset \Char(\tSphic)$ be the fiber of the restriction map \eqref{restriction_morphism}. 

Then $\mathrm{LKT}(\pi(\phi, \chi))$ is the set which corresponds to $\Omega$
under the canonical bijection~\eqref{vogan_bijection}, i.e.
\begin{equation}
  \LKT(\pi(\phi,\chi))=\{ \, \mu(\phi_c,\chi_c)\mid \chi_c\in\Char(\tSphic), \ \Res(\chi_c)=\chi\, \}.
\end{equation}
\end{theo}

\subsection{}
Proposition \ref{prop:injectivity} is motivated by work of Knapp--Stein and Shelstad.
If $G(\mathbb{R})$ is a real form of~$G$,
consider a parabolic subgroup $P(\mathbb{R}) = M(\mathbb{R})N(\mathbb{R})$
with Levi factor $M(\mathbb{R})$,
consider the Langlands decomposition $M(\mathbb{R})= M_0 \,{}A(\mathbb{R})$,
and  fix a square-integrable representation~$\sigma$ of~$M_0$.
Knapp and Stein study the reducibility of the representations
\mbox{$\pi_{\sigma, \nu} = \mathrm{Ind}_{P(\mathbb{R})}^{G(\mathbb{R})}(\sigma \otimes e^{i\nu})$},
when $\nu$ is a linear form on the Lie algebra~$\mathfrak{a}$ of~$A(\R)$.
If the irreducible constituents of $\pi_{\sigma, \nu}$ are in an $L$-packet $\Pi(\phi)$,
then those of $\pi_{\sigma, 0}$ are in the $L$-packet~$\Pi(\phi_c)$.
The reducibility of $\pi_{\sigma, \nu}$ is governed by the Knapp--Stein group~$R_{\sigma, \nu}$.
Proposition~\ref{prop:injectivity} is an $L$-group analogue of Knapp and Stein's observation
that $R_{\sigma, \nu}$ always embeds naturally in $R_{\sigma, 0}$.
Now, Langlands and Shelstad showed how to realize the $R$-group $R_{\sigma, \nu}$
as a quotient $\mathbb{S}_{\phi}/\mathbb{S}^1_{\phi}$,
where $\mathbb{S}^1_{\phi}$ is the component group
for a discrete series parameter of a Levi subgroup of $G$
(see \cite{Langlands_notes_KZ}, \cite{Shelstad}).
To understand the injectivity of $\iota\colon \mathbb{S}_{\phi} \to \mathbb{S}_{\phi_c}$
proved in Proposition~\ref{prop:injectivity}, it may be helpful to say
that the groups $\mathbb{S}^1_{\phi}$ and $\mathbb{S}^1_{\phi_c}$
can be seen to be identical,
and that the morphisms just discussed fit into a commutative diagram
\[
\begin{CD}
0 @>>> \mathbb{S}_{\phi}^1 @>>> \mathbb{S}_\phi @>>> R_{\sigma, \nu}
@>>> 0 \\
@. @| @V{\iota}VV @V{}VV \\
0 @>>> \mathbb{S}^1_{\phi_{\mathrm{c}}} @>>>
\mathbb{S}_{\phi_{\mathrm{c}}}
@>>> R_{\sigma, 0} @>>> 0.
\end{CD}
\]
It is actually possible to combine arguments of Knapp--Stein and Shelstad
to give a proof that~$\iota$ is injective.
But we shall follow a different path and avoid the use of harmonic analysis and intertwining operators,
in favor of more elementary structure theory on the dual side.
This will also make it possible to incorporate coverings. See Section~\ref{sec:waldspurger}.

Our proof of Theorem~\ref{main_theorem_rgroups} is based on ideas on structure theory and lowest $K$-types
implemented in the \texttt{atlas} software,
and an algorithm for the determination of lowest $K$-types due to David Vogan. See Section~\ref{sec:r_groups}. This uses  \emph{Cayley transforms} and \emph{cross actions} in the \texttt{KGB} space, which is at the heart of the \texttt{atlas} parametrization of representations.

\subsection{}\label{sec:leftover_beginning}
Theorem \ref{main_theorem_rgroups} determines the lowest $K$-types
attached to any pair $(\phi, \chi)$, in terms of the bijection (1)--(3) in \S \ref{sec:various_bijections}.
However, if the parameter $\phi$ we began with is itself trivial on $\mathbb{R}^\times_+$,
then Theorem~\ref{main_theorem_rgroups} is tautological.
Therefore it would be good to know whether the bijection (1)--(3) 
is easy to understand in terms of the Langlands correspondence.

One problem is that the group~$K$ is in general disconnected,  albeit not very badly; therefore it does not have an $L$-group, at least not in the sense of \cite{Borel}. One could turn instead to recent work of Kaletha~\cite{Kaletha}, which seeks to extend the Langlands parametrization to mildly disconnected groups such as~$K$. But even in the simpler case where~$K$ is connected, and has an $L$-group, it does not seem that there is an easy and general description of the bijection of~\S \ref{sec:various_bijections} in terms of the $L$-group of~$K$.

\subsection{}
Our results are of course entirely dependent on the fine details of the local Langlands parametrization.
The parametrization of $L$-packets by $L$-homomorphisms is widely known:
the classical description is \cite{Borel}, see also \cite{Contragredient}.

For the refined version, which includes a parametrization of each $L$-packet by characters of the component group,
the situation is less satisfying.
There are (at least) two versions: Shelstad's classical work \cite{Shelstad, Shelstad08},
and the version of Adams--Barbasch--Vogan \cite{ABV}.
It is not obvious how to match these two parametrizations
(see however the recent paper \cite{Arancibia_Mezo} for a comparison).
Therefore we need a choice; for reasons which should become obvious later, we use~\cite{ABV}.

We have taken this opportunity to give a slightly simplified
exposition of the refined Langlands correspondence (in the Adams--Barbasch--Vogan version).
The simplifications come from ideas crucial to the \texttt{atlas} software package~\cite{Algorithms}:
we will use notions of \cite{Algorithms} to express some ideas from \cite{ABV, HermitianFormsSMF} in a hopefully easier way.
The simplified exposition may be of independent interest.
It is the contents of \S \ref{sec:dictionary_KGB_charcompgroup} and \S \ref{sec:local_langlands}.
We treat the general case here; if one specializes to the discrete series case, then several ingredients of our exposition are in
\cite{DiscreteSeriesSigns} and \cite{Contragredient}.

\subsection{}
The paper divides naturally into two parts.
The first part, in \S \ref{sec:structure_theory}--\ref{sec:local_langlands},
leads up to our description of the local Langlands correspondence.
Section \ref{sec:structure_theory} introduces the vocabulary (strong real forms, \texttt{KGB} elements) that is needed on the $G$-side,
and  Section \ref{sec:structure_LG} collects basic structure theory on the $\lgr{G}$-side.
The exposition of the refined Langlands parametrization is in \S \ref{sec:dictionary_KGB_charcompgroup} and \S \ref{sec:local_langlands}.

The second part of the paper turns to lowest $K$-types.
Section \ref{sec:ric_temp} fills in the details concerning the tempiric representations discussed above, and
their Langlands parameters.
Sections \ref{sec:waldspurger} and \ref{sec:r_groups} are concerned with the forgetful map $\phi \mapsto \phi_c$; there 
we prove our main results, Proposition~\ref{prop:injectivity} and Theorem~\ref{main_theorem_rgroups}. 

The \hyperref[sec:appendix]{Appendix} discusses the relationship between the refined Langlands correspondence  of Section~\ref{sec:local_langlands} and Whittaker data. 

\subsection*{Acknowledgements}
We thank Tasho Kaletha, David Renard and David Vogan for helpful discussions, and Jean-Loup Waldspurger for suggesting a simple proof of Lemma~\ref{lemm:walds}. 

This research was started within the online research community on Representation theory and Noncommutative geometry
sponsored by the American Institute of Mathematics;
we are grateful to Pierre Clare, Nigel Higson and Birgit Speh for putting it together during the early stages of the COVID-19 pandemic.


\section{Structure theory for $G$: strong real forms and \texttt{KGB} elements}
\label{sec:structure_theory}

This section is a review of material from~\cite{Algorithms}.
Throughout the paper, we fix
\begin{itemize}
\item[$\bullet$] a connected complex reductive group $G$,
\item[$\bullet$] a pinning $\mathcal{P}=(B, H, \{X_\alpha\})$ of $G$,
\item[$\bullet$] and an inner class of real forms of $G$.
\end{itemize}

The pinning $\mathcal{P}$ consists of a Borel subgroup $B$, a Cartan subgroup $H \subset B$,
and a set  $\{X_\alpha\}$ of root vectors for the simple roots of $H$ in $B$.
For the notion of inner class of real forms, see \S \ref{sec:notation_real_forms}.

\subsection{Ordinary real forms}
\label{sec:notation_real_forms}
\label{sec:def_realform} 
\label{extended_group}
 \label{sec:comments_realform}

A real form of $G$ is the fixed points $G^\sigma$
of an \emph{antiholomorphic} involutive automorphism $\sigma$ of $G$.
It is well known that it is equivalent to work instead in terms of the Cartan involution,
which is a \emph{holomorphic} involution.

Therefore an \emph{involution} of~$G$ will mean, in this paper, a
holomorphic automorphism $\theta$ of~$G$ satisfying $\theta^2 = 1$.
We define a \emph{real form} of~$G$ to be an involution of~$G$, and
say two real forms are \emph{equivalent} if they are $G$-conjugate.
Given a real form $\theta$, there is an antiholomorphic involution $\sigma$, commuting with $\theta$, 
such that $G^\theta$ is the complexification of a maximal compact subgroup of $G(\R)=G^\sigma$. This induces a bijection between equivalence classes of holomorphic and anti-holomorphic involutions.
Given~$\theta$ and a choice of~$\sigma$, we will write $G(\R,\theta)$ for $G^\sigma$.

Let $\mathrm{Aut}(G)$ denote the group of \emph{holomorphic} automorphisms of $G$,
and $\mathrm{Int}(G)$ denote the subgroup of inner automorphisms.
We have the usual exact sequence
\begin{equation} \label{exact_aut}
1 \to \mathrm{Int}(G) \to \mathrm{Aut}(G) \to \mathrm{Out}(G)\to 1
\end{equation}
where $\mathrm{Out}(G)$ is the quotient $\mathrm{Aut}(G)/\mathrm{Int}(G)$.

Two involutions $\theta, \theta' \in \mathrm{Aut}(G)$ are said to be \emph{inner to each other}
if they have the same image in~$\mathrm{Out}(G)$.
Thus an \emph{inner class of real forms} is determined by an element $\gamma\in\mathrm{Out}(G)$ 
of order $1$ or $2$. The 
pinning $\mathcal{P}$ of $G$ determines  a splitting of the exact sequence~\eqref{exact_aut},
taking $\gamma$ to a $\mathcal{P}$-distinguished involution of $G$. See \cite[\S 2.1]{Algorithms}.
We still denote by $\gamma$ the corresponding element of~$\mathrm{Aut}(G)$.


Attached to the inner class $\gamma$ and the pinning $\mathcal{P}$ is an \emph{extended group} $G^\Gamma$,
containing~$G$ as a subgroup of order two.
Let $\Gamma = \{ 1, \varsigma\}$ be the Galois group of $\mathbb{R}$.
We define $G^\Gamma$ to be  $G \rtimes \Gamma$,
where $\varsigma \in \Gamma$ acts by the distinguished involution~$\gamma$.
Write $\xi_{\gamma}$ for the element $(1, \varsigma)$ of~$G^\Gamma$,
thus~$G^\Gamma = \langle G, \xi_{\gamma}\rangle$, with $\xi_\gamma^2 = 1$.


\subsection{Strong real forms}
\label{sec:strong}
\label{sec:pure}
Let $Z(G)$ denote the center of $G$.



\subsubsection{} A \emph{strong real form} of $G$ (in the inner class attached to
$\gamma$) is an element $\xi \in G^\Gamma \setminus G$ satisfying
$\xi^2 \in Z(G)$.  By analogy with the Cartan involution, we also
refer to $\xi$ as a \emph{strong involution}.  We say two strong real
forms are equivalent if they are conjugate by $G$.  We denote the set
of strong real forms by $\mathcal{I}(G, \gamma)$, or simply
$\mathcal{I}$.  For $\xi \in \mathcal{I}$, we let $\theta_\xi$ denote
the involution $\mathrm{int}(\xi)$ of $G$, and write~$K_\xi$ for the subgroup
$G^{\theta_\xi}$ of~$G$.


If $\xi$ is a strong real form, 
we  denote by $\mathcal{C}_{\xi}$ its equivalence (conjugacy) class.


The map $\xi\mapsto \theta_\xi$ is a surjection from strong
real forms to real forms (in the given inner class), and factors to a
surjective map \
$\{\text{strong real forms}\}/\!\!\sim \ \twoheadrightarrow \ \{\text{real
  forms}\}/\!\!\sim$.


We define the \emph{central cocharacter} of a strong real form $\xi$ to be 
the element $\xi^2$ of $Z(G)$; this is well defined on equivalence classes. 
Let $z_\ast=\exp(2i\pi\ch{\rho})\in Z(G)$,
where~$\ch{\rho}$ is one-half the sum of positive coroots. This element is 
independent of the choice of Borel subgroup~$B$, and is fixed by all automorphisms of~$G$.
We say a strong real form is \emph{pure} if its central cocharacter is~$z_\ast$.
This includes the quasisplit strong real forms. (When we refer to a property of strong real forms,
such as being quasisplit, we pull it back from real forms, provided it is constant on the fibers.)


\subsubsection{Representations} 
\label{sec:def_representations}
When $\xi$ is a strong real form of $G$, we can consider Harish-Chandra's notion of $(\mathfrak{g}, K_\xi)$-module.
We define a \emph{representation of a strong real form of $G$} to be a pair $(\xi, X)$
where $\xi$ is a strong real form of $G$ and $X$ is a $(\mathfrak{g}, K_\xi)$-module.
We call two pairs $(\xi,X), (\xi', X')$ \emph{equivalent} when there exists $g \in G$ such that
$\xi'=g\xi g^{-1}$ and $X'=X^{g}$, where $X^{g}$ is the $(\g,K_{\xi'})$-module defined by transport of structure using $\int(g)$.
We write $[\xi, X]$ for the equivalence class of~$(\xi, X)$. We say $[\xi, X]$ is \emph{irreducible} if~$X$ is irreducible; this is independent of the choice of~$(\xi,X)$.

It is important to keep in mind that this definition keeps track of the strong real form, not just the real form.
For a useful example \cite[Example 3.3]{DiscreteSeriesSigns}, set $G=\mathrm{SL}(2,\mathbb{C})$ and $\xi=\mathrm{diag}(i,-i)$.
Let $\pi$ be the $(\mathfrak{g},K_\xi)$-module of a holomorphic discrete series for the real form of $G$ corresponding to $\theta_\xi$,
and let $\overline{\pi}$ be the contragredient.
Then there is an inner automorphism of $G$ which takes $\pi$ to $\overline{\pi}$, and also takes $\xi$ to $-\xi$.
This is reflected in the fact that $[\xi, \pi]=[-\xi, \overline{\pi}]$, but $[\xi, \pi] \neq [\xi, \overline{\pi}]$.
Thus we can view the two discrete series representations of $\SL(2,\R)$ in 
the usual way as  $[\xi,\pi]$ and $[\xi,\overline\pi]$ (fixing $\xi$ and varying the representation), 
or alternatively as $[\xi,\pi]$ and $[-\xi,\pi]$ (fixing the representation and varying the strong real form).

\subsubsection{An example of conjugating to do representation theory on a forever-fixed Cartan} \label{example_conjugating}
Let us still consider $G=\mathrm{SL}(2, \mathbb{C})$,  the strong real form $\xi = \mathrm{diag}(i, -i)$, and the involution $\theta=\theta_\xi$.
For an antiholomorphic automorphism $\sigma$ giving $G(\mathbb{R}, \theta)$ as in \S \ref{sec:notation_real_forms},
we may take $g \mapsto J \phantom{a}^t\overline{g}^{-1}J$ with $J=\mathrm{diag}(1, -1)$,
and then $G(\mathbb{R}, \theta)=G^\sigma$ is $\mathrm{SU}(1,1)$.

Let $H$ be the diagonal subgroup of $G$.
Then $H(\mathbb{R}, \theta)$ is a compact Cartan subgroup of $G(\mathbb{R}, \theta)$,
and its importance to the classical theory is that its regular characters can be used
to parametrize the discrete series of $\mathrm{SU}(1,1)$.

Now, suppose we turn to  the principal series of $\mathrm{SU}(1,1)$.
The traditional way of building it uses characters
of another Cartan subgroup $H_1(\mathbb{R}, \theta)$ of $\mathrm{SU}(1,1)$.
As the notation indicates, this real torus comes
from another Cartan subgroup $H_1 \neq H$ of the complex group~$G$.

An important idea for what follows is that it is possible to do representation theory
using only real forms of the fixed Cartan subgoup $H$.
The reason is, of course, that we can conjugate $H_1$ back to $H$
using an element of the complex group $G$.
If we do this, then we have to change the real group under discussion:
it is easy to find an element $u$ of $\mathrm{SL}(2, \mathbb{C})$ such that
$\mathrm{int}(u)$ takes the real group $\mathrm{SU}(1,1)$ to $\mathrm{SL}(2,\mathbb{R})$
and the Cartan subgroup $H_1(\mathbb{R}, \theta)$ to the diagonal subgroup of~$\mathrm{SL}(2,\mathbb{R})$.
Conjugation by $u$ will also take $\xi$ to another strong real form~$\xi'$, with $\mathcal{C}_\xi=\mathcal{C}_{\xi'}$;
and the discussion of the principal series will then switch
 from $(\mathfrak{g}, K_\xi)$-modules to $(\mathfrak{g}, K_{\xi'})$-modules.

Thus we can always  fix the Cartan subgroup $H$
and use a real form of $H$ to discuss representations;
but depending on the representation, the appropriate real torus will sit
either in  $\mathrm{SL}(2, \mathbb{R})$ or in $\mathrm{SU}(1,1)$.
Passing back and forth requires conjugating some of the classical data by the complex group $G$.
The conjugation will affect the strong real forms and representations of strong real forms under discussion,
but not their equivalence classes.


\subsection{The \texttt{KGB} space} 
\label{sec:def_KGB_space}
\label{sec:KGB}

The space of $K$-orbits on $G/B$ plays an important role in this version of the Langlands classification.
We parametrize this space, simultaneously for all strong real forms in the inner class, using the 
{\tt{KGB}} space. 

As in  \S \ref{example_conjugating} we work in terms of our fixed Cartan subgroup $H$.
Define
\[
\widetilde{\mathcal{X}}=\left\{\, \xi \in \mathrm{Norm}_{G^\Gamma\backslash G}(H) \ : \ \xi^2 \in Z(G)\,\right\}.
\]
Thus $\Xt$ is a set of strong real forms and every strong real form is conjugate to one of these.

The group $H$ acts by conjugation on $\widetilde{\mathcal{X}}$, and we set
\begin{equation} \label{kgb_space}
\mathcal{X} = \widetilde{\mathcal{X}}/H.
\end{equation}

Suppose $x\in \X$ and let $\Xt_x$ be the fiber over $x$ in $\Xt$. 
The groups $K_\xi$, $\xi\in\Xt_x$, are 
all isomorphic, by an isomorphism which is canonical up to an inner automorphism of $K_\xi$.
Therefore we set $K_x=\lim_{\xi}(K_\xi)$ (direct limit), 
and define a $(\g,K_x)$ module
to be the corresponding direct limit of $(\g,K_\xi)$-modules.
This makes it possible to talk about $(\g,K_x)$ modules without having to say  ``$(\g,K_\xi)$ modules for a choice of $\xi$ lying over $x$''.
If $X$ is such a $(\g,K_x)$-module, it gives rise to a well-defined equivalence class of representations of strong real forms, which we denote by $[x,X]$; and it is safe to abuse notation slightly and use the terminology  \emph{representation of a strong real form} for the pair $(x,X)$.

Along these lines, we abuse notation slightly when talking about $G$-conjugacy of elements of~$\Xt$ and~$\X$.
For instance, if $x\in \X$, then it makes sense to define
\begin{equation}
\label{e:X[x ]}
\X[x]=\{x\in \X\mid x\text{ is $G$-conjugate to }x  \}.
\end{equation}
Similarly, if 
$\xi$ is a strong real form, then it makes sense to talk about the set
\begin{equation}
\label{e:X[xi]}
\X[\xi]=\{x\in \X\mid x\text{ is $G$-conjugate to the image of }\xi\}.
\end{equation}

As in Section \ref{sec:various_bijections}, fix a set
$\{\xi_i\}_{i\in I}$ of representatives of the equivalence classes $\I/G$ of
strong real forms. If $G$ is semisimple, or more generally if contains 
no $\gamma$-fixed torus, then $I$ is finite.
There is a canonical bijection
\[
\mathcal{X} \ \longleftrightarrow \ \coprod_{i\in I}K_{\xi_i} \backslash G / B.
\]
In this bijection $K_\xi\backslash G/B$ corresponds to the set $\X[\xi]$ of \eqref{e:X[xi]}
For this reason we call $\mathcal{X}$ the {\texttt{KGB} space}.
See \cite[Corollary 9.9]{Algorithms} for more details.

Suppose $x\in \X$. Choose $\xi\in \Xt$ mapping to $x\in \X$. The
restriction of $\theta_\xi$ to $H$ is independent of the choice of
$\xi$, and is denoted $\theta_{x,H}$. Thus each element of $\X$
defines a real form of the Cartan subgroup $H$. Set
\[
 \mathcal{I}_{W} = \left\{\, \theta_{x, H} \ : \ x \in \mathcal{X} \,\right\}.
 \]
The notation $\mathcal{I}_{W}$ comes from \cite[(9.11)]{Algorithms},
which views the involutions $\theta_{x,H}$ as twisted involutions in the Weyl group.
We will come back to this in \S \ref{sec:twisted_involutions} below.

For $\tau \in \mathcal{I}_W$, define the \emph{fiber} $\mathcal{X}_{\tau}$
to be the set of elements $x \in \mathcal{X}$ such that $\theta_{x, H}=\tau$.
This gives us a partition
\begin{equation}\label{partition_tau}
\mathcal{X} = \coprod \limits_{\tau \in \mathcal{I}_W} \mathcal{X}_\tau.
\end{equation}
If $\tau\in \I_W$ and $x\in \X$ let 
\begin{equation}
\label{e:X[x]}
\begin{aligned}
\X_\tau[x]&=\X_\tau\cap \X[x]\\&=\{x'\in \X\mid \theta_{x',H}=\theta_{x,H}\text{ and } x'\text{ is $G$-conjugate to }x\},
\end{aligned}
\end{equation}
and define $\X_\tau[\xi]$ similarly for $\xi \in \Xt$. Then
$$
\X_\tau=\coprod_{i\in I}\X_\tau[\xi_i].
$$
Some of the sets $\X_\tau[\xi_i]$ may be empty.

If $x\in \X$, define an element $z(x)$ of $Z(G)$ by $z(x)=\xi^2$ where $\xi\in\Xt$ is any element mapping to~$x$;
and for $z\in Z(G)$, set $\X(z)=\{x\in \X\mid z(x)=z\}$. In this way we
 get a different partition
$$
\X=\coprod_{z\in Z(G)}\X(z).
$$
This respects the previous partition: set $\X_\tau(z)=\X_\tau\cap \X(z)$, then we have
$$
\X_\tau=\coprod_{z\in Z(G)}\X_\tau(z).
$$

\subsection{Cross action and Cayley transforms for  \texttt{KGB}}
\label{s:crosscayleykgb}
\label{sec:action_weyl_kgb}
\label{sec:def_cross_action}
\label{sec:transitivity_action_srf}

Let $W=\mathrm{Norm}_G(H)/H$ be the Weyl group of $(G,H)$.
Then~$W$ has a natural action on $\mathcal{X}$, as follows.
For $w \in W$ and $x \in \mathcal{X}$, choose representatives
$n \in \mathrm{Norm}_G(H)$ and $\xi \in \widetilde{\mathcal{X}}$
for $w$ and $x$, respectively.
Then $n\xi n^{-1} \in \widetilde{\mathcal{X}}$,
and we define $w \times x$ to be the image of $n\xi n^{-1}$ in $\mathcal{X}$.
We refer to this as the \emph{cross action} of $W$ on $\mathcal{X}$.


Let $\Delta=\Delta(G,H)$ be the set of roots of $H$ in $G$.
Suppose $\tau\in \I_W$ is an involution of $H$. This preserves $\Delta$, 
and we set
\begin{equation}
\begin{aligned}
\Delta_{i,\tau}&=\left\{ \alpha \in \Delta \ : \ \tau(\alpha)=\alpha\right\}\quad \text{the $\tau$-imaginary roots;}\\
\Delta_{r,\tau}&=\left\{ \alpha \in \Delta \ : \ \tau(\alpha)=-\alpha\right\}\quad \text{the $\tau$-real roots.}
\end{aligned}
\end{equation}
See \cite[(12.1)]{Algorithms}.
These are root systems, and  we let $W_{i, \tau},W_{r, \tau}$  denote the Weyl groups of $\Delta_{i,\tau},\Delta_{r,\tau}$ respectively.
These act on $\mathcal{X}_\tau$, and on each piece 
$\mathcal{X}_{\tau}[x]$ of the decomposition \eqref{e:X[x]}. The action of $W_{i, \tau}$ on each piece~$\mathcal{X}_{\tau}[x]$ is transitive.

We say $\alpha$ is $\tau$-complex if it is neither $\tau$-real nor $\tau$-imaginary.
In this case  the cross action of the reflection $s_\alpha \in W$ defines a bijection
\begin{equation} 
(s_\alpha \times~\,{} ) \colon \mathcal{X}_\tau \to \mathcal{X}_{s_\alpha \tau s_\alpha}.
\end{equation}

Suppose $\alpha\in \Delta_{i,\tau}$. Choose an $\alpha$-root vector $X_\alpha$, and a representative $g_\alpha\in\Norm_G(H)$ of~$s_\alpha$.
Then $\mathrm{Ad}(g_\alpha)(X_\alpha)=\pm X_\alpha$. We say $\alpha$ is $\tau$-compact if $\mathrm{Ad}(g_\alpha)$ fixes $X_\alpha$, and $\tau$-noncompact otherwise.
If $x\in\X_\tau$, we also use the term $x$-imaginary for $\tau$-imaginary. We say $\alpha$ is $x$-compact if $\theta_\xi$ fixes $X_\alpha$ for some (equivalently any) $\xi \in \widetilde{\X}_x$, and $x$-noncompact otherwise.

If $x\in\X$, then the stabilizer of $x$ in~$W$ is naturally isomorphic to 
$W(K_\xi,H)=\mathrm{Norm}_{K_\xi}(H)/H$, where $\xi\in\Xt$ is any representative of $x$.
This group is isomorphic to $W(G(\mathbb{R}, \theta_\xi), H(\mathbb{R}, \theta_\xi))$,
and we refer to it as the real Weyl group. See \cite[\S 12]{Algorithms}.

Given an involution $\tau$ of $H$, define
\begin{equation}
  \label{e:Xtau}
  \mathcal{X}_\tau[\alpha] = \{ x \in \mathcal{X}_\tau \ : \ \alpha \text{ is $x$-noncompact}\}.
\end{equation}
Suppose $x\in \X_\tau[\alpha]$.
Choose a representative $g_\alpha\in \Norm_G(H)$ of $s_\alpha$, and  a strong involution  $\xi \in \widetilde{\mathcal{X}}$ representing~$x$.
Then $g_\alpha\xi\in \widetilde{\mathcal X}$.
We refer to the image of this element in $\X$  as the \emph{Cayley transform}~$c^\alpha(x)$; as the notation indicates it is independent of the choices. 
Also $c^\alpha(x)$ is $G$-conjugate to $x$.

The map
\begin{equation}
c^{\alpha}\colon \mathcal{X}_\tau[\alpha] \to \mathcal{X}_{s_\alpha \tau}
\end{equation}
is surjective, and at most two-to-one. We write $c_{\alpha}$ for the inverse Cayley transform:
if $c^\alpha$ is injective, then $c_\alpha(x)$ is a single element $x'$ satisfying $c^\alpha(x')=x$; on the other hand, if $c^\alpha$ is two-to-one then $c_\alpha(x)=\{x',x''\}$ where $c^\alpha(x')=c^\alpha(x'')=x$.
In the latter case, $\alpha$ is $\theta_{x'}$-imaginary, and $s_\alpha\times x'=x''$. 
For all this, see \cite[\S 14]{Algorithms}.



\section{Basic structure theory for $^LG$}
\label{sec:structure_LG}

This section collects standard material on Langlands parameters. We  follow \cite{ABV, Algorithms, Contragredient}.
Recall we have fixed $G$,
a pinning $(B,H,\{X_\alpha\})$, and an inner class of real forms defined by an involution $\gamma\in\Out(G)$.



\subsection{The $L$-group} 
\label{sec:Lgroup}
\label{dual_involution}


Let $X^\ast(H), X_\ast(H)$ be the character and cocharacter lattices of $H$, respectively.
Consider $\ch H=X^\ast(H)\otimes \C^\times$, the complex torus dual to $H$.
Then $X^\ast(H)=X_\ast(\ch H)$ and $X_\ast(H)=X^\ast(\ch H)$ (these are canonical identifications).
Also the Lie algebra $\ch{\mathfrak{h}}$ is canonically identified with $\mathfrak{h}^\ast$ 
(the vector space dual of the Lie algebra of $H$).

We make frequent use of the elements $\rho\in X^*(H)= X_*(\ch H)$ and $\ch\rho\in X_*(H)=X^*(\ch H)$,
where $\rho$ (respectively $\ch\rho$) is one-half the sum of the positive roots (resp. coroots) of $B$.

We  often  consider involutive automorphisms of $H$ and $\ch{H}$
that are dual to each other.
If~$\tau$ is a holomorphic involutive automorphism of $H$, with differential $d\tau$, consider
the transpose $d\tau^\top$ 
as an endomorphism of $\ch\h$. Let $\ch\tau$ be the holomorphic involution of 
$\ch H$ with differential $d(\ch\tau)=-d\tau^\top$.
Notice the minus sign: if $\tau$ is the identity on $H$,
then~$\ch{\tau}$ is inversion on $\ch{H}$.


Let $\Pi, \ch{\,\Pi}$ be the sets of simple roots and simple coroots of $G$ defined by $B$.
The \emph{based root datum} defined by~$(B,H)$ is the quadruple
$\mathscr{D}=(X^\ast(H), \Pi, X_\ast(H), \ch{\,\Pi})$.
The \emph{dual} of $\mathscr{D}$ is the quadruple
$\ch{\mathscr{D}}=(X_\ast(H), \ch{\,\Pi}, X^\ast(H), {\Pi})$.
We  use it to construct the connected complex dual group~$\ch{G} \supset \ch{H}$
(see \cite[\S 2]{Algorithms} and \cite[\S 2 and \S 6]{Contragredient}),
and we equip it with a pinning $\ch{\mathcal{P}}=(\ch{H}, \ch{B}, \{X_{\ch{\alpha}}\})$
such that  the based root datum defined by~$(\ch{B}, \ch{H})$ is $\ch{\mathscr{D}}$.
The Weyl group of~$(\ch{G}, \ch{H})$ canonically identifies with~$W=W(G,H)$.


We turn to the group $\lgr{G}$.
This is a split extension of $\ch{G}$ by $\mathbb{Z}/2\mathbb{Z}$, 
depending on the inner class $\gamma$,
so describing it amounts to describing an automorphism of $\ch{G}$.

The automorphism $\gamma\in\Out(G)$ 
determines an automorphism $\overline{\gamma}$ of the based root datum $\mathscr{D}$
(for the notions of isomorphism of based root data and transpose isomorphism, see  \cite[\S 2]{Algorithms}).
On the dual side, define an automorphism of $\ch{\mathscr{D}}$ as
$\vartheta_0=-w_0 \overline{\gamma}^\top$, where $w_0$  is the long element of the Weyl group
(again note the minus sign).
Let $\ch\gamma$ be the unique $\ch{\P}$-distinguished automorphism of $\ch G$ 
which induces $\vartheta_0$ on $\mathscr{D}^\vee$.
The $L$-group of $G$ is  $\lgr{G} = \langle \ch{G}, \ch{\delta}\rangle$,
where $(\ch{\delta})^2=1$ and $\ch\delta$ acts by~$\ch{\gamma}$ on~$\ch{G}$.


\subsection{Langlands parameters} 
\label{sec:l_params}
\label{sec:l_hom}

Let $\mathbf{W}_{\mathbb{R}}$ be the Weil group of $\mathbb{R}$.
By definition  $\WR=\langle \mathbb{C}^\times, j\rangle$,
with relations $j^2 = -1$ and $jzj^{-1}=\overline{z}$ for $z \in \mathbb{C}^\times$,
where the bar denotes complex conjugation.
A map $\phi\colon \mathbf{W}_{\mathbb{R}} \to \lgr{G}$ is said to be an \emph{$L$-homomorphism}
if it is a continuous group homomorphism,   $\phi(\mathbb{C}^\times)$ consists of semisimple elements,
and $\phi(j)  \in \lgr{G}\setminus\!\!\ch{G}.$
The connected  group $\ch{G}$ acts on the set of $L$-homomorphisms by conjugation on the range,
and we define an \emph{$L$-parameter} to be a~$\ch{G}$-conjugacy class of $L$-homomorphisms.


Suppose $\phi\colon \mathbf{W}_{\mathbf{R}} \to\lgr{G}$ is an $L$-homomorphism.
We set
\begin{equation}
\label{e:L}
\ch L_\phi=\Cent_{\ch G}(\phi(\C^\times)).
\end{equation}
Since $\phi(\mathbb{C}^\times)$ is connected, abelian, and consists of semisimple elements,
this is a connected reductive group.
We say $\phi$ is \emph{aligned with $\ch H$} if $\phi(\C^\times)$ is contained in~$\ch{H}$ and $\phi(j)$ normalizes~$\ch{H}$.

\begin{lemma}\label{l:stdLhom}
Every $L$-homomorphism is $\ch G$-conjugate to one which is aligned with~$\ch H$.

Assume $\phi$ is an $L$-homomorphism aligned with~$\ch{H}$. Then there exists $\lambda\in X_*(\ch H)\otimes\C$ and $y \in \lgr{G}\setminus \ch{G}$,
normalizing $\ch H$, 
satisfying:
\begin{subequations}
\renewcommand{\theequation}{\theparentequation)(\alph{equation}}
\label{e:phistd}
\begin{equation}
\begin{aligned}
\lambda-y\lambda&\in\X_*(\ch H),\\
y^2&=e^{2\pi i\lambda},
\end{aligned}
\end{equation}
such that $\phi$ is defined by:
\begin{equation}
\label{e:phistd2}
\begin{aligned}
\phi(z)&=z^\lambda\overline z^{\mathrm{Ad}(y)\lambda},\\
\phi(j)&=e^{-\pi i\lambda}y.
\end{aligned}
\end{equation}
\end{subequations}
If a pair $(\lambda,y)$ satisfies  \textup{(a)}, then the map $\phi$ defined by \textup{(b)} is an $L$-homomorphism.
\end{lemma}

Given $(\lambda,y)$ satisfying~\eqref{e:phistd}(a), we denote by $\phi(\lambda,y)$ the $L$-homomorphism defined by~\eqref{e:phistd2}.

\begin{proof}
We know that $\phi(\C^\times)$ is contained in the identity component of the center of $\ch L_\phi$, which is a torus.
So, after conjugating by $\ch G$ we may assume $\phi(\C^\times)\subset \ch H$, and then $\ch H\subset \ch L_\phi$.
Now the involutive automorphism 
$\ch\tau=\mathrm{int}(\phi(j))$ normalizes $\ch L_\phi$, and therefore normalizes a Cartan subgroup of $\ch L_\phi$,
which is also a Cartan subgroup of $\ch G$. 
So, after conjugating by $\ch L_\phi$ we may assume $\ch\tau$ normalizes $\ch H$.
The remaining assertions are straightforward. See \cite[Section 6]{Contragredient}.
\end{proof}

So assume $\phi$ is aligned with $\ch H$. Set
\begin{subequations}
\renewcommand{\theequation}{\theparentequation)(\alph{equation}}
\label{e:phi}
\begin{equation}
\label{e:phia}
\ch\tau=\int(\phi(j));
\end{equation}
this is an involution of~$\ch{L}_\phi$ which preserves $\ch{H}$. 
Write
\begin{equation}
\label{e:phib}
 \phi(z) = z^{\lambda} \overline{z}^{\ch\tau(\lambda)}  \quad (z \in \mathbb{C}^\times)
\end{equation}
\end{subequations}
where $\lambda \in\,^\vee\mathfrak{h}$.
The $W(G,H)$-orbit of $\lambda$  (considered as an element of $\ch\h\simeq\h^\ast$) depends only on the $\ch{G}$-conjugacy class of $\phi$.
We call it the \emph{infinitesimal character of $\phi$}.

\begin{exem}
\label{ex:fundamental}
Let $y=e^{\pi i\ch\rho}$.
The pair $(\rho,y)$ satisfies \eqref{e:phistd}(a). If~$\phi=\phi(\rho, y)$, then we shall see that the corresponding $L$-packet $\Pi(\phi)$ is 
the $L$-packet of fundamental series of the quasisplit form of $G$, with infinitesimal character $\rho$. 
See Example~\ref{ex:fundamental2} and the~\href{sec:appendix}{Appendix}.
\end{exem}

\subsection{Cross action and Cayley transforms for $L$-homomorphisms}
\label{s:crosscayleyL}

As discussed in the Introduction, cross actions and Cayley transforms play a central role.
We already discussed these in the context of the {\tt KGB} space. Here is the
corresponding discussion for $L$-homomorphisms.

Defining Cayley transforms and the cross action for $L$-homomorphisms is already nearly covered by the corresponding definitions for {\tt KGB}:
the action on $\phi(\lambda,y)$ is via the action on $y$. The only minor issue is that for $y$ to be a {\tt KGB}
element it has to satisfy $y^2\in Z(\ch G)$. Since $y^2=\exp(2\pi i\lambda)$ this is the case when $\phi$ has 
integral infinitesimal character, but does not hold in general.
Nevertheless, the definitions of Section \ref{s:crosscayleykgb} carry over with minor changes.

Suppose $\phi=\phi(\lambda,y)$ is an $L$-homomorphism aligned with~$\ch{H}$ (Lemma \ref{l:stdLhom}).
Let $\Psiint=\Psiint(\lambda)$ be the set of integral roots: $\{\alpha\mid \langle\lambda, \ch\alpha\rangle\in\Z\}$.

Suppose $w\in W(\Psiint)$  and  $n\in \Norm_{\ch G}(\ch H)$ represents
$w$.
Then
just as in Section \ref{s:crosscayleykgb}, it is easy to see
that $nyn\inv$ satisfies $(nyn\inv)^2=y^2$,
so 
$\phi(\lambda,nyn\inv)$ is well defined, and its~$\ch G$-conjugacy class is independent
of the choice of $n$. We define $w\times\phi(\lambda,y)=\phi(\lambda, nyn\inv)$.

Now assume $\alpha$ is a simple root of $\Psiint$ which is $\theta_{y, \ch{H}}$-imaginary and noncompact. The latter condition
is that $\mathrm{Ad}(y)$ does not fix $X_\alpha$. 
The definition of $c^\alpha$ in Section \ref{s:crosscayleykgb} carries over immediately to this setting; the only change
is that $c^\alpha(y)^2$ is no longer central, but satisfies \mbox{$c^\alpha(y)^2=y^2=\exp(2\pi i\lambda)$}.
We define $c^\alpha(\phi(\lambda, y))=\phi(\lambda,c^\alpha(y))$, and refer to this as the Cayley transform of $\phi$.

\subsection{$E$-groups for tori, and characters of two-fold covers}
\label{sec:duality_tori}
\label{sec:description_duality}


Suppose  $\phi\colon \mathbf{W}_{\mathbf{R}} \to\lgr{G}$ is an $L$-homomorphism aligned with $\ch{H}$.
Then $\langle \ch{H}, \phi(j)\rangle$ is an extension of $\ch{H}$ of order two.
In general, it is not isomorphic to the $L$-group of~$H$,
but it is still crucial for describing the Langlands correspondence.

Suppose we are given a complex torus $H$ and an involution $\tau$.
An \emph{$E$-group for $H$ and $\tau$} is a group $\egr{H}=\langle \ch{H}, \ch{\xi}\rangle$
where $\ch{\xi}$ acts on $\ch{H}$ by the dual involution $\ch{\tau}$,
and  $\xi^2\in\ch H^{\ch\tau}$.


The way $E$-groups enter the discussion below is through their relation
with characters of certain double covers of real forms of $H$.
Suppose $\varrho$ is an element of $\frac{1}{2}X^\ast(H)$.
Define
\[
 \widetilde{H}_\varrho=
 \left\{\, (h,z) \in H \times \mathbb{C}^\times \ : \ (2\varrho)(h)=z^2\,\right\}.
 \] This is a two-fold cover of $H$, with projection  $(h,z) \mapsto h$.

Given an involution $\tau$ of $H$, consider the corresponding real form $H(\mathbb{R}, \tau)$.
(Since $H$ is abelian, this is canonically defined; see \S \ref{sec:notation_real_forms}.)
Lift it to a subgroup $\widetilde{H}(\mathbb{R}, \tau)_{\varrho}$ of~$\widetilde{H}_\varrho$.
Again, this is a two-fold cover of $H(\mathbb{R}, \tau)$; we will call it the \emph{$\varrho$-cover} of $H(\mathbb{R},\tau)$.
A character of $\widetilde{H}(\mathbb{R}, \tau)_{\varrho}$ is called \emph{genuine}
if is nonconstant on the fibers of the covering $\widetilde{H}(\mathbb{R}, \tau)_{\varrho}$.


We will use $E$-groups to parametrize the genuine characters of  $\widetilde{H}(\mathbb{R}, \tau)_{\varrho}$.
Viewing~$\varrho$ as an element of $\ch{\mathfrak{h}}$, consider  $\egr{H}=\langle \ch{H}, \ch{\xi}\rangle$,
where $\ch{\xi}$ acts on $\ch{H}$ by $\ch{\tau}$ and~$\ch{\xi}^2=\exp(2i\pi\varrho)$.
This is  uniquely determined by $\varrho$ up to isomorphism.

We say a homomorphism $\phi\colon \mathbf{W}_{\mathbb{R}} \to \egr{H}$  is \emph{admissible}
if it is continuous and $\phi(j) \in \egr{H}\setminus\ch{H}$.
Then the genuine characters of $\widetilde{H}(\mathbb{R}, \tau)_{\varrho}$ can naturally be parametrized
by $\ch{H}$-conjugacy classes of admissible homomorphisms $\mathbf{W}_{\mathbb{R}} \to \egr{H}$.
For a complete discussion of the correspondence, see \cite[Lemma 3.3]{Contragredient} and \cite[Section 5]{AV1}.
To give a quick description, note that an admissible homomorphism  $\phi\colon \mathbf{W}_{\mathbb{R}} \to \egr{H}$ can be written
\begin{equation}\label{phi_for_duality} \begin{cases}
 \phi(z)  = z^{\lambda} \overline{z}^{\ch{\tau}(\lambda)}  \quad (z \in \mathbb{C}^\times), \\
 \phi(j)  = \exp(2i\pi\mu) \ch{\xi},\end{cases}
\end{equation}
where $\lambda, \mu \in \ch{\mathfrak{h}}$,
and where the fact that $\phi$ is a group homomorphism forces
\[ \kappa = \frac12 (1-\ch{\tau})\lambda - (1+\ch{\tau})\mu \]
to be an element of $\varrho + X^\ast(H)$ satisfying $(1+\tau)\lambda = (1+\tau)\kappa$.
There is a unique genuine character~$\Lambda_\phi$ of~$\widetilde{H}(\mathbb{R}, \tau)_{\varrho}$ which satisfies:
(1) $d\Lambda_\phi=\lambda\in\mathfrak h^*$
and
(2) the restriction of $\Lambda_\phi$
to the canonical maximal compact subgroup of~$\widetilde{H}(\mathbb{R}, \tau)_{\varrho}$
is given by $\kappa$. See \cite[Proposition~5.8]{AV1}.
The character~$\Lambda_\phi$ depends only on the~$\ch{H}$-conjugacy class of $\phi$,
and  $\phi \rightsquigarrow \Lambda_\phi$ induces a bijection
between~$\ch{H}$-conjugacy classes of admissible homomorphisms $\phi\colon \mathbf{W}_{\mathbb{R}} \to \egr{H}$
and genuine characters of $\widetilde{H}(\mathbb{R}, \tau)_{\varrho}$.

\subsection{Complete Langlands Parameters}
\label{sec:componentgroups}
\label{sec:def_sphi}
\label{subsec:cover}
\label{sec:def_complete_parameters}

Given an $L$-homomorphism $\phi\colon \mathbf{W}_{\mathbb{R}} \to \lgr{G}$,
consider the centralizer
\begin{equation} \label{centralizer_phi}
\ch{G}_\phi=\mathrm{Cent}_{\ch{G}}(\phi(\mathbf{W}_{\mathbb{R}})).
\end{equation}
Let $\mathbb{S}_\phi$ be the component group of $\ch{G}_\phi$.
It is a finite product of copies of $\mathbb{Z}/2\mathbb{Z}$.


We introduce a canonical covering $\widetilde{\mathbb{S}}_\phi$, as in \cite[Definition 5.11]{ABV}.
Let $\ch{G}^{\mathrm{alg}}$ be the projective limit of all finite coverings of $\ch{G}$.
There is an exact sequence
\[ 1 \to \pi_1(\ch{G})^{\mathrm{alg}} \to \ch{G}^{\mathrm{alg}}\,{}\overset{\mathrm{proj}}\longrightarrow\,{} \ch{G}\to 1\]
where $\pi_1(\ch{G})^{\mathrm{alg}}$ is the projective limit of all finite quotients of $\pi_1(\ch{G})$.
Consider the inverse image \mbox{$\ch{G}_{\phi}^{\mathrm{alg}} = \mathrm{proj}^{-1}(\ch{G}_\phi)$},
and define  $\widetilde{\mathbb{S}}_\phi$ to be the component group of $\ch{G}^{\mathrm{alg}}_\phi$.
This is  an abelian group, possibly infinite.
There is a canonical surjection $\widetilde{\mathbb{S}}_\phi \to \mathbb{S}_\phi$. See  \cite[p.~61]{ABV}.
When a character of $\widetilde{\mathbb{S}}_\phi$ is in the image
of the dual injection $\Char(\mathbb{S}_\phi) \hookrightarrow \Char(\widetilde{\mathbb{S}}_\phi)$,
we shall sometimes say (improperly) that it  is, in fact, a character of $\mathbb{S}_\phi$.

We define a \emph{complete Langlands parameter} to be a pair $(\phi,\chi)$ consisting of
a Langlands parameter $\phi$ and a character $\chi$ of $\tSphi$.
The conjugation action of $\ch G$ on $L$-homomorphisms extends, after passage to coverings,
to an action of $\chGalg$ on pairs $(\phi,\chi)$.
More precisely, suppose $\phi$ is an $L$-homomorphism.
Then for any element~$\ch{\tilde g}$  in the covering~$\ch{G}^{\mathrm{alg}}$,
conjugation by the element~$\ch{g}=\mathrm{proj}(\ch{\tilde g})$ of $\ch{G}$ takes $\ch{G}_{\phi}$ to $\ch{G}_{\mathrm{int}(\ch{g})\phi}$.
Furthermore, conjugation by $\ch{\tilde g}$ induces a bijection
$\Char(\ch{\tilde g}): \Char(\widetilde{\mathbb{S}}_\phi) \to \Char(\widetilde{\mathbb{S}}_{\mathrm{int}(\ch{g})\phi})$
of character groups.
We say two complete Langlands parameters are \emph{equivalent} if they are conjugate by $\chGalg$.

\begin{exem}
\label{ex:fundamental2}
The complete Langlands parameter $(\phi,\mathbf{1})$, where $\mathbf{1}$ is the trivial character of~$\widetilde{\mathbb{S}}_\phi$,
plays a special role:  this defines a particular generic representation of the quasisplit form of $G$.

In particular, consider the $L$-homomorphism $\phi=\phi(\rho,y)$ of Example \ref{ex:fundamental}. Then $(\phi,\mathbf{1})$ is a complete Langlands parameter. 
The corresponding representation is a {\it large} fundamental series of the quasisplit form of $G$. This plays an important
role in the relationship with Whittaker models. For details see the \href{sec:appendix}{Appendix}.
\end{exem}



\section{Dictionary between characters of component groups and \texttt{KGB} elements}\label{sec:dictionary_KGB_charcompgroup}

Suppose $\phi\colon \mathbf{W}_{\mathbb{R}} \to \lgr{G}$ is an $L$-homomorphism.
Let $\Pi(\phi)$ be the corresponding `large' $L$-packet:
it comprises representations of the various (strong) real forms of $G$ in the given inner class.
For a description of the subset of $\Pi(\phi)$ attached to each real form of $G$,
see \cite{Borel, Contragredient} and \S \ref{sec:L_packets} below.

The refined version of the Langlands correspondence
parametrizes the individual representations in $\Pi(\phi)$
by characters of the abelian group $\widetilde{\mathbb{S}}_\phi$.
In this section and the next, we give an exposition of the refined parametrization.
In the case of  discrete series  this is discussed in~\cite{DiscreteSeriesSigns}.

The subtle aspect of the refined correspondence
is how a character of $\tSphi$ determines first of all a strong real form of $G$,
and secondly a representation in the $L$-packet for this real form.
The view that we shall adopt here is that
the \texttt{KGB} space $\mathcal{X}$ of \S \ref{sec:KGB} is perfectly suited for this.
We describe a natural bijection between the character group $\Char(\widetilde{\mathbb{S}}_\phi)$
and a subset of $\mathcal{X}$.
This is precisely the information we need. For instance, it makes the first step quite clear:
given a character  of $\tSphi$, the corresponding element of $\X$ defines
the appropriate strong real form of $G$.


\subsection{$L$-homomorphisms in standard form} \label{sec:tau_phi}
\label{sec:involutions}
\label{sec:ambiguity_tau}
\label{sec:construction_involution}


Let us begin with an $L$-homomorphism  $\phi\colon \mathbf{W}_{\mathbb{R}} \to \lgr{G}$.
After conjugating by $\ch G$ we may assume it is
aligned with $\ch{H}$, and write it as in \eqref{e:phi}.
Our bijection between $\Char(\widetilde{\mathbb{S}}_\phi)$ and a subset of $\mathcal{X}$
will have its image entirely contained in a single fiber $\mathcal{X}_{\tau}$ of $\Char(\widetilde{\mathbb{S}}_\phi)$.

We are given the involution $\int(\phi(j))$ of $\ch H$.
If the infinitesimal character for~$\phi$ is regular, then this involution defines the correct
real form of the dual Cartan subgroup. If the infinitesimal character is singular,
there is a choice involved and we may want to use a different involution~$\ch{\tau}$.

For example, suppose $G=\mathrm{PGL}(2,\R)$, $\ch G=\SL(2,\C)$,
$\ch H=\{\mathrm{diag}(z,z^{-1})\}$ and
$\phi(\C^\times)=1$. If  $\phi(j)=\mathrm{diag}(i,-i)\ch\delta$,
then $\phi(j)$ acts trivially on $\ch H$. However $\phi(j)$ is conjugate to $\left(\begin{smallmatrix}0&1\\-1&0\end{smallmatrix}\right)\ch\delta$,
which acts by inversion on $\ch H$; so after conjugating $\phi$ we may get a different involution of~$\ch{H}$.
We prefer the second choice, which makes $\ch H(\R)$ split, and therefore makes $H(\R)$ compact.
This example is key to the discussion of the $L$-packet of limits of discrete series for $\mathrm{SL}(2,\R)$.

We are free to modify $\phi$ by the Weyl group $W(\ch G,\ch H)$. We could therefore assume $\lambda$ is weakly dominant for the set~$\Delta^+$ of positive roots attached to our fixed
Borel subgroup~$B$. For the needs of Section~\ref{sec:local_langlands} it is convenient to impose a weaker condition, which appears in the following definition.

\begin{defi}
\label{d:stdform}
Suppose $\phi$ is an $L$-homomorphism. We say $\phi$ is in \emph{standard form} if it satisfies the following three conditions. First,
\begin{enumerate}
\item[(1)] $\phi$ is aligned with $\ch H$ (cf. \eqref{e:phi}).
\end{enumerate}
Define $\ch L_\phi=\Cent_{\ch G}(\phi(\C^\times))$ and let $\ch\tau$ be the involution $\int(\phi(j))$ of $\ch{L}_\phi$, as in \eqref{e:L} and~\eqref{e:phia}. Then we require:
\begin{enumerate}
\item[(2)] $\ch H$ is maximally $\ch{\tau}$-split in $\ch L_\phi$;
\item [(3)] $\lambda$ is {\it weakly integrally dominant} for the positive $\ch\tau$-real roots: \\ if $\alpha$ is a positive $\ch\tau$-real root then $\langle \gamma,\ch\alpha\rangle\not\in\{-1, -2, -3, \dots\}$.
\end{enumerate}

\end{defi}

By definition (2) means $\ch{H}$ is a $\ch\tau$-stable maximal torus of $\ch{L}_\phi$
and the~$(-1)$ eigenspace of $\ch\tau$ on ${\ch{\mathfrak{t}}}$ is a maximal semisimple subalgebra
in the $-1$ eigenspace $(\ch{\mathfrak{l}}_\phi)^{-\ch\tau}$. See \cite[Lemma~12.10]{ABV}.
For a discussion related to (3) see \cite[(3.12)--(3.15)]{TwistedParameters}.

If $\phi$ is in standard form, we denote by $\ch{\tau}(\phi)$ the involution $\int(\phi(j))$ of $\ch{H}$, and denote by~$\tau(\phi)$ the corresponding involution $-\ch\tau(\phi)^\top$ of~$H$.

For further discussion of this notion in the case of $\mathrm{SL}(2, \R)$, see Section~\ref{sec:example_SL2}.




Suppose  $\phi\colon \mathbf{W}_{\mathbb{R}} \to \lgr{G}$ is an $L$-homomorphism in standard form, and
let $\tau=\tau(\phi),\ch\tau=\ch\tau(\phi)$. In the rest of this section, we will explore the link between the component group~$\widetilde{\mathbb{S}}_\phi$ and the fiber~$\mathcal{X}_{\tau}$ of the \texttt{KGB} space.
In \S\ref{sec:S_tau} and \S\ref{sec:S_phi_and_S_tau}, we will describe a free action of $\Char(\widetilde{\mathbb{S}}_\phi)$ on  $\mathcal{X}_{\tau}$.
In~\S\ref{sec:basepoint}, we will choose a basepoint in~$\mathcal{X}_{\tau}$,
and obtain a bijection between $\Char(\widetilde{\mathbb{S}}_\phi)$ and a subset of $\mathcal{X}_{\tau}$.


\subsection{The groups $\widetilde{\mathbb{S}}_{\ch{\tau}}$}\label{sec:S_tau}
Given an involution~$\tau$ of~$H$, let~$\ch{\tau}$ be the dual involution of~$\ch{H}$. Set
\begin{equation} \label{s_tau}
 \mathbb{S}_{\ch{\tau}} = \text{component group of $(\ch{H})^{\ch{\tau}}$.}
  \end{equation}
 Form the covering $\widetilde{\mathbb{S}}_{\ch{\tau}}$, as in \S \ref{subsec:cover}: if $\ch{H}^{\mathrm{alg}}$ is the preimage of $\ch{H}$ in the covering $\ch{G}^{\mathrm{alg}} \to \ch{G}$, let~$\ch{H}^{\alg, \ch{\tau}}$ be the preimage of~$(\ch H)^{\ch\tau}$, and let~$\Schtautilde$ be the component group of~$\ch{H}^{\alg, \ch{\tau}}$. It is an abelian group, possibly infinite. The natural map $\ch{H}^{\alg, \ch{\tau}} \to \ch{H}^{\ch{\tau}} \to \mathbb{S}_{\ch{\tau}}$ is surjective, and its kernel contains the identity component $(\ch{H}^{\alg, \ch{\tau}})_0$; therefore it induces a surjection $\widetilde{\mathbb{S}}_{\ch{\tau}} \to {\mathbb{S}}_{\ch{\tau}}$.
 
 The  character group~$\Char(\widetilde{\mathbb{S}}_{\ch{\tau}})$ is crucial to our dictionary between characters of component groups and \texttt{KGB} elements,
because it turns out to have a natural simply transitive action on the fiber  $\mathcal{X}_{\tau}$. Let us explain this.


\subsubsection{} \label{sec:isomorphism_torusgroups}
First, we observe that $\Char(\widetilde{\mathbb{S}}_{\ch{\tau}})$ is isomorphic with a group
that can be defined entirely on the $G$-side.
This discussion follows \cite[Chapter 9]{ABV}. Define 
$$
H^{-\tau}=\{h\in H\mid \tau(h)=h\inv\}.
$$
We will need the larger group
$$
H^{-\tau}_Z=\{h\in H\mid h\tau(h)\in Z(G)\}\supset H^{-\tau}.
$$
Let $A_\tau$ be the identity component of $H^{-\tau}$. Then $A_\tau=\{h\tau(h\inv)\mid h\in H\}$ and
$$
A_\tau\subset H^{-\tau}\subset H^{-\tau}_Z.
$$
We denote by~$\U_\tau$ the quotient $H^{-\tau}/A_\tau$, and by $\Ut_\tau$ the quotient $H^{-\tau}_Z/A_\tau$. The group $\U_\tau$ may be viewed as a subgroup of $\Ut_\tau$. We will define explicit isomorphisms $\Ut_\tau\simeq \Char(\widetilde{\mathbb{S}}_{\ch{\tau}})$ and $\U_\tau\simeq \Char({\mathbb{S}}_{\ch{\tau}})$.

Let us begin with the natural short exact sequence
\begin{equation}\label{exact_sequence_1}
0\rightarrow \mathbb S_{\ch\tau}={\ch H}^{\ch\tau}/(\ch H^{\ch\tau})_0 \rightarrow \ch H/(\ch H^{\ch\tau})_0 \rightarrow \ch H/\ch H^{\ch\tau}\rightarrow 0
\end{equation}
where all maps are induced by the inclusion of $\ch{H}^{\ch{\tau}}$ into~$\ch{H}$. Passing to character groups, we get an exact sequence
$$
0\rightarrow \Char(\ch H/\ch H^{\ch \tau}) \rightarrow \Char(\ch H/(\ch H^{\ch\tau})_0)\rightarrow \Char(\Schtau)\rightarrow 0.
$$
We may view the characters of $\ch H/\ch H^{\ch \tau}$ (resp. $\ch H/(\ch H^{\ch \tau})_0$) as (algebraic) characters of~$\ch{H}$ which vanish on~$\ch{H}^{\ch{\tau}}$ (resp. $(\ch H^{\ch \tau})_0$). Under the canonical isomorphism $X^*(\ch H)\simeq X_*(H)$, the group of characters of~$\ch{H}$ which vanish on~$\ch{H}^{\ch{\tau}}$ is identified with $(1-\tau)X_*(H)$, and the group of characters which vanish on~$(\ch H^{\ch\tau})_0$  is identified with~$X_\ast(H)^{-\tau}$ (see~\cite[Lemma 9.5]{ABV}). Therefore the previous exact sequence becomes
\begin{equation}\label{exact_sequence_2}
0\rightarrow (1-\tau)X_*(H)\rightarrow X_*(H)^{-\tau}\rightarrow \Char(\Schtau)\rightarrow 0
\end{equation}
 as in~\cite[Proposition 9.6]{ABV}. 
Now, viewing $X_*(H)$ as a subset of the Lie algebra of $H$, the map $\mu\mapsto \exp(\mu/2)$
is a surjective homomorphism from $X_*(H)^{-\tau}$ to $H^{-\tau}/A_\tau=\U_\tau$, with kernel $(1-\tau)X_*(H)$ (see \cite[Proposition 9.10]{ABV}):
$$
0\rightarrow (1-\tau)X_*(H)\rightarrow X_*(H)^{-\tau}\rightarrow  \U_\tau\rightarrow 0.
$$
Combined with the previous exact sequence this yields an isomorphism
\begin{equation}\label{iso_U_Pi} \Char(\Schtau)\overset{\sim}{\longrightarrow} \U_\tau. \end{equation}

We modify this slightly in order to incorporate coverings; see \cite[Propositions 9.8 and 9.10]{ABV}. 
The group $X^*(\ch H^{\mathrm{alg}})$ 
 of rational characters of~$\ch{H}^{\mathrm{alg}}$
can be identified with $X_\ast(H)_\Q=X_\ast(H) \otimes_{\mathbb{Z}} \mathbb{Q}$, see \cite[(9.7)]{ABV}.
Consider $\ch{H}^{\alg, \ch{\tau}}$ in place of~$\ch{H}$ in~\eqref{exact_sequence_1}, pass to character groups and identify characters of $\ch{H}^{\alg}/\ch{H}^{\alg,\ch \tau}$ (resp. $\ch{H}^\alg/(\ch H^{\alg,\ch{\tau}})_0$) with (rational) characters of~$\ch{H}^\alg$ which vanish on~$\ch{H}^{\alg,\ch{\tau}}$ (resp. $(\ch H^{\alg, \ch \tau})_0$). Then the exact sequence~\eqref{exact_sequence_2} is replaced by
$$
0\rightarrow (1-\tau)X_*(H)\rightarrow X_*(H)^{-\tau}_\Q\rightarrow \Char(\Schtautilde)\rightarrow 0.
$$
As in the previous case the map $\mu\mapsto \exp(\mu/2)$ induces an exact sequence
\begin{equation}\label{compatibility}
0\rightarrow (1-\tau)X_*(H)\rightarrow X_*(H)^{-\tau}_\Q\rightarrow \Ut_\tau\rightarrow 0
\end{equation}
and an isomorphism 
\begin{equation}\label{iso_Ut_Pit}
\Char(\Schtautilde)\overset{\sim}{\longrightarrow} \Ut_\tau.
\end{equation}
Furthermore the exact sequence~\eqref{compatibility} is compatible with the inclusions of $X_*(H)^{-\tau}$ into $X_*(H)^{-\tau}_\Q$ and of $\U_\tau$ into $\Ut_\tau$. Therefore we can  sum up this discussion as follows:
 
\begin{lemm} \label{isoms_U} The maps $\Char(\Schtau)\to \U_\tau$ and $\Char(\Schtautilde)\to\Ut_\tau$ in~\eqref{iso_U_Pi} and~\eqref{iso_Ut_Pit}  are isomorphisms, and fit into a commutative diagram
\begin{equation}
\xymatrix{
\U_\tau\ar[d]_\simeq\ar@{^{(}->}[r]&\Ut_\tau\ar[d]^\simeq\\
\Char(\Schtau)\ar@{^{(}->}[r]&\Char(\Schtautilde)
}
\end{equation}
where the lower horizontal arrow is dual to the surjection $\widetilde{\mathbb{S}}_{\ch{\tau}} \to {\mathbb{S}}_{\ch{\tau}}$, and the top arrow is the inclusion of $\U_{\tau}$ into $\Ut_\tau$.
\end{lemm}


\subsubsection{}\label{sec:action_F_fiber}
We can now relate $\Pi(\Schtautilde)$ to the \texttt{KGB} space: we describe a simply transitive action of~$\Char(\Schtautilde)$ on~$\X_\tau$.
By Lemma~\ref{isoms_U} this amounts to describing a simply transitive action of~$\Ut_\tau$ on~$\X_\tau$. We follow \cite[Proposition~11.2]{Algorithms}.

Recall from \S \ref{sec:def_KGB_space} that $\mathcal{X}$ is a quotient of
$\widetilde{\mathcal{X}} = \left\{ \, \xi \in \mathrm{Norm}_{G^\Gamma}(H) \ : \ \xi^2 \in Z(G)\, \right\}$,
and the projection $p\colon\widetilde{\mathcal{X}} \to \mathcal{X}$
sends an element $\xi \in  \widetilde{\mathcal{X}}$ to its $H$-conjugacy class.
Set $\widetilde{\mathcal{X}}_{\tau}=p^{-1}(\mathcal{X}_{\tau})$.
Then if we fix an element $\xi \in \widetilde{\mathcal{X}}_{\tau}$,
we have  $\widetilde{\mathcal{X}}_{\tau} = \{ \, h\xi \ : \ h \in H'_{-\tau}\, \}$.

Thus the group $H^{-\tau}_Z$ acts on $\widetilde{\mathcal{X}}_{\tau}$
by multiplication on the left,
and clearly that action is simply transitive.
It descends to an action of $H^{-\tau}_Z$ on the fiber $\mathcal{X}_\tau$,
which is still transitive but no longer free.
For $h \in H$, we have $h\xi h^{-1}=h\tau(h^{-1})\xi$,
and so the stabilizer of $p(\xi)$ is $A_{\tau}$. Since $H_Z^{-\tau}/A_\tau=U_\tau$ we deduce:
\begin{lemm}\label{action_toruscomp}
The action of $H^{-\tau}_Z$ on $\mathcal{X}_\tau$
induces a simply transitive action of $\Ut_\tau$ on $\mathcal{X}_\tau$.
 
\end{lemm}

\subsection{The canonical basepoint in  $\mathcal{X}_{\tau}$}
\label{sec:basepoint}
\label{sec:twisted_involutions}
\label{sec:def_tits}
\label{def_representative_basepoint}

Choosing a basepoint in $\mathcal{X}_{\tau}$ will convert the action of Lemma \ref{action_toruscomp}
into a bijection between $\Char(\Schtautilde)$ and $\X_\tau$.
Therefore we are looking for a privileged choice of basepoint $x_{b,\tau} \in \mathcal{X}_{\tau}$
in each fiber $\X_\tau$. 
It  corresponds to the trivial character of $\widetilde{\mathbb{S}}_{\ch{\tau}}$, and therefore
determines a special element of the $L$-packet $\Pi(\phi)$. See the end of 
Section \ref{sec:def_complete_parameters}, as well as the \href{sec:appendix}{Appendix} for the relationship with Whittaker models.

Let us define $x_{b,\tau}$ for $\tau \in \mathcal{I}_W$.
We begin with a special case.
Recall we are working with the extended group $G^\Gamma = \langle G, \xi_\gamma\rangle$,
where $\xi_{\gamma}$ satisfies $\xi_{\gamma}^2 = 1$
and $\mathrm{int}(\xi_\gamma)$ acts on $G$ by the distinguished automorphism $\gamma$
(see \S\ref{extended_group}).
Since $\gamma$ preserves $H$, the element $\xi_{\gamma}$ is a strong involution of $G$;
we denote by $\mathcal{X}_{\tau_{\gamma}}$ the corresponding fiber,
and call it \emph{distinguished}.
Define the basepoint $x_{b, \tau_\gamma}$ in that fiber to be the image in $\X$ of
\[
\xi_{b, \tau_\gamma} = \exp(i\pi\ch{\rho})\,{}\xi_{\gamma}.
\]
This satisfies $\xi_{b, \tau_\gamma}^2 =\exp({2i\pi\ch{\rho}})= z_\ast$,
so the strong real form corresponding to $x_{b, \tau_\gamma}$ is pure;
we call it the distinguished strong real form of $G$ (attached to $\gamma$).

Let us mention that the element $x_{b, \tau_\gamma}$ is `large', i.e.
the simple imaginary roots for $\tau_{\gamma}$ are all noncompact with respect to $x_{b, \tau_\gamma}$.
See \cite[\S 12]{Algorithms} and Remark \ref{remark:noncompact_singular} below.
In particular, the corresponding real form of $G$ is quasisplit.
Furthermore, $\mathrm{int}(\xi_\gamma)$ acts on $\mathrm{Norm}_G(H)$,
and descends to an automorphism $\gamma^W$ of $W$.

To define the basepoints in the other fibers,
we need to interpret the involutions $\tau$ of~$H$ in terms of the Weyl group.
Thus fix $\tau \in \mathcal{I}_W$.
Given a strong involution $\xi$ representing an element of $\mathcal{X}_\tau$,
we can consider the element $\xi \xi_\gamma^{-1}$ of $G$;
it normalizes $H$, and we let $w_{\tau}$ be its image in the Weyl group.
This is independent of the choice of element in~$\mathcal{X}_{\tau}$ and representative,
and satisfies $w_\tau\,{}\gamma^W(w_\tau)=1$.
We say $w_{\tau}$ is a twisted involution in $W$.
The map $\tau \mapsto w_{\tau}$ sets up a bijection
between $\mathcal{I}_W$ and the set of twisted involutions in $W$:
given a twisted involution~$w$,
the corresponding involution of $H$ is $w \circ \mathrm{int}(\xi_\gamma)$.
The Weyl group acts on $\I_W$ by conjugation, and on twisted involutions
by twisted conjugation ($w\in W$ acts by $y\mapsto wy\gamma^w(w\inv)$); the bijection $\tau \mapsto w_{\tau}$ intertwines these two actions.
The involution $\tau_\gamma$ of $H$ is mapped to the identity.

For our last ingredient, we use the \emph{Tits group}  to choose a canonical set-theoretic splitting
of the map $p:\Norm_G(H)\rightarrow W$. 
This depends on our choice of pinning $\P$ of $G$.
See \cite[\S 15]{Algorithms}. Here is a short summary of the construction.

For  $\alpha$ a simple root, with corresponding root vector $X_\alpha$ from the pinning, 
there is a canonical homomorphism $\phi_\alpha \colon \mathrm{SL}(2, \mathbb{C}) \to G$
satisfying:
$\phi_\alpha(\text{diagonal matrices}) \subset H$
and
$d\phi_\alpha\left(\begin{smallmatrix} 0 & 1 \\ 0 & 0 \end{smallmatrix}\right) = X_\alpha$.
Then we set  $\sigma_{s_\alpha} = \phi_\alpha \left(\begin{smallmatrix} 0 & 1 \\ -1 & 0 \end{smallmatrix}\right)\in\Norm_G(H)$,
and also denote it by $\sigma_{\alpha}$.

When $w \in W$ is arbitrary,  with reduced expression $w=s_{\alpha_1} \dots s_{\alpha_r}$,
we set $\sigma_w = \sigma_{\alpha_1} \dots \sigma_{\alpha_r}$.
Then $\sigma_w$ is independent of the reduced expression, and the map $W\rightarrow \Norm(H):w\rightarrow \sigma_w$ 
satisfies $p(\sigma_w)=w$ for all $w\in W$. 
The Tits group for $(G, \mathcal{P})$ is the subgroup of $\mathrm{Norm}_{G}(H)$ generated by  the~$\sigma_{\alpha}$.

We can now define the basepoint in each fiber, following \cite[Section 3]{TwistedParameters}.
For $\tau$ in $\mathcal{I}_W$, consider the element $w_{\tau}$ of $W$, its representative $\sigma_{w_\tau}$ in the Tits group, and set
\begin{equation} \label{def_basepoint}
 \xi_{b,\tau}= \exp(i\pi\ch{\rho})\,{}\sigma_{w_\tau}\,{}\xi_{\gamma}.
 \end{equation}
This is a strong involution of $G$, satisfies $\xi_{b,\tau}^2 = z_\ast$, and is conjugate to $\xi_{b,\tau_\gamma}$.
See \cite[Proposition~3.2]{TwistedParameters} and the proof of Lemma~\href{l:Q}{A.4}.
Finally, we define the basepoint in $\mathcal{X}_\tau$ to be
\[
 x_{b,\tau} = \text{image of $\xi_{b, \tau}$ in the fiber $\mathcal{X}_\tau$}.
 \]
Thus the various basepoints $x_{b,\tau}$, $\tau \in \mathcal{I}_W$, are all $G$-conjugate:
 they all define the same strong real form of $G$,
namely the distinguished quasisplit strong real form of $G$ attached to $\gamma$.


Using this choice of basepoints we deduce: 

\begin{prop} \label{prop:bij_char_KGB_4}
Let~$\tau$ be an involution of~$H$. Using  the free action of~$\Char(\Schtautilde)$ on~$\mathcal{X}_{\tau}$ 
and the basepoint for $\mathcal{X}_{\tau}$ defined in the previous section, we obtain a bijection
\begin{equation}
\label{e:D}
\Char(\Schtautilde)\leftrightarrow \X_\tau.
\end{equation}
This bijection maps the trivial character of $\Schtautilde$ to the canonical basepoint $x_{b,\tau}$ in $\mathcal{X}_{\tau}$.
\end{prop}

\subsection{Relationship between  $\widetilde{\mathbb{S}}_{\phi}$ and $\widetilde{\mathbb{S}}_{\ch{\tau}}$}\label{sec:S_phi_and_S_tau}

We now observe that $\Char(\widetilde{\mathbb{S}}_\phi)$ naturally embeds in~$\Char(\widetilde{\mathbb{S}}_{\ch{\tau}})\simeq \Ut_\tau$. The crucial fact is that the ``maximally $\ch{\tau}$-split'' condition in Definition~\ref{d:stdform} implies that~$\ch{H}$ meets every component of $\ch{G}_\phi = \mathrm{Cent}_{\ch{G}}(\phi(\mathbf{W}_{\mathbb{R}}))$: see \cite[Lemma 12.10]{ABV}. Since~$\mathbb{S}_\phi$ is the component group of $\ch{G}_\phi$
and  $\mathbb{S}_{\ch{\tau}}$ is the component group of $\ch{H}^{\ch{\tau}} \subset \ch{G}_\phi$,
 we get a \emph{surjection} $ \mathbb{S}_{\ch{\tau}} \to \mathbb{S}_{\phi}$. The same fact applied to coverings, as in \cite[(12.11)(e)]{ABV}, yields a canonical surjection
\begin{equation} \label{surj_compgroup}
p_{\phi} \colon \widetilde{\mathbb{S}}_{\ch{\tau}} \twoheadrightarrow \widetilde{\mathbb{S}}_{\phi}.
\end{equation}

By Pontryagin duality, we get a {canonical injection}
\begin{equation} \label{inj_toruscomp}
\beta_\phi\colon \Char(\widetilde{\mathbb{S}}_\phi)\hookrightarrow \Char(\widetilde{\mathbb{S}}_{\ch{\tau}}).
 \end{equation}
If the infinitesimal character for $\phi$ is regular,
then \eqref{surj_compgroup} and \eqref{inj_toruscomp} are bijections: see \cite[(12.4)(c)]{ABV}.
In general the image of \eqref{inj_toruscomp} is the set of characters of $\widetilde{\mathbb{S}}_{\ch{\tau}}$ that are $\phi$-\emph{final}
in the sense of \cite[Definition 12.8]{ABV}.
This has the following meaning.

Suppose $\alpha$ is a~$\ch{\tau}$-real root of~$\ch{H}$ in~$\ch{G}$.
It determines a distinguished element~\mbox{$\ch{\tilde{m}}_\alpha =\alpha^\vee(-1)$},
of order  $1$ or $2$ in~$(\ch{H})^{\mathrm{alg}, \ch{\tau}}$, as follows. Consider the root subgroup morphism~\mbox{$\mathrm{SL}(2,\C)\to \ch{G}$} attached to $\alpha$; since $\mathrm{SL}(2,\C)$ is simply connected, it lifts to a continuous morphism~\mbox{$\mathrm{SL}(2,\C)\to~\chGalg$}. Set~$\ch{\tilde{m}}_\alpha =\alpha^\vee(-I_2)$. This is an element
of order $1$ or $2$ in~$(\ch{H})^{\mathrm{alg}, \ch{\tau}}$, and we let  $\ch{\overline{m}}_\alpha$ be the corresponding element of $\widetilde{\mathbb{S}}_{\ch{\tau}}$. See \cite[pp. 141--142]{ABV}. 

The kernel of~\eqref{surj_compgroup} is generated by the elements $\ch{\overline{m}_\alpha}$
for those roots $\alpha$ which are~$\ch{\tau}$-real  and $\phi$-singular
(i.e. orthogonal to the infinitesimal character representative~$\lambda$).
Then  $\chi \in \Char(\widetilde{\mathbb{S}}_{\ch{\tau}})$ is called \emph{$\phi$-final}
if $\chi(\ch{\overline{m}_\alpha})=1$ whenever~$\alpha$ is a~$\ch{\tau}$-real and $\phi$-singular root.
We denote by $\Char_{\phi,\,{}\mathrm{fin}} (\widetilde{\mathbb{S}}_{\ch{\tau}})$
the $\phi$-final part of~$\Char(\widetilde{\mathbb{S}}_{\ch{\tau}})$, and sum~up:

\begin{lemm} \label{lem:action_charcompgroup}
Restricting the action of Lemma \ref{action_toruscomp}
to the image $\Char_{\phi,\,{}\mathrm{fin}} (\widetilde{\mathbb{S}}_{\ch{\tau}})$ of \eqref{inj_toruscomp},
we obtain a free action of $\Char(\widetilde{\mathbb{S}}_\phi)$ on $\mathcal{X}_{\tau}$.
\end{lemm}

Combining this with the choice of basepoint in the previous subsection, we obtain:

\begin{prop} \label{prop:bij_char_KGB_param}
Let $\phi\colon \mathbf{W}_{\mathbb{R}} \to\lgr{G}$ be an $L$-homomorphism in standard form,
and let $\tau$ be the involution $\tau(\phi)$ of~$H$ \textup{(}Definition~\ref{d:stdform}\textup{)}. The injection $\Char(\widetilde{\mathbb{S}}_\phi)\hookrightarrow \Char(\widetilde{\mathbb{S}}_{\ch{\tau}})$ from \eqref{inj_toruscomp} and the bijection $\Char(\widetilde{\mathbb{S}}_{\ch{\tau}}) \leftrightarrow \X_\tau$ from \eqref{e:D}  yield  a natural injection
\begin{equation}
\label{e:E}
\E_\phi: \Char(\tSphi)\hookrightarrow \X_\tau.
\end{equation}
The map $\E_\phi$ sends the trivial character of $\widetilde{\mathbb{S}}_\phi$ to the canonical basepoint $x_{b,\tau}$ in $\mathcal{X}_{\tau}$,
and sends characters of $\mathbb{S}_\phi$ to elements of  $\mathcal{X}_\tau(z_\star)$. \end{prop}

\begin{rema} \label{remark:noncompact_singular}
We record for later use the following property of $\phi$-singular roots with respect to the image of~\eqref{e:E},
which transcribes  \cite[Proposition 13.12(c)]{ABV}.
Suppose $\alpha$ is a $\tau$-imaginary simple root. Let 
\begin{equation}
\label{e:X[alpha]}
\X_\tau[\alpha]=\{x\in \X_\tau\mid \alpha\text{ is $x$-noncompact}\}.
\end{equation}
Suppose $\chi\in \Char(\tSphi)$, $x=\E_\phi(\chi)\in \X_\tau$, and $\alpha$ is $\phi$-singular. Then $x\in \X_\tau[\alpha]$.
 \end{rema}

\subsection{Example: $\mathrm{SL}(2,\mathbb{R})$}\label{sec:example_SL2}

Let $G=SL(2,\C)$. Then $\ch G=PSL(2,\C)\simeq SO(3,\C)$. We choose to write this group as
$$
SO(3,\C)=\{g\in \GL(2,\C)\mid gJg^{t}=J\}, \quad \text{where $J=\left(\begin{smallmatrix}0&1&0\\1&0&0\\0&0&1\end{smallmatrix}\right)$. }
$$
We choose a Cartan subgroup to be  $\ch H=\diag\{(z,\frac{1}{z},1)\}\simeq \C^\times$.
There is only one inner class of real forms of $G$, containing $\SL(2,\R)$ and $\SU(2)$.
Then $\LG =\ch G\times (\Z/2\Z)$, and we can ignore the extension. 

Consider the parameter $\phi:W_\R\rightarrow \LG$ given by 
$$
\begin{aligned}
\phi(z)&=\diag(|z|^\nu,|z|^{-\nu},1)\quad (\nu\in \C),\\
\phi(j)&=\diag(\epsilon,\epsilon,1)\quad (\epsilon=\pm 1).
\end{aligned}
$$
Let $\ch\tau=\int(\phi(j))$.

First suppose $\nu\ne 0$. Then $\phi$ is in standard form: we have $\ch{L}_\phi=\Cent_{\ch G}(\phi(\C^\times))=\ch H$, so~$\ch{H}$ is (obviously) maximally split in $\ch L$ with respect to $\int(\phi(j))$.
Then $\ch\tau$~acts trivially on $\ch H$, i.e.~$\ch H$ is $\ch\tau$-compact,
so by duality $H$ is $\tau$-split, i.e. $H(\R,\tau)\simeq \R^\times$. In this case $\Cent_{\ch G}(\phi)=\ch{H}$, and $\Sphi=1$.

The $L$-packet for~$\mathrm{SL}(2,\mathbb{R})$ attached to~$\phi$ will
consist of a single principal series representation with infinitesimal
character $\nu$, which is  spherical if and only if $\epsilon=1$.

The situation changes significantly when we take $\nu=0$, in which case  $\ch{L}_\phi=\ch{G}$. 
\begin{enumerate}[(a)]
\item  If $\phi(j)=I_3$ then $\phi$ is in standard form: the involution $\ch\tau$ is trivial, the corresponding real form of $\ch G$ is compact, and~$\ch{H}$ is maximally split in~$\ch{G}$ with respect to $\ch\tau$. In this case $\Cent_{\ch G}(\phi)=\ch G$, and $\Sphi=1$ once more.

The $L$-packet for~$\mathrm{SL}(2,\mathbb{R})$ attached to~$\phi$  consists of the induced representation $\Ind_B^G(1)$: the irreducible spherical tempered principal series.

\item On the other hand, suppose $\phi(j)=\diag(-1,-1,1)$,
  so $\ch\tau\in\Aut(\ch H)$ is non-trivial, 
and the corresponding real form of $\ch G$ is split. However~$\ch H$ is~$\ch\tau$-compact, i.e. 
not maximally split in $\ch G$. 

In this case $\phi$ is $\ch G$-conjugate to the parameter $\phi'$ defined by:
$$
\begin{aligned}
\phi'(z)\phi(z)&=I_3,\\
\phi'(j)&=\begin{pmatrix}0&1&0\\1&0&0\\0&0&-1\end{pmatrix}.
\end{aligned}
$$
This time $\ch H$ is split with respect to  $\ch\tau'=\int(\phi'(j))$, so $H(\R,\tau')$ is compact, and~$\phi'$ is in standard form. 

We have  $\Cent_{\ch G}(\phi)=\Cent_{\ch G}(\phi')=S(O(2)\times O(1))\simeq O(2)$, and $\Sphi=\Z/2\Z$.
On the other hand $\ch H^{\ch\tau}$ is connected, so $\Schtau$ is trivial: therefore the map
$\Schtau\rightarrow \Sphi$ defined just before~\eqref{surj_compgroup} is {\it not} surjective, and neither is~\eqref{surj_compgroup}. This is one reason we need to use $\phi'$ instead:
$(\ch H)^{\ch \tau'}=\pm 1$, and $\widetilde{\mathbb{S}}_{\ch{\tau'}} \rightarrow \widetilde{\mathbb{S}}_{\phi'}$ is surjective.

So, in this case we consider the induced representation $\Ind_B^G(\sgn)$. This is tempered, with infinitesimal character $0$, but not spherical, and is reducible: 
it is the direct sum of the two limits of discrete series. The~$L$-packet for~$\mathrm{SL}(2, \R)$ attached to $\phi$  consist of these two limits of discrete series, and switching from $\phi$ to $\phi'$ makes it possible to view them as attached to a compact Cartan subgroup.

\end{enumerate}

\subsection{Cayley transforms, cross actions and characters of component groups} 
\label{sec:cayley_4}
 \label{sec:def_cross}

Using the dictionary in Proposition~\ref{prop:bij_char_KGB_4}, any natural operation on the \texttt{KGB} space can be interpreted in terms of component groups on the dual side. In this section, we spell out such an interpretation in the case of Cayley transforms, which are a key ingredient in  the computation of lowest $K$-types.  The material in this section is not used in the description of the Langlands correspondence in Section~\ref{sec:local_langlands}, but it is important for our proof of Theorem~\ref{main_theorem_rgroups} in Section~\ref{sec:r_groups}.


\subsubsection{Cayley transforms and cross actions in the \texttt{KGB} space}\label{sec:intro_cayley_cross} 

 Given an involution $\tau$ of $H$, recall
 $\mathcal{X}_\tau[\alpha] = \{ x \in \mathcal{X}_\tau \ : \ \alpha \text{ is $x$-noncompact}\}$, see~\eqref{e:Xtau}.
 In Section~\ref{s:crosscayleykgb} we defined the imaginary Cayley transform
\begin{equation} 
c^{\alpha}\colon \mathcal{X}_\tau[\alpha] \to \mathcal{X}_{s_\alpha \tau};
\end{equation}
it is surjective, and at most two-to-one. 

In \S\ref{sec:def_cross_action} we also discussed the cross action of the Weyl group $W=W(G,H)$ on $\X$. Given a $\tau$-complex root $\alpha$, the cross action of the reflection $s_\alpha \in W$ gives rise to a bijection 
\begin{equation} 
(s_\alpha \times~\,{} ) \colon \mathcal{X}_\tau \to \mathcal{X}_{s_\alpha \tau s_\alpha}.
\end{equation}

These operations on \texttt{KGB} elements  are key to the computation of lowest $K$-types in terms of \texttt{atlas} parameters.
First we point out that they preserve the basepoints $x_{b, \tau}$.

\begin{lemm} \label{basepoint}
Let $\tau$ be an involution of $H$.
Suppose $\alpha$ is a simple root. If $\alpha$ is $\tau$-imaginary then $c^\alpha(x_{b,\tau})=x_{b,s_\alpha\tau}$. 
If $\alpha$ is $\tau$-complex then
$s_\alpha\times x_{b,\tau}=x_{b,s_{s_\alpha\tau s_\alpha}}$.
\end{lemm}

\begin{proof}
The basepoint $x_{b,\tau}$ can be characterized as the element of $\mathcal{X}_\tau$
with \emph{trivial normalized torus part}, see \cite[Proposition 3.2]{TwistedParameters}.

Thus what we need to check is that $c^{\alpha}(x_{b, \tau})$ has trivial normalized torus part.
This follows from the calculation of normalized torus parts in Table 2 on page 65 of~\cite{TwistedParameters}:
see lines 3--5 of the table; the calculation in the fourth column (out of six)
shows that if $x$ has torus part zero and central cocharacter $z_\ast$,
then $c^{\alpha}(x)$ also has torus part~0.  

The case of complex cross actions is similar, using once more 
the calculation of normalized torus parts in Table~2 of \cite[p. 65]{TwistedParameters}:
see the first line and fourth column there,
again using the fact that $x_{b, \tau}$ has torus part zero and central cocharacter $z_\ast$.
\end{proof}

\subsubsection{Cayley transforms and characters of component groups} \label{sec:discussion_cayley} Now let $\alpha$ be a $\tau$-imaginary root, and let $\tau'=s_\alpha\tau$. Recall  $\tau'$ is an involution and $\alpha$ is $\tau'$-real.
Using Proposition \ref{prop:bij_char_KGB_4} we want to understand the Cayley transform $c^\alpha\colon \X_{\tau}[\alpha]\to \X_{\tau'}$ in 
terms of  $\Char(\widetilde{\mathbb{S}}_{\ch{\tau}})$ and  $\Char(\widetilde{\mathbb{S}}_{\ch{\tau'}})$.
Recall we  have a canonical bijection $\mathcal{D}_\tau \colon \X_\tau\to \Char(\Schtautilde)$,  and  $\Char(\Schtautilde)$ identifies with $\Ut_\tau$ (Section \ref{sec:isomorphism_torusgroups}); therefore we have a bijection, still denoted $\mathcal{D}_\tau$, between $\X_\tau$ and $U_\tau$. A key point is that under this bijection, the domain
$\X_\tau[\alpha]$ corresponds to the kernel of a certain character of $\Ut_\tau$.

\begin{lemm}\label{lemm:inv_rep}
\begin{enumerate}[(1)]
\item If $h$ is an element of $H^{-\tau}_Z$, then $\alpha(h) = \pm 1$. If $h \in A_\tau$ then $\alpha(h)=1$. Therefore $\alpha$ induces a character $\overline\alpha$ of $\Ut_\tau$, of order at most~$2$.
\item Let $\Vtau \subset \Ut_\tau$ be the kernel of $\overline\alpha$.   Then the image of $\X_\tau[\alpha]$ under $\mathcal{D}_\tau$ is $\Vtau$. 
\item  If $\overline\alpha$ is trivial then $\Vtau=\Ut_\tau$ and $\X_\tau[\alpha]=\X_\tau$.
\item  If $\overline\alpha$ is non-trivial then there exists $h_0\in H^{-\tau}_Z$ with $\alpha(h_0)=-1$,
$\Vtau$ has index $2$ in $\Ut_\tau$, and $\X_\tau=\X_\tau[\alpha]\cup h_0\X_\tau[\alpha]$ (disjoint union).
\end{enumerate}
\end{lemm}


\begin{proof} For~(1), by definition every element $h \in H^{-\tau}_Z$ satisfies $h \tau(h) \in Z(G)$, therefore $\alpha(h\tau(h))=1$; but $\alpha\circ \tau = \alpha$ since $\alpha$ is $\tau$-imaginary, so $\alpha(h)^2 = 1$ and $\alpha(h)=\pm 1$. If $h \in A_{\tau}$, i.e. if  $h= s\tau(s^{-1})$ with $s \in H$, then $\alpha(h)=\alpha(s) \alpha(s^{-1})=1$. 

Let us prove~(2). Let~$x_b$ be the basepoint of $\X_\tau$. The image of $\X_\tau[\alpha]$ under $\mathcal{D}_\tau$ consists of those elements~$u \in \Ut_\tau$ such that $u \cdot x_b$ remains in $\X_\tau[\alpha]$. Let~$h$ be an element of~$H^{-\tau}_Z$, and fix a representative $\xi \in \widetilde{\X}_{x_b}$. Since~$\alpha$ is $x_b$-noncompact, the root vector~$X_\alpha$  attached to our pinning $\mathcal{P}$ satisfies $\mathrm{int}(\xi)X_\alpha = -X_\alpha$. Then $h\xi$ is a representative of $x= u \cdot x_b$. Now $\mathrm{int}(h)(X_\alpha) = \alpha(h)X_\alpha$ by definition of the roots. Therefore~$\alpha$ is $x$-compact if $\alpha(h)=-1$, and $x$-noncompact if $\alpha(h)=1$. This means $u \cdot x_b$ remains in $\X_\tau[\alpha]$ if and only if $u$ is in~$\Vtau$, which proves (2).

Parts (3) and (4) are immediate from the preceding arguments. \end{proof}

It is useful  to  realize $\Vtau$ in a different way: as the fixed points of the reflection $s_\alpha$ acting on~$\Ut_\tau$.
Let $m_\alpha=\sigma_\alpha^2=\alpha^\vee(-1)\in H$.

\begin{lemm}\label{lemm:reflection}
The action of $s_\alpha$ on $H$ preserves $H^{-\tau}_Z$ and is trivial on $A_\tau$. Therefore it induces an action of $s_\alpha$ on $\Ut_\tau$. If $u\in \Ut_\tau$, then 
$$
s_\alpha(u)=\begin{cases}u&\text{if $u\in \Vtau$,}\\
m_\alpha\,{}u &\text{otherwise.}
\end{cases}
$$
\end{lemm}
\begin{proof} 
The action of $s_\alpha$ on $H$ is given by:
\begin{equation} \label{formula_s_alpha} s_\alpha(h)=h\,{} \alpha^{\vee}(\alpha(h^{-1})).\end{equation}
(To prove the formula, write~$h$ as $\exp(2i\pi X)$, with $X \in \ch{\mathfrak{h}}$, and calculate
\begin{align*} s_\alpha(h) &= \exp(2i\pi\,{}s_\alpha(X)) = \exp(2i\pi (X-\alpha(X)\alpha^\vee))\\
& = \exp(2i\pi X)\exp(-2i\pi\alpha(X)\alpha^\vee)=h\,{} \alpha^\vee(\exp(-2i\pi\alpha(X))) =
h \,{}\alpha^\vee(\alpha(h^{-1}))\end{align*}
 which gives the desired result.)
 If~$h$ is an element of $H^{-\tau}_Z$, then $\alpha(h) = \pm 1$, so~$s_\alpha(h) = h \alpha^{\vee}(\pm 1)$:
 therefore  $s_\alpha(h)=h$ if $\alpha(h)=1$, 
 and $s_\alpha(h)=m_\alpha\,{}h$ if $\alpha(h)=-1$. We have  $m_\alpha \in H^{-\tau}_Z$  since $\tau(m_\alpha)=m_\alpha$ and $m_\alpha^2 = 1$. Therefore the action of~$s_\alpha$ preserves~$H^{-\tau}_Z$, and is trivial on~$A_\tau$ since $\alpha=1$ on~$A_\tau$. This proves all assertions in the Lemma. 
\end{proof}

Now consider the unique map 
 \begin{equation} \label{lambda_alpha}
\lambda_\alpha\colon \Vtau \to \Ut_{\tau'}
\end{equation}
which implements the Cayley transform $\X_\tau[\alpha] \to \X_{\tau'}$, in the sense that the following diagram commutes: 
\begin{equation}\label{diagram_cayley_v2}
\xymatrix{
\Vtau \ar@{<-}[d]_{\mathcal D_{\tau}}^{\simeq} \ar@{->}[r]^{\lambda_\alpha}& \Ut_{\tau'} \ar@{<-}[d]_{\simeq}^{\mathcal D_{\tau'}}\\
\X_{\tau}[\alpha]\ar@{->>}[r]^{c^\alpha}&\X_{\tau'}.
}
\end{equation}

Clearly $\lambda_{\alpha}$ is surjective.
We will show that it can be described by  a very simple formula.

\begin{lemm}\label{lemm:inclusion_agroups} We have $A_\tau \subset A_{\tau'}$ and $m_\alpha \in A_{\tau'}$. 
\end{lemm}
\begin{proof} Recall $A_\tau$ is the identity component of $H^{-\tau}$. Therefore if $a \in A_\tau$, there is a continuous path $t \mapsto a(t)$ from $1$ to $a$. We just checked  $s_\alpha(a(t)) = a(t)$ for all $t$, therefore $a(t)$ is also in $H^{-\tau'}$ for all $t$, and belongs to the identity component $A_{\tau'}$. This proves the first assertion. For the second, set $s = \alpha^\vee(i)$; this is an element of~$\ch{H}$; by  \cite[(14.3)]{Algorithms} we have $s \tau(s)^{-1} =  s^2 = m_\alpha$. \end{proof}

\begin{lemm} \label{lemm:constraint_lambda} 
Let $\lambda_\alpha$ be the unique map $\Vtau \to \Ut_{\tau'}$ making Diagram~\eqref{diagram_cayley_v2} commute. Then  for $h \in H^{-\tau}_Z$ with $\alpha(h)=1$, 
\begin{equation}\label{lambda_on_Vtau} \lambda_\alpha(h A_\tau) = hA_{\tau'}.\end{equation}
The map $\lambda_{\alpha}$ is a surjective group homomorphism. 
\end{lemm}
To elaborate on the formula, suppose $u$ is an element of~$\Vtau$. By Lemma~\ref{lemm:reflection} we have~$u = h A_\tau$, where $h$ is an element of $H^{-\tau}_Z$ with $s_\alpha(h)=h$. Since the involution $\tau'$ is equal to $s_\alpha \tau$, we also have $h \in H^{-\tau'}_Z$: we have $\tau'(h) = s_\alpha \tau(h)=s_\alpha(zh^{-1}) = s_\alpha(z) h^{-1}$ for some $z \in Z(G)$, and $s_\alpha(z)=\mathrm{int}(\sigma_\alpha)(z)=z$, so $\tau'(h) h \in Z(G)$ as claimed. 
By Lemma~\ref{lemm:inclusion_agroups} the image of $h$ in $\Ut_{-s_\alpha\tau}$, i.e. the coset $u'=h A_{\tau'}$, depends only on $u=hA_\tau$ and not on the choice of representative~$h$. Formula~\eqref{lambda_on_Vtau} means we must have~$\lambda_\alpha(u)=u'$, therefore it does describe $\lambda_\alpha$ completely.
 
\smallskip

Let us now prove Lemma~\ref{lemm:constraint_lambda}. It is enough to check that the map defined by~\eqref{lambda_on_Vtau} makes Diagram \eqref{diagram_cayley_v2} commutative. Let $\lambda\colon \Vtau\to \Ut_{\tau'}$ be defined by $\lambda(hA_{\tau})=hA_{\tau'}$ for all $h \in H^{-\tau}_Z$ with $\alpha(h)=1$. Fix $y \in \X_\tau[\alpha]$ and write $y = v \cdot x_b$, where $x_b$ is the basepoint in $\X_\tau$ and $v \in \Vtau$, $v = hA_\tau$ with  $\alpha(h) = 1$. Let $\xi_{b}$ be the representative for $x_b$ defined in~\S\ref{sec:basepoint}; then $h \xi_b$ is a representative of $y$, and~$c^\alpha(y)$ is the image in $\X$ of $\sigma_\alpha h \xi_b = s_\alpha(h) (\sigma_\alpha \xi_b)$. Now $s_\alpha(h)=h$ by Lemma~\ref{lemm:reflection}, and $\sigma_\alpha \xi_b$ is a representative of the basepoint $x_{b,\tau'}$ by Lemma~\ref{basepoint}. This means $c^\alpha(y) = v' \cdot x_{b, \tau'}$, where $v' = h A_{\tau'}$. Therefore $\mathcal{D}_{\tau'}(c^\alpha(y)) = v' = \lambda(v) =\lambda(\mathcal{D}_\tau(y))$, which proves the commutativity of the diagram.  \qed
\bigskip

This gives a description of Cayley transforms in terms of the group $\Ut_\tau=H^{-\tau}_Z/A_\tau$. To convert this into an understanding in terms of component groups on the dual side, recall we have an explicit isomorphism $\Ut_\tau \simeq \Char(\widetilde{\mathbb{S}}_{\ch{\tau}})$ (Section~\ref{sec:isomorphism_torusgroups}). Let $\PiSpecial\subset \Char(\widetilde{\mathbb{S}}_{\ch{\tau}})$ be the subgroup corresponding to~$\Vtau$ under that isomorphism; by transport of structure from~$\lambda_\alpha$, we get a surjective homomorphism 
\begin{equation} \label{kappa_sub_alpha}
\lambda_\alpha\colon \PiSpecial \to \Char(\widetilde{\mathbb{S}}_{\ch{\tau}'})
\end{equation}
making the following diagram commutative: 
\begin{equation}\label{diagram_cayley_v1}
\xymatrix{
\PiSpecial\ar@{<-}[d]_{\mathcal D_{\tau}}^{\simeq}\ar@{->>}[r]^{\lambda_\alpha}&\Char(\wt{\mathbb S}_{\ch{\tau}'})\ar@{<-}[d]_{\simeq}^{\mathcal D_{\tau'}}\\
\X_{\tau}[\alpha]\ar@{->>}[r]^{c^\alpha}&\X_{\tau'}.
}
\end{equation}
Later on we will be interested in the dual of the top horizontal arrow, so define
\begin{equation} \label{tildeSnew} \tildeSnew = \text{Pontryagin dual of $\PiSpecial$.}\end{equation} 
Since $\PiSpecial$ is a subgroup of $\Char(\Schtautilde)$ of index $1$ or $2$,
this is a quotient of 
$\Schtautilde$ by a subgroup of order $1$ or $2$.
Here is a precise description.

\begin{lemm}\label{lemm:pispecial}
The kernel of the map $\Schtautilde\twoheadrightarrow \tildeSnew$ is generated by $\tilde m_\alpha$. Furthermore
\begin{equation}
\label{e:tildeSnew}
\tildeSnew\simeq \ch H^{\alg,\tau}/\langle (\ch H^{\alg,\tau})_0,\ch\tilde m_\alpha\rangle.
\end{equation}

\end{lemm}

\begin{proof}
By Lemma~\ref{lemm:reflection} the group $\Vtau$ is the subgroup of $s_\alpha$-invariants of~$\Ut_{\tau}$. Inspecting the construction of the isomorphism $\Ut_{\tau} \simeq \Char(\wt{\mathbb S}_{\ch{\tau}})$ in Section~\ref{sec:isomorphism_torusgroups}, we see that the isomorphism takes the action of $s_\alpha$ on $\Ut_{\tau}$ to the action on $\Char(\wt{\mathbb S}_{\ch{\tau}})$ inherited from the action of $s_{\alpha}$ on $\ch{H}^\alg$. Therefore~$\PiSpecial$ is the group of characters invariant under the latter action of $s_\alpha$. 
Arguing as in the proof of Lemma~\ref{lemm:reflection}, we find that for all~$u$ in $\wt{\mathbb S}_{\ch{\tau}}$, the element $s_\alpha(u)$ is equal to~$u$ or~$\ch{\overline{m}}_{\alpha}u$ (see also \cite[p.~156]{ABV}).
This proves that the elements of $\PiSpecial$  are the characters of~$\wt{\mathbb S}_{\ch{\tau}}$ trivial on~$\ch{\overline{m}}_{\alpha}$, and that the kernel of $\Schtautilde\twoheadrightarrow \tildeSnew$ is generated by~$\ch{\overline{m}}_{\alpha}$. The final assertion follows.
\end{proof}

Now taking the dual of $\lambda_\alpha:\PiSpecial\rightarrow\Char(\Schtautilde)$
gives an injective homomorphism
$$
\lambda^\alpha:\Schtautildeprime\hookrightarrow \tildeSnew.
$$
Here is a concrete description of $\lambda^\alpha$.

\begin{lemm}\label{lemm:technical_statements_dualside} Let $u$ be an element of~$\Schtautildeprime$.
\begin{enumerate}[(1)]
\item There exists an element~$h$ of~$\ch{H}^{\alg, \tau}$ such that $u=h(\ch H^{\alg,\ch\tau'}_0)$. 
\item Let $v$ be the image of $h$ in
$\tildeSnew\simeq \ch H^{\alg,\tau}/\langle (\ch H^{\alg,\tau})_0,\ch\tilde m_\alpha\rangle$ 
\textup{(}see \eqref{e:tildeSnew}\textup{).} Then~$v$ depends only on~$u$, not on the choice of~$h$, and   $\lambda^\alpha(u)=v$. 
\end{enumerate}
\end{lemm}

\begin{proof} For~(1), we begin without coverings and prove that every element of $(\ch{H}^{\ch{\tau}'})/(\ch{H}^{\ch{\tau}'})_0$ has an $s_\alpha$-invariant representative. 

Let us use~\eqref{formula_s_alpha} on the dual side: if we view $\alpha^\vee$ as a map $\ch{H}\to~\C$ and $\alpha$ as a map  $\C^\times \to \ch{H}$, then $s_\alpha(\ch{\eta})=\ch{\eta}\,{} \alpha(\alpha^\vee(\ch{\eta}))$ for all~$\ch{\eta}$ in~$\ch{H}$.
Now we begin with~$\ch{h}$ in~$\ch{H}^{\ch{\tau'}}$; since~$(\ch{H}^{\ch{\tau'}})_0$ comprises all elements of the form~$\ch{s}\ch{\tau'}(\ch{s})$ with $\ch{s} \in \ch{H}$,
we look for  $\ch{s}$ such that $\ch{h}\ch{s}\ch{\tau'}(\ch{s})$ is $s_\alpha$-invariant.
Using the formula for the action of $s_\alpha$, this is equivalent to
$\alpha(\alpha^\vee(\ch{h}))~\alpha(\alpha^\vee(\ch{s}))~\alpha(\alpha^\vee(\ch{\tau}'(\ch{s})))=1$, i.e. $\alpha(\alpha^\vee(\ch{s})^2) = \alpha(\alpha^\vee(\ch{h})^{-1})$ since $\alpha$ is $\ch{\tau}'$-imaginary.
We may choose $\ch{s}$ such that $\alpha^\vee(\ch{s})^2=\alpha^\vee(\ch{h})^{-1}$ since these are complex numbers,
and we have found an $s_\alpha$-invariant representative for the class of~$\ch{h}$ in  $(\ch{H}^{\ch{\tau}'})/(\ch{H}^{\ch{\tau}'})_0$.

Next we go over to coverings. Let~$\ch{\tilde{h}}$ be an element of $(\ch{H}^{\alg,\ch{\tau}'})$, let~$\ch{h}$ be its projection in~$\ch{H}^{\ch{\tau'}}$, and let~$\ch{a}$ be an element of $(\ch{H}^{\ch{\tau}'})_0$ such that~$\ch{h}\ch{a}$ is $s_\alpha$-invariant. Fix a continuous path $t \mapsto \ch{\eta}(t)$ in $(\ch{H}^{\ch{\tau}'})_0$ such that~$\ch{\eta}(0)=1$ and $\ch{\eta}(1) = \ch{a}$. Let $t \mapsto \ch{\tilde{\eta}}(t)$ be the unique lift of $\eta$ to $\ch{H}^{\alg}$ such that $\ch{\tilde{\eta}}(0)=1$, and let $\ch{\tilde{a}} = \ch{\tilde}{\eta}(1)$. Then $t \mapsto \ch{\tilde{h}}\ch{\tilde{\eta}}(t)$ is a continuous path in~$(\ch{H}^{\alg,\ch{\tau}'})$ between $\ch{\tilde{h}}$ and $\ch{\tilde{h}}\ch{\tilde{a}}$. The latter projects to~$\ch{h}\ch{a}$ which is in $\ch{H}^{\ch{\tau}}$, therefore it belongs to~$(\ch{H}^{\alg, \ch{\tau}})$; and it is a representative of the class of~$\ch{\tilde{h}}$ in~$\wt{\mathbb{S}}_{\ch{\tau'}}$. This concludes the proof of~(1). 

For the independence statement in~(2), it is enough to see that $(\ch{H}^{\alg, \ch{\tau}'})_0 \cap (\ch{H}^{\alg, \ch{\tau}})$ is contained in $\langle \ch{\tilde{m}}_\alpha,(\ch{H}^{\alg, \ch{\tau}})_0 \rangle$. To check this we may work in the group generated by~$\ch{H}^{\alg}$ and the image of the root subgroup morphism~$\mathrm{SL}(2,\C) \to \ch{G}^{\alg}$, which is locally isomorphic to $\mathrm{SL}(2)$; this reduces matters to a computation for~$\mathrm{SL}(2,\C)$ in the case~$\tau=1$, where the result is straightforward: see ~\cite[p.~202]{ABV}. The final  thing to check is that $\lambda^\alpha(u)=v$. This can be seen by carefully chasing definitions, beginning with the description of~$\lambda_\alpha$ in Lemma~\ref{lemm:constraint_lambda}, inserting the description of the isomorphism  $\Ut_\tau \simeq \Char(\widetilde{\mathbb{S}}_{\ch{\tau}})$ from Section~\ref{sec:isomorphism_torusgroups}, and using Pontryagin duality for group morphisms. We omit the details.
\end{proof}

\subsubsection{Cross actions and characters of component groups} \label{sec:cross_and_characters} Let~$\tau$ be an involution of~$H$ and let~$\alpha$ be a simple $\tau$-complex root. Suppose~$\tau' = s_\alpha \tau s_\alpha$. This situation is  much simpler than the case of Cayley transforms. First, the cross action $(s_\alpha \times~\,{})\colon\X_{\tau}\to\X_{\tau'}$ is a bijection. Second, the groups~$\ch{H}^{\ch{\tau}}$ and~$\ch{H}^{\ch{\tau}'}$ are conjugate, and after passage to coverings this induces an isomorphism of component groups $\widetilde{\mathbb{S}}_{\ch{\tau}}\simeq \widetilde{\mathbb{S}}_{\ch{\tau}'}$. To be more explicit about the second point, we can lift  the root subgroup morphism $\mathrm{SL}(2,\C)\to\ch{G}$ to a morphism $\mathrm{SL}(2,\C) \to \chGalg$ using the fact that~$\mathrm{SL}(2,\C)$ is simply connected (as in Section~\ref{sec:basepoint}), and define~${\tilde{\sigma}}_\alpha$ to be the image of~$\left(\begin{smallmatrix} 0 & 1\\-1 & 0\end{smallmatrix}\right)$. Then conjugation by~$\tilde{\sigma}_\alpha$ takes $\ch{H}^{\alg,\ch{\tau}}$ to~$\ch{H}^{\alg, \ch{\tau}'}$, and takes the identity component to the identity component; therefore it induces an isomorphism $\mathrm{int}(\tilde{\sigma}_\alpha)\colon  \widetilde{\mathbb{S}}_{\ch{\tau}}\to \widetilde{\mathbb{S}}_{\ch{\tau}'}$. Inspecting the definition of the bijections~$\mathcal{D}_{\tau}$, $\mathcal{D}_{\tau'}$ in Section~\ref{prop:bij_char_KGB_4}, and arguing as in the Cayley case (especially the proof of Lemma~\ref{lemm:constraint_lambda} using Lemma~\ref{basepoint}), we find that the following diagram commutes:
 \begin{equation}
\label{diagram_cross_v1}
\xymatrix{
\Char(\widetilde{\mathbb{S}}_{\ch{\tau}})\ar@{<->}[r]^{\mathrm{int}(\tilde{\sigma}_\alpha)}\ar@{<->}[d]_{\mathcal{D}_{\tau}}&\Char(\widetilde{\mathbb{S}}_{\ch{\tau}'})\ar@{<->}[d]^{\mathcal{D}_{\tau'}}\\
\X_{\tau}\ar@{<->}[r]^{(s_\alpha \times~\,{}) }&\X_{\tau'}.
}
\end{equation}
This gives an easier analogue of~\eqref{diagram_cayley_v1} for the case of complex cross actions.


\section{Description of the local Langlands correspondence} \label{sec:local_langlands}

\subsection{Atlas Parameters}
\label{s:atlasparam}
\label{sec:atlas_parameters}

\begin{defi}
\label{d:atlasparam}
  An \texttt{atlas} parameter is a pair $(x,\Lambda)$  satisfying the following conditions.
  First of all $x\in \X$ is a \texttt{KGB} element.
  Let $\tau=\int(x)\in \Aut(H)$, so the real torus $H(\R,\tau)$ is defined. Then $\Lambda$ is a genuine character of
  of the $\rho$-cover   $\widetilde{H}(\R,\tau)_\rho$ of $H(\R,\tau)$. Furthermore
  \begin{subequations}
    \label{e:dom}
  \renewcommand{\theequation}{\theparentequation)(\alph{equation}}
  \begin{equation}
        \label{e:doma}
  \langle d\Lambda,\ch\alpha\rangle \ge 0\quad\text{for all positive, $\tau$-imaginary roots}
\end{equation}
and
  \begin{equation}
            \label{e:domb}
  \langle d\Lambda,\ch\alpha\rangle> 0\quad\text{for all positive, $x$-compact roots.}
\end{equation}
We say $(x,\Lambda)$ is {\it final}  if 
\begin{equation}
\label{e:domc}
 \langle d\Lambda,\ch\alpha\rangle=0,\text{ for } \alpha \,\,\tau\text{-real}\Rightarrow\alpha\text{  satisfies the parity condition of \cite[Theorem 6.3(5)]{HermitianFormsSMF}.}
\end{equation}
\end{subequations}
\end{defi}

Attached to an \texttt{atlas} parameter $(x,\Lambda)$ is a  representation defined as follows.
First of all let~$G(\R)$ be the real form of $G$ defined by $\theta=\int(\xi)$ where $\xi\in \X_x$.

First we assume that $H(\R,\tau)$ is a relatively compact Cartan subgroup of $G(\R)$.
Assume
\begin{equation}
  \label{e:strict}
  \langle d\Lambda,\ch\alpha\rangle>0\quad\text{for all $\alpha\in \Delta^+$}.
\end{equation}
We define $\pi(\xi,\Lambda)$ to be the unique  discrete series representation of of $G(\R)$ whose character formula on the regular elements of $H(\R,\tau)$ is:

\begin{equation} \label{character_formula}
 \Theta_{\pi(\xi,\Lambda)}(g)=
(-1)^qD(\Delta^+,\tilde g)\sum_{w} \text{sgn}(w)(w\Lambda)(\tilde g)
\end{equation}

Here
$D(\Delta^+,\tilde g)$ is the Weyl denominator function
of \cite[Definition 13.5]{ComputingGlobalCharacters}, $q\in \mathbb Z$ is defined in \cite[Eq.~(2.8)]{ComputingGlobalCharacters},
and the sum is over the Weyl group $W(G(\R),H(\R,\tau))$. Finally $\tilde g$ is an inverse image of $g$ in  $\wt g\in \widetilde{H}(\R,\tau)_\rho$;
the right hand side factors to $H(\R,\tau)$.

If \eqref{e:strict} doesn't hold then $(\xi,\Lambda)$ defines a limit of discrete series representation, obtained by the Zuckerman translation
principle from a discrete series representation. See \cite[Definition 3.1]{ComputingGlobalCharacters}.
This representation is nonzero thanks to  \eqref{e:domb}.

We now drop the assumption that $H(\R,\tau)$ is relatively compact.
Let $A$ be the identity component of $H^{-\tau}$. The centralizer of $A$
is defined over $\R$, it can be written
$M(\R)A(\R)$, and $H(\R,\tau)$ is a relatively compact Cartan subgroup of $M(\R)A(\R)$.
Let $\Psi_r$ be the set of positive $\tau$-real roots and let $\rho_r=\frac12\sum_{\alpha\in\Psi_r}\alpha$.
Define $\Psi_i,\rho_i$ similarly using the $\tau$-imaginary roots.

Define a genuine character
~$\gamma$ of the $\rho_r$-cover of $H(\R,\tau)$:
\begin{equation}
  \label{e:gamma}
  \gamma(\tilde g)=\rho_r(\tilde g)/|\rho_r(\tilde g)|.
\end{equation}


Then it makes sense to define $\Lambda\otimes\gamma$, a genuine character of the $\rho_i$-cover of $H(\R,\tau)$ (see~\mbox{\cite[(4.4)]{ComputingGlobalCharacters}}).

Then $(x,\Lambda\otimes\gamma)$ is an \texttt{atlas} parameter for
$M(\R)A(\R)$, so it defines a nonzero, limit of discrete series
representation $\sigma$ of $M(\R)A(\R)$ by the preceding discusion.
Let $P(\R)=M(\R)A(\R)N(\R)$ be a parabolic subgroup of $G(\R)$ let
$I(\xi,\Lambda)$ be the induced representation
$\mathrm{Ind}_{P(\R)}^{G(\R)}(\sigma)$.  The composition factors of
$I(\xi,\Lambda)$ are independent of the choice of parabolic subgroup.
It is convenient to choose $P(\R)$ satisfying $\mathrm{Re}\langle d\Lambda,\ch\alpha\rangle\ge 0$ for 
all real roots $\alpha$. 
With this choice we define 
$\pi(\xi,\Lambda)$ to be the {\it cosocle}  of $I(\xi,\Lambda)$, i.e.  the direct sum of the irreducible quotients of $I(\xi,\Lambda)$.

\begin{prop}
The lowest $K$-types of $I(\xi,\Lambda)$ are all contained in $\pi(\xi,\Lambda)$. Furthermore, if $(x,\Lambda)$ is final, then $\pi(\xi,\Lambda)$ is irreducible. 
\end{prop}

Using the language of representations of strong real forms (Section \ref{sec:def_representations}), we write $I(x, \Lambda)$ and $\pi(x, \Lambda)$ for  $[\xi, I(\xi, \Lambda)]$ and $[\xi, \pi(\xi, \Lambda)]$ respectively. 

Here are two types of equivalence of parameters we will need.
Suppose $(x,\Lambda)$ is an \texttt{atlas} parameter and $w\in W$. Let $w\Lambda$ be the character of the $\rho$-cover
of the real form of $H$ defined by $w\times x$ (cf. Section \ref{s:crosscayleykgb}).
If $x \in \X_\tau$ and $w$ preserves the positive $\tau$-imaginary roots, then \eqref{e:dom} holds, and $(w\times x,w\Lambda)$ is a valid \texttt{atlas} parameter.

\begin{lemma}
  \label{l:independent}
Suppose $(x,\Lambda)$ is an \texttt{\textup{atlas}} parameter, with $x \in \X_\tau$.
Let $\Psi_i$ be the set of positive $\tau$-imaginary roots.
\begin{enumerate}
\item  Suppose $w\in W$ satisfies
$w\Psi_i=\Psi_i$. Let $\gamma(w)$ be the character $e^{w\rho_r-\rho_r}$ of $H(\R,\tau)$.
Then $\pi(x,\Lambda)$ is isomorphic to $\pi(w\times x,w\Lambda\otimes \gamma(w))$.
\item Suppose $\alpha$ is $\tau$-complex, and simple for $\Psiint(\Lambda)$.
  Then $\pi(s_\alpha\times x,s_\alpha \Lambda)$ is isomorphic to~$\pi(x,\Lambda)$.
\end{enumerate}
\end{lemma}
For (1) see \cite[Lemma 8.24]{AV1}, and (2) is an immediate consequence of (1).

In case (2) we write $s_\alpha(x,\Lambda)=(s_\alpha\times x,s_\alpha\Lambda)$.

\subsection{Refined Local Langlands correspondence}

\label{sec:pin_down_rep}

Now suppose $(\phi,\chi)$ is a complete Langlands parameter. We define the corresponding representation $\pi(\phi,\chi)$ as follows.

After conjugating by $\ch G$ we may assume $\phi$ is in standard form
(Definition \ref{d:stdform}). 

Set
$\ch\tau=\ch\tau(\phi)=\int(\phi(j))$ and let $\tau=\tau(\phi)$ be the
dual automorphism of $H$.  This defines the real Cartan subgroup
$H(\R,\tau)$.

Let $y_{b,\ch\tau}$ be the basepoint in $\ch\X_{\ch\tau}$ (Section \ref{sec:basepoint}). Then \mbox{$\egr{H}=\langle \ch H,y_{b,\ch\tau}\rangle$}  is an $E$-group for~$H$ (Section~\ref{sec:duality_tori}) corresponding to the double cover~$\widetilde{H}(\R,\tau)_{\rho}$ of  $H(\mathbb{R}, \tau)$, and
the image of $\phi$ is contained in~$\egr{H}$. 
Now apply Section  \ref{sec:duality_tori} to the parameter~$\phi\colon \mathbf{W}_{\R}\to\egr{H}$ to define a genuine character~$\Lambda_\phi$
of the double cover $\widetilde{H}(\R,\tau)_{\rho}$.

Next, by Proposition~\ref{prop:bij_char_KGB_4}, the character $\chi$ determines a $\texttt{KGB}$ element~$x$ in $\mathcal{X}_{\tau}$.
Then $(x,\Lambda_\phi)$ is an \texttt{atlas} parameter in the sense of Definition \ref{d:atlasparam}. It is final (see Remark~\ref{remark:noncompact_singular} and \cite[Proposition 13.12]{ABV}).
Finally, given $\xi\in \X_x$ we define the $(\g,K_\xi)$-module~$\pi(\xi,\Lambda_\phi)$ using Section~\ref{s:atlasparam}.
We define $I(\phi, \chi)$ and $\pi(\phi,\chi)$ to be the equivalence classes~$[\xi, I(\xi, \Lambda_\phi)]$ and~$[\xi, \pi(\xi,\Lambda_\phi)]$ respectively. These are independent
of the choice of $\xi\in \X_x$.

\begin{theo}[Refined local Langlands correspondence] \label{thm:complete_langlands}
Let $G$ be a complex connected reductive group,
equipped with an inner class $\gamma$ of real forms
and a pinning $\mathcal{P}$.
Use those to define notions of complete, final Langlands parameters,
as in \S \ref{sec:def_complete_parameters},
and representations of strong real forms of $G$
\textup{(}in the  inner class $\gamma$\textup{)}, as in \S \ref{sec:def_representations}.
\begin{enumerate}[(1)]
\item Suppose $(\phi,\chi)$ is a complete Langlands parameter. Then
  $\pi(\phi, \chi)$
depends only on the equivalence class of $(\phi, \chi)$.
\item The correspondence $(\phi, \chi) \rightsquigarrow \pi(\phi, \chi)$
induces a bijection between equivalence classes of  complete, final Langlands parameters
and equivalence classes of irreducible representations of strong real forms of $G$.
\end{enumerate}
\end{theo}

This follows from \cite{Algorithms}  and \cite{ComputingGlobalCharacters}. We discuss a few details which are
important for our application.


\subsubsection{Independence of choices}  \label{sec:independence}
Suppose $(\phi,\chi)$ is a complete Langlands parameter.
Our construction of $\pi(\phi, \chi)$
began by picking a conjugate of $\phi$ in standard form. We need to show that the equivalence class of $\pi(\phi,\chi)$
is independent of this choice.

So assume $\phi$ is in standard form.
We need to show that replacing $(\phi,\chi)$ with $(w\phi,w\chi)$,
where $w\phi$ is also in standard form, does not change the equivalence class of $\pi(\phi,\chi)$. 
Let $\tau=\tau(\phi)$. The assumption of standard form implies that $w$ takes the positive $\tau$-imaginary roots to themselves. 

Write $(x,\Lambda)$ for the \texttt{atlas} parameter attached to $(\phi,\chi)$ in Section \ref{sec:pin_down_rep}.
By the discussion in Section \ref{sec:basepoint}, $w$ takes the basepoint in the fiber of $\phi(j)$ 
to the basepoint in the fiber of $w\phi(j)$. Consequently, examining the construction
we see that $(w\phi,w\chi)$ then defines the \texttt{atlas} parameter $(w \times x, w\Lambda\otimes \gamma(w))$. 
Therefore the result follows from  Lemma \ref{l:independent}.

\subsubsection{Distinction inside $L$-packets} \label{sec:L_packets}
Given an $L$-homomorphism $\phi$, we define the (`large') $L$-packet attached to $\phi$ to be
\[ \Pi(\phi)= \left\{ \, \pi(\phi, \chi) \ : \ \chi \in \Char(\widetilde{\mathbb{S}}_\phi)\,\right\}.\]
There is no redundancy in that description of the large $L$-packet:

\begin{lemm}
The map $\chi \mapsto \pi(\phi, \chi)$ is a bijection between $\Char(\widetilde{\mathbb{S}}_\phi)$ and $\Pi(\phi)$.
\end{lemm}

\begin{proof} The only thing to prove is the injectivity. The argument generalizes \cite[\S 5]{DiscreteSeriesSigns}.
We may assume $\phi$ is in standard form and set~$\tau=\tau(\phi)$.
Suppose we are given $\chi, \chi'$ in $\Char(\widetilde{\mathbb{S}}_\phi)$.
Let $x, x'$ be the  \texttt{KGB} elements in $\mathcal{X}_\tau$
that correspond to $\chi, \chi'$ under Proposition \ref{prop:bij_char_KGB_4},
and let $\xi, \xi'$ be representatives of $x,x'$ in $\widetilde{\mathcal{X}}$.
Assume $\pi(\phi, \chi)=\pi(\phi, \chi')$;
then $x$ and $x'$ must be $G$-conjugate. Since the Weyl group $W_{i, \tau}$ acts transitively on $\mathcal{X}_{\tau}[x]$
(see \S \ref{sec:transitivity_action_srf}),
there must exist $w \in W_{i, \tau}$ such that $\xi'=w\xi$.
Then $\pi(\phi, \chi)=\pi(\phi, \chi')$ means
$[\xi, \pi(\xi, \Lambda_\phi)] = [\xi', \pi(\xi', {\Lambda_\phi})] =  [\xi, \pi(\chi, {(w^{-1}\Lambda_\phi})]$.
Now, formula \eqref{character_formula} is invariant under the real Weyl group
$W(G(\mathbb{R}, \xi), H(\mathbb{R}, \tau))$.
In fact  $[\xi, \pi(\chi, (w^{-1}\Lambda_\phi)] = [\xi,  \pi(\chi, {\Lambda_\phi})]$
if and only if $w\in W(G(\mathbb{R}, \xi), H(\mathbb{R}, \tau))$:
see \cite[Definitions 5.7, 5.9 and Lemma 13.4]{ComputingGlobalCharacters}.
But this finite group is also the stabilizer of $x$ for the action of~$W_{i, \tau}$ on~$\mathcal{X}_{\tau}[x]$,
see \S\ref{sec:transitivity_action_srf}.
Thus if $\pi(\phi, \chi)=\pi(\phi, \chi')$, then necessarily $x'=wx = x$, and~$\chi=\chi'$. \end{proof}

\subsubsection{Disjointness of $L$-packets} \label{sec:disjointness}
In order to prove that the map taking~$(\phi, \chi)$ to $\pi(\phi, \chi)$ induces an injection on conjugacy classes,
it is therefore enough to see that the $L$-packets $\Pi(\phi)$ and $\Pi(\phi')$ are disjoint
if $\phi$, $\phi'$ are nonconjugate $L$-homomorphisms.
For this it is useful to relate these packets with the original $L$-packets of  \cite{Borel, Contragredient}.

The definition of $L$-packets according to Langlands involves a fixed real form.
Thus, fix a Cartan involution~$\theta$ of~$G$ in the given inner class, and set $K=G^\theta$. Replace~$\phi$ by a conjugate in standard form; given $\chi \in  \Char(\widetilde{\mathbb{S}}_\phi)$, let~$x_\chi$ be the corresponding \texttt{KGB} element. Define
\[
\Pi(\phi, \theta)= \left\{\pi(\phi, \chi) :  \chi \in \Char(\widetilde{\mathbb{S}}_\phi)
\text{ and $x_\chi \in \mathcal{X}$ has a representative $\xi \in \widetilde{\mathcal{X}}$ such that }
\theta_{\xi} = \theta\right\}.
\]
This can be viewed as a set of irreducible $(\mathfrak{g}, K)$-modules.

To compare with the classical notion of $L$-packet for the single real form $G(\mathbb{R}, \theta)$, write
\[ \Pi_{\mathrm{classical}}(\phi, \theta)\]
for the set of irreducible $(\mathfrak{g}, K)$-modules
which Langlands defined in \cite{Langlands_CIRRAG}. See \cite{Borel} and  \cite[\S 6]{Contragredient}.

Given an involution $\theta$, the $L$-packets $\Pi_{\mathrm{classical}}(\phi, \theta)$ and  $\Pi_{\mathrm{classical}}(\phi', \theta)$
are disjoint as soon as $\phi$ and $\phi'$ are not conjugate.
Thus the disjointness of our packets $\Pi(\phi)$ and $\Pi(\phi')$ follows from:

\begin{lemm} \label{lemm:packets}
The packet $\Pi(\phi, \theta)$ is contained in $\Pi_{\mathrm{classical}}(\phi, \theta)$.
\end{lemm}

In fact these two packets are equal, but we will not prove it.

\begin{proof}[Proof of Lemma \ref{lemm:packets}]
We may assume $\phi$ is in standard form.
Consider an element of the $L$-packet~$\Pi(\phi, \theta)$;
write it as $[\xi, X]$, where $\xi \in \widetilde{\mathcal{X}}$ is a representative of $x_\chi$
such that $\theta_\chi=\theta$.
Let  $\phantom{}^d{S} \subset \LG$ be the subgroup of $\LG$ generated by $\ch{\delta}$ and the centralizer of the identity component of $H^{-\tau(\phi)}$, as in Section~\ref{s:atlasparam}. 
Then $\phantom{}^d S$ can be viewed as the $L$-group
for a $\theta$-stable Levi subgroup $S(\mathbb{R}, \theta)$ of $G(\mathbb{R}, \theta)$,
and we can view $\phi$ as a parameter $\mathbf{W}_{\mathbb{R}}\to\phantom{}^d{S}$.
In fact $S(\R)=S(\mathbb{R}, \theta)$ is the Levi subgroup $M(\mathbb{R})A(\R)$
used in \S\ref{s:atlasparam},
and the limit of discrete series representation~$\sigma$ used there
must be in the $L$-packet for $S(\mathbb{R}, \theta)$ attached to~$\phi$.
Now consider the standard Levi subgroup~$\phantom{}^dM$
constructed in \cite[Definition 6.15]{Contragredient}.
Then~$\phantom{}^dM \subset \phantom{}^d S$, the parameter~$\phi$ factors through $\phantom{}^dM$, and the $L$-packet for $S(\R)$ attached to~$\phi$ is obtained by parabolic induction from a discrete series $L$-packet
for a Levi subgroup of $S(\mathbb{R})$ with $L$-group~$\phantom{}^d M$.
By the compatibility with parabolic induction of~\cite[Section 6]{Contragredient},
double induction,
and the description of $\Pi_{\mathrm{classical}}(\phi, \theta)$ by parabolic induction from a discrete series $L$-packet for $M$,
 the $(\mathfrak{g}, K_{\xi})$-module~$\pi(\xi, \Lambda_\phi)$
must be in~$\Pi_{\mathrm{classical}}(\phi, \theta)$.
\end{proof}


\section{Tempiric parameters and representations}
\label{sec:ric_temp}
\label{sec:real_inf_char}
\label{sec:alternate_tempiric}

In the rest of the paper, we turn to the relationship
between the local Langlands correspondence and lowest $K$-types.
This is a preparatory section: it  collects basic remarks and definitions,
and recalls a theorem of Vogan  which is crucial for Section \ref{sec:r_groups}.

\subsection{Tempiric representations and $L$-packets}

Recall we discussed the notion of tempiric parameters and representations in Section \ref{sec:intro_real_tempered}.
An $L$-homomorphism $\phi$ is said to be tempiric if its restriction to $\R_+$ is trivial, and associated to any $L$-parameter  $\phi$ is
a tempiric one $\phi_c$.
A representation is tempiric if it is irreducible, tempered, and has real infinitesimal character.
Here is  some more detail on these definitions, and the relation between them.

First we consider representations, and see that the class of tempiric representations can be easily characterized in terms of parabolic induction from discrete series.

Let $\xi$ be a strong involution of $G$, let $\theta$ be the corresponding Cartan involution and let $K=G^\theta$ be the fixed points;
and let $G_\mathbb{R}$ be a real form of $G$ with maximal compact subgroup $K_\mathbb{R}=K \cap G_{\mathbb{R}}$.
Let $P_{\mathbb{R}}$ be a cuspidal parabolic subgroup of $G_\mathbb{R}$, with Levi factor $L_{\mathbb{R}}$,
and let $\sigma$ be a discrete series representation of $L_{\mathbb{R}}$
whose central character is \emph{trivial on the split component $A_{\mathbb{R}}$ of $L_{\mathbb{R}}$}.
Consider the induced representation $\mathrm{Ind}_{P_{\mathbb{R}}}^{G_{\mathbb{R}}}(\sigma)$;
this standard representation is tempered and splits into a finite sum of irreducible submodules. Each of these irreducible factors is then tempiric.

Conversely, if $(\xi, X)$ is a tempiric representation of a strong real form of $G$,
then $\pi$ is isomorphic with the underlying $(\mathfrak{g}, K)$-module of an irreducible factor
of some induced representation~$\mathrm{Ind}_{P_{\mathbb{R}}}^{G_{\mathbb{R}}}(\sigma)$,
where $P_{\mathbb{R}}\subset G_{\mathbb{R}}$ is a cuspidal parabolic subgroup
and $\sigma$ is a discrete series representation of $L_{\mathbb{R}}\subset P_{\mathbb{R}}$
whose central character is trivial on the split component~$A_{\mathbb{R}}$.

Next, we consider tempiric Langlands parameters. Let $\phi$ be an $L$-homomorphism aligned with~$\ch H$. Use Lemma~\ref{l:stdLhom} to write
\begin{equation}\label{canonical_form_phi_64}
\begin{cases}
 \phi(z) = z^{\lambda} \overline{z}^{\mathrm{Ad}(y)\lambda}  \quad z \in \mathbb{C}^\times \\
 \phi(j)  = e^{-i\pi\lambda} y
 \end{cases}
\end{equation}
where $\lambda \in\ch{\mathfrak{h}}$ and
$y \in \mathrm{Norm}_{\lgr{G}}(\ch{H}) \setminus\ch{H}$ satisfies $y^2 = e^{2i\pi\lambda}$.

Let $\ch{\tau}$ be the involution $\mathrm{int}(y)$ of $\ch{H}$,
and let $\tau$ be the dual involution of $H$ (see \S \ref{dual_involution}).
Decompose the Lie algebra $\mathfrak{h}$ as $\mathfrak{h}=\mathfrak{t}+\mathfrak{a}$,
where $\mathfrak{t}, \mathfrak{a}$ are the $+1$ and $-1$ eigenspaces of $\tau$.
\begin{lemm} \label{lemm:char_real_tempered}
The following conditions are equivalent:
\begin{enumerate}[(i)]
\item $\phi$ is tempiric,
\item $\ch\tau(\lambda)=-\lambda$,
\item $\lambda$ restricted to $\mathfrak a$ is trivial.
\end{enumerate}
\end{lemm}
\begin{proof}
If (i) is satisfied, then $\phi(r)=r^{\lambda+\ch{\tau}(\lambda)}$ must be $1$ for all $r>0$,
and thus $\lambda + \ch{\tau}(\lambda)$ must be zero, therefore (i) implies (ii).
Conversely, if (ii) is satisfied, then
$\phi(z) = z^{\lambda} \overline{z}^{-\lambda} = ({z}/{\overline{z}})^{\lambda}$ for $z \in \mathbb{C}^\times$,
and $\phi$ is trivial on $\mathbb{R}^+$; thus (i) and (ii) are equivalent.
The equivalence of (ii) and (iii) is immediate from the relationship between $\ch{\tau}$ and $\tau$ explained in \S \ref{dual_involution}.
\end{proof}

Here is the relation between tempiric parameters and representations.

\begin{prop}\label{characterization_real_tempered_parameters}
Let $\phi\colon \mathbf{W}_{\mathbb{R}}\to\lgr{G}$ be an $L$-homomorphism. The following are equivalent:
 \begin{enumerate}[(i)]
 \item $\phi$ is tempiric,
   \item $\Pi(\phi)$ contains a tempiric representation,
\item all representations in $\Pi(\phi)$ are tempiric.
\end{enumerate}
\end{prop}

Thus it makes sense to refer to a tempiric $L$-packet.

\begin{proof}
After replacing $\phi$ by a conjugate, we may write it as in~\eqref{canonical_form_phi_64}.
View $\lambda \in \ch{\mathfrak{h}}$ as a linear form on $\mathfrak{h}$.

First suppose $\phi$ is trivial on $\mathbb{R}^+$;
then the linear form $\lambda$ is trivial on $\mathfrak{a}$  by Lemma \ref{lemm:char_real_tempered}.
Let~$\pi$ be a representation in the $L$-packet $\Pi(\phi)$; write it as $[\xi, X]$
where $\xi$ is a strong involution and~$X$ is a $(\mathfrak{g}, K_\xi)$-module,
and let $G(\mathbb{R}, \theta_\xi)$ be a real form corresponding to $\xi$ under \S \ref{sec:comments_realform}.
By Lemma \ref{lemm:packets}, we can view $X$ as an element of $\Pi_{\mathrm{classical}}(\phi, \theta_\xi)$.
Run through the constructions of  \cite[Definitions 6.15, 6.6 and 4.10]{Contragredient}
to write $X$ as an irreducible quotient of a representation induced from a discrete series representation
of a cuspidal Levi subgroup. Inspecting the duality for tori sketched in \S \ref{sec:description_duality},
and following \cite[\S 4 and \S 6]{Contragredient},
the discrete series in question must be trivial on the split part of the Levi subgroup.
By the preceding discussion, the $(\mathfrak{g}, K_\xi)$-module $X$ must be tempiric. Thus (i) implies~(ii).

Conversely, suppose all representations in $\Pi(\phi)$ are tempiric.
Pick one of them: fix a strong involution $\xi$
and let $X$ be a $(\mathfrak{g}, K_\xi)$-module in the $L$-packet $\Pi(\phi, \theta_\xi)$, so $[\xi, X] \in \Pi(\phi)$.
Write~$X$ as an irreducible quotient of a representation induced from a discrete series representation $X^\flat$
of a  Levi subgroup $L_\mathbb{R}$ of $G(\mathbb{R}, \theta_\xi)$.
We may assume $L_\mathbb{R}$ to be $\theta_\xi$-stable, with Cartan subgroup~$H(\mathbb{R}, \theta_\xi)$,
and we may assume $\lambda$ to be a representative of the infinitesimal character of $X^\flat$.
Since $X$ is tempiric,  $X^\flat$ must have central character trivial on the split part of $L_{\mathbb{R}}$,
which has Lie algebra~$\mathfrak{a}$; thus $\lambda$ must be trivial on $\mathfrak{a}$.
Applying Lemma \ref{lemm:char_real_tempered} again, we see that $\phi$ must be trivial on $\mathbb{R}^+$.
\end{proof}

\subsection{Lowest $K$-types of tempiric representations}
\label{s:lkt_tempiric}

The behavior of tempiric representations under restriction to maximal compact subgroups is remarkable.
The next result \cite[Theorem 11.9]{KHatHowe} can be traced back to \cite[Chap.~6]{VoganGreenBook}.
\begin{theo}[Vogan] \label{vogan_theorem_ktypes} Fix a strong involution $\xi$ of $G$.
\begin{enumerate}[(1)]
\item  Let $\pi$ be a $(\mathfrak{g}, K_{\xi})$-module. If $\pi$ is tempiric, then $\pi$ has a unique lowest $K_{\xi}$-type.
\item If $\pi, \pi'$ are inequivalent tempiric $(\mathfrak{g}, K_\xi)$-modules, then their lowest $K_\xi$-types  are distinct.
\item Every irreducible representation of $K_\xi$ is the lowest $K_\xi$-type of a tempiric $(\mathfrak{g}, K_{\xi})$-module.
  \item The map taking $\pi$ to its lowest $K$-type induces a bijection between the set of equivalence classes of tempiric $(\g,K_\xi)$-modules and the unitary dual $\widehat{K_\xi}$.
\end{enumerate}   \end{theo}
%

Fix representatives $\xi_i$, $i \in I$, of the equivalence classes of strong involutions of~$G$ in the given inner class.
Each of these determines a real form $G(\mathbb{R}, \theta_{\xi_i})$ of $G$,
and a maximal compact subgroup~$K_{\xi_i}$. Consider the disjoint union of their unitary duals:
\[ \widehat{K}_{\mathrm{all}}=\coprod \limits_{i\in I} \widehat{K}_{\xi_i}.\]
If $\phi$ is a tempiric $L$-homomorphism, then every element  $\pi \in \Pi(\phi)$ is equivalent to a pair~$(\xi_i, X)$ where $i \in I$ and $X$ is a tempiric $(\g, K_{\xi_i})$-module. The equivalence class of $X$ is determined by $\pi$. The map (4) in Theorem~\ref{vogan_theorem_ktypes} takes~$X$ to an element of~$\Khatall$. This defines a map $\mathcal{V}\colon \Pi(\phi) \to \widehat{K}_{\mathrm{all}}$, and we have the following easy consequence of Theorems~\ref{vogan_theorem_ktypes} and~\ref{thm:complete_langlands}.
\begin{prop}
Suppose $\phi$ is a tempiric $L$-homomorphism. 
Then the map $\mathcal{V}\colon \Pi(\phi) \to \widehat{K}_{\mathrm{all}}$ is injective.
As $\phi$ runs over the tempiric $L$-parameters, the sets $\mathcal{V}(\Pi(\phi))$ exhaust $\widehat{K}_{\mathrm{all}}$.
\end{prop}

\subsection{Lowest $K$-types and $L$-packets}

Suppose for the moment that $\phi$ is a \emph{tempered} $L$-homomorphism, and $\Pi(\phi)$ is the
corresponding tempered $L$-packet (of a single real form). 
Consider the restriction of $\phi$ to $\R_+$. We can rescale it by composing $\phi$ with the automorphism $(u, x) \to (u, \alpha x)$ of the Weil group $\WR = \WRc \times \R_+$. This yields a continuous family~$(\phi_{\alpha})_{\alpha\geq 0}$ of tempered $L$-homomorphisms and $L$-packets, which corresponds to rescaling the $\nu$-parameter of a family of induced representations.
It is well known that the restrictions to $K$ of the representations in these $L$-packets are independent of this parameter.
In particular this holds if we set the $\nu$-parameter to $0$, i.e. if we replace $\phi$ with $\phi_c$.
Now $\Pi(\phi_c)$ consists of tempiric representations, each with a distinct lowest $K$-type. Roughly speaking, this says that the lowest $K$-types of the representations $\pi_i$ should be found among the lowest $K$-types
of the tempiric representations in~$\Pi(\phi_c)$. 
Note that if $\phi$ is non-tempered, then the $K$-types of the representations in the packet do depend
on $\nu$.

Here is a precise statement which holds across all $L$-packets.
Let $\xi$ be a strong involution of~$G$;
set $K=K_\xi$ and consider a real form $G(\mathbb{R})=G(\mathbb{R}, \theta_\xi)$,
as in \S \ref{sec:comments_realform}.
Let $X$ be an irreducible $(\mathfrak{g}, K)$-module,
and consider the equivalence class $\pi=[\xi, X]$ of representations of strong real forms.
Write $\mu_1, \dots, \mu_r$ for the lowest $K$-types of $X$.
For each $i \in \{1, \dots, r\}$, let $Y_i$ be a tempiric $(\mathfrak{g}, K)$-module with lowest $K$-type $\mu_i$, and set $\varpi_i=[\xi, Y_i]$.

\begin{lemm} \label{lemm:packet_LKTs_phic}
If $\pi$ is in the $L$-packet $\Pi(\phi)$, then the representations $\varpi_i$ are all in~$\Pi(\phi_c)$.
\end{lemm}

\begin{proof}
Let $\lgr{M}$ be a Levi subgroup of $\lgr{G}$, dual to a Levi subgroup $M(\R) \subset G(\R)$,
such that~$\phi$ factors through $\lgr{M} \subset \lgr{G}$
and that the resulting parameter $\phi_M\colon \WR \to \lgr{M}$
defines a relative discrete series $L$-packet~$\Pi_M(\phi_M)$ for~$M(\R)$.

By definition of $\Pi(\phi)$, there exists a parabolic subgroup $P(\R) \subset G(\R)$ with Levi factor $M(\R)$,
and a relative discrete series representation $\sigma'$ in the $L$-packet~$\Pi_M(\phi_M)$,
such that the $(\mathfrak{g},K)$-module $X$ is an irreducible quotient of $\mathrm{Ind}_{P(\R)}^{G(\R)}(\sigma')$.
See \cite[\S 6.3]{Contragredient}.

Let $A$ be the maximal split subtorus in the center of $M$.
We may write $\sigma'$ as $\sigma \otimes \nu$ where~$\sigma$ is trivial on $A(\R)$
and $\nu$ is a character of $A(\R)$, viewed as an unramified character of $M(\R)$.
Thus~$\pi$ is an irreducible quotient of $\mathrm{Ind}_{P(\R)}^{G(\R)}(\sigma\otimes \nu)$.

Using \cite[Proposition~6.6.(2)]{HermitianFormsSMF},
we deduce that the lowest $K$-types of $X$ must all occur in the tempered representation $\mathrm{Ind}_{P(\R)}^{G(\R)}(\sigma)$;
thus the $Y_i$ must all occur there.

To  prove that the $Y_i$ all belong to the $L$-packet $\Pi(\phi_c)$,
it is now enough to check that $\sigma$ must belong to the $L$-packet $\Pi_M((\phi_M)_c)$.
See \cite[Definition 6.15]{Contragredient}.

Now the inclusion $A \subset M$ induces a surjection $\ch{M} \to \ch{A}$,
and we may push $\phi_M$ to a map $\psi\colon \WR \to \ch{A}$.
Because $A$ is split, $\phi_M(j)$ must act by the identity on $\ch{A}$.
Using this, it is easy to check that  $\psi(z)=\phi(\lvert z\rvert)$ for $z \in \C^\times$;
and inspecting infinitesimal characters, that~$\psi$  corresponds to the character~$\nu$ of~$A$
under duality for split tori.
By the compatibility of the local Langlands correspondence with unramified twists,
the representation $\sigma \otimes \nu^{-1}$ of~$M(\R)$ must be an element of the $L$-packet $\Pi_M(\phi_M \cdot \psi^{-1})$.
The latter homomorphism is equal to $(\phi_M)_c$, and of course $\sigma'\otimes \nu^{-1}=\sigma$.
This proves $\sigma \in \Pi_M((\phi_M)_c)$, and the Lemma.
\end{proof}

%
%
%
%

\section{Embedding of component groups: proof of Proposition~\ref{prop:injectivity}}\label{sec:waldspurger}

This section gives a direct proof of Proposition~\ref{prop:injectivity} which uses only elementary structure theory on the dual side, and depends only on the definitions of $L$-groups, coverings and component groups (Sections~\ref{sec:l_hom} and~\ref{sec:componentgroups}).

Let $\phi\colon \mathbf{W}_{\mathbb{R}}\to\lgr{G}$ be an $L$-homomorphism. Recall $\WR=\WRc\times \R_+$ (Section~\ref{sec:intro_real_tempered}) , and let~$\phi_c$ be the tempiric $L$-homomorphism attached to~$\phi$: 
\begin{equation*}
\phi_c = \begin{cases} \phi & \text{ on $\WRc$;} \\ 1 & \text{on $\R^+$.}\end{cases}
\end{equation*}

Let us begin with a version of Proposition~\ref{prop:injectivity} which does not use coverings on the dual side. Set
 \begin{equation*} H=\ch{G}_\phi=\Cent_{\ch G}(\phi(\WR)), \quad J=\ch{G}_{\phi_c}=\Cent_{\ch G}(\phi(\WRc)).\end{equation*}
 Then $\Sphi$ is the component group $H/H^0$, and $\Sphic$ is $J/J^0$. Furthermore $H \subset J$ and $H^0 \subset H \cap J^0$; therefore the natural map $H \to J/J^0=\mathbb{S}_\phi$ factors through $H^0$, giving a group homomorphism 
 \[ \iota\colon \mathbb{S}_{\phi} \to \mathbb{S}_{\phi_c}.\]
 \begin{lemm}\label{lemm:walds}
 The homomorphism $\iota$ is injective. 
 \end{lemm}
 \begin{proof}[Proof (suggested by J-L.~Waldspurger)] 
 The kernel of~$\iota$ is $(H \cap J^0)/H^0$, so we need to prove the inclusion $H \cap J^0 \subset H^0$. For this it is enough to show that $H \cap J^0$ is connected.  Since $\WR=\WRc\times \R^+$, we have
\[ H = \Cent_{\ch G}(\phi(\WR))=\Cent_{\ch G}(\phi(\WRc))\cap \Cent_{\ch G}(\phi(\R^+))=
J\cap \Cent_{\ch G}(\phi(\R^+)).\]
Since $\R_+$ is in the center of~$\WR$ we have $\phi(\R_+) \subset J$, therefore
\[ H = \Cent_J(\phi(\R^+))\]
and in fact  $\phi(\R_+) \subset J^0$, whence
$$
H\cap J^0=\Cent_{J^0}(\phi(\R^+)).
$$
To see that the right-hand side is connected, we point out that  $J^0$ is a connected complex reductive group. Indeed, 
$$
J=\Cent_{\ch G}(\phi(\WRc))=\Cent_{\ch G}(\phi(\C^*))^{\phi(j)}
$$
and  this is the the fixed points of an involutive automorphism of the reductive group $\Cent_{\ch G}(\phi(\C^*))$. Therefore~$J$ is reductive, and $J^0$ is connected and reductive. Now $\phi(\R^+)$ is a one-parameter subgroup of~$J^0$ consisting of semisimple elements. Let~$X$ be its infinitesimal generator: this is a semisimple element in the Lie algebra of~$J$ such that $\phi(t)=\exp(tX)$ for all $t \in \R_+$. Then $H\cap J^0=\Cent_{J^0}(X)$. This is a Levi subgroup of $J^0$, and therefore connected. \end{proof}

We now go over to coverings. Recall the groups $\tSphi$ and $\tSphic$ in Proposition~\ref{prop:injectivity} arise as component groups of the pullbacks $\widetilde{H}$, $\widetilde{J}$ of $H$ and $J$ to the algebraic covering $\chGalg$ of~$\ch{G}$. The latter is the projective limit of all finite coverings of~$\chGalg$, and we will deduce Proposition~\ref{prop:injectivity}  from an analogue of Lemma~\ref{lemm:walds} attached to any \emph{finite} covering.

Suppose~$\ch{G}_Q \to \ch{G}$ is a finite covering of~$\ch{G}$, and let $H_Q, J_Q \subset \ch{G}_Q$ be the preimages of~$H$ and $J$ respectively. Set $\mathbb{S}_{\phi,Q} = H_Q/(H_Q)^0$ and $\mathbb{S}_{\phi,Q} = J_Q/(J_Q)^0$. Then the inclusion $H_Q \hookrightarrow J_Q$ induces a group homomorphism 
 \[ \iota_Q\colon \mathbb{S}_{\phi,Q} \to \mathbb{S}_{\phi_c,Q}.\]
\begin{lemm}\label{lemm:finite_covering}
The homomorphism $\iota_Q$ is injective.
\end{lemm}
\begin{proof} 
As before we have to prove the inclusion $H_Q \cap (J_Q)^0 \subset (H_Q)^0$, and it is enough to prove that $H_Q \cap (J_Q)^0$ is connected. The groups~$J$ and~$J_Q$ have the same Lie algebra, therefore the element~$X$ considered in the proof of Lemma~\ref{lemm:walds} (the infinitesimal generator of $\phi_{|\R^+}$) may be viewed as an element of $\mathrm{Lie}(J_Q)$. As in the proof of Lemma~\ref{lemm:walds} we have
\[ {H_Q}\cap ({J}_Q)^0 = \mathrm{Cent}_{({J}_Q)^0}(X).\]
Now the group $({J}_Q)^0$ is a connected finite covering of~$J^0$. Every connected finite covering of a connected reductive complex group is still in that category; therefore $(J_Q)^0$ is connected reductive, and $\mathrm{Cent}_{({J}_Q)^0}(X)$ is a Levi subgroup of $(J_Q)^0$, hence connected.
\end{proof}

We turn to the proof of Proposition~\ref{prop:injectivity}. The map
\[ \widetilde{\iota}\colon\widetilde{\mathbb{S}}_{\phi} \to \widetilde{\mathbb{S}}_{\phi,c}\]
in the Proposition is induced by the inclusion $\widetilde{H} \subset \widetilde{J}$, where   $\widetilde{H}$ and $\widetilde{J}$ are the pullbacks of $H$ and~$J$ to $\ch{G}^{\mathrm{alg}}$.  In order to deduce Proposition~\ref{prop:injectivity} (the injectivity of~$\widetilde{\iota}$) from Lemma~\ref{lemm:finite_covering}, we observe that $\widetilde{\mathbb{S}}_{\phi}$ identifies with the projective limit of the groups $\mathbb{S}_{\phi,Q}$. 

For every finite covering~$\ch{G}_Q$ of~$\ch{G}$, the canonical map $\chGalg \to \ch{G}_Q$ restricts to a map $\widetilde{H} \to H_Q$, and induces a group homomorphism $\widetilde{\mathbb{S}}_{\phi} \to \mathbb{S}_{\phi,Q}$. Similarly there is a canonical homomorphism $\widetilde{\mathbb{S}}_{\phi} \to \mathbb{S}_{\phi,Q}$, and the following diagram commutes:
\begin{equation} \label{diagram_iotas}
\xymatrix{
\widetilde{\mathbb{S}}_{\phi}\ar@{->}[d] \ar@{->}[r]^{\widetilde{\iota}}& \widetilde{\mathbb{S}}_{\phi_c}\ar@{->}[d]   \\
{\mathbb{S}}_{\phi,Q}\ar[r]^{\iota_Q}&{\mathbb{S}}_{\phi_c,Q}.
}
\end{equation}
In the terminology of Bourbaki \cite[Chapter III, \S~7, n${}^\circ$2]{Ensembles},  the maps $\iota_Q$ constitute a projective system of group homomorphisms, and~\eqref{diagram_iotas} means $\widetilde{\iota}$ is the projective limit of the maps $\iota_Q$.  By the Corollary to Proposition~2 in \cite[loc. cit.]{Ensembles}, the injectivity of all maps $\iota_Q$ implies that of~$\widetilde{\iota}$.

\section{General Langlands parameters}
\label{sec:r_groups}
\label{sec:discussion_algo}
\label{sec:strategy_injectivity}
We can finally study the interplay between the Langlands correspondence and lowest $K$-types.
We begin with an $L$-homomorphism $\phi\colon \mathbf{W}_{\mathbb{R}}\to\lgr{G}$,
a character $\chi$ of $\widetilde{\mathbb{S}}_\phi$,
and want to study the lowest $K$-types of $\pi(\phi, \chi)$.

In this section, we implement the program of \S \ref{sec:program_rgroups}.
We have seen that the lowest K-types of~$\pi(\phi,\chi)$ are a subset
of the lowest $K$-types of the tempiric representations in the $L$-packet~$\Pi(\phi_c)$
obtained by making the parameter~$\phi$ trivial on~$\R_+$. We seek to explicitly describe this set of $K$-types,
or equivalently the corresponding set $\{Y_i\}$ of tempiric representations.
Each $Y_i$ is of the form $\pi(\phi_c,\chi_{c,i})$ for a character $\chi_{c,i}$ of $\widetilde{\mathbb{S}}_{\phi_c}$. The main point is to show that the $\chi_{c,i}$ which occur 
are precisely those mapping to $\chi_c$ via the restriction map
$\Char(\widetilde{\mathbb{S}}_{\phi_c}) \to \Char(\widetilde{\mathbb{S}}_{\phi})$
of Proposition~\ref{prop:injectivity}.

The precise relationship between $X$ and the~$Y_i$
goes back (at least) to~ \cite[Chapter~6]{VoganGreenBook}.  The formulation of the result that we shall use is based on the reformulation
in~\cite{KHatHowe}, which was motivated by, and used in, the {\tt
  atlas} software.
The resulting algorithm is inductive, and  is based on a series of \emph{Cayley transforms} and \emph{cross actions}. We defined those in Section~\ref{s:crosscayleykgb} for \texttt{KGB} elements. In order to describe the algorithm, and in order to use it to prove our results, we first
\begin{enumerate}[(a)]
\item state the definitions of Cayley transforms and cross actions that we will use for the \texttt{atlas} parameters of  Section~\ref{s:atlasparam};
\item give a translation in terms of $L$-homomorphisms, in order to connect the algorithmic computation of lowest $K$-types with our statement about $L$-homomorphisms.
\end{enumerate}

In Section~\ref{sec:cayley_7} we explain these operations on {\tt atlas}
parameters. In Section~\ref{s:step} we give the parallel construction of
Langlands parameters on the dual side, and explain the relation
between the two. In Section~\ref{sec:description_finalize} we put the pieces together, and prove 
Theorem~\ref{main_theorem_rgroups}. 

\subsection{Cross action and Cayley transforms for parameters}
\label{sec:cayley_7}

Suppose $\textup{\texttt{p}}=(x,\Lambda)$ is an \texttt{atlas} parameter (Section~\ref{s:atlasparam}), and $\alpha$ is a simple root for the integral system $\Psiint(\Lambda)$. (In the notation of Section~\ref{s:crosscayleyL}, this is $\Psiint(d\Lambda)$.)

The cross action of~$s_\alpha$ on \texttt{atlas} parameters is, up to translation of language, defined in \cite[Chapter 8]{VoganGreenBook}. A convenient reference in our terminology is~\cite{TwistedParameters}.
 It satisfies: $s_\alpha\times (x,\Lambda)=(s_\alpha\times x,\Lambda')$ where
$s_\alpha\times x$ is the cross action on {\tt KGB} defined in Section \ref{s:crosscayleykgb}, and $\Lambda'$ is another genuine character,
satisfying $d\Lambda'=s_\alpha(d\Lambda)$.

An important special case, and the only one we need, is:   if $\langle d\Lambda,\ch\alpha\rangle=0$ and $\alpha$ is a simple root of $\Psiint(\Lambda)$ which is complex with respect to $\theta_x= \theta_{x,H}$, 
then $s_\alpha\times \textup{\texttt{p}}$ is equivalent to $\textup{\texttt{p}}$. 

We shall use the cross action for complex roots.  If~$\alpha$ is a simple $\theta_x$-complex root, we say $\alpha$ is of type $\texttt{C}^+$ if~$\theta_{x}(\alpha)$ is positive, and of type~$\texttt{C}^-$ otherwise.

Now suppose $\alpha$ is a $\theta_{x}$-real root which does not satisfy the parity condition of \cite[Theorem~6.3(5)]{HermitianFormsSMF}. In that case, the Cayley transform $c_\alpha(\textup{\texttt{p}})$ is
defined. This is a set of $1$ or $2$ parameters, depending on whether the real Cayley transform $c_\alpha(x)$ is
single- or double-valued (see~\S\ref{s:crosscayleykgb}). In the latter case $c_\alpha(\textup{\texttt{p}})=\{(x',\Lambda'), (x'',\Lambda')\}$ where $c_\alpha(x)=\{x',x''\}$.
  Also $x''=s_\alpha\times x'$ and $s_\alpha\times (x',\Lambda')=(x'',\Lambda')$. 
See \cite[Section 8.3]{VoganGreenBook}, or \cite[Section 14]{Algorithms}.

Here are the properties which we need. These are all basic properties of Vogan duality \cite{VoganIC4}, and in 
this language can be read off from the tables in \cite{TwistedParameters}. When $\texttt{p}$ is an \texttt{atlas} parameter and $\phi$~is an $L$-homomorphism, we write $\texttt{p} \in \Pi(\phi)$ when all irreducible constituents of the representation attached to~$\texttt{p}$ (Section~\ref{s:atlasparam}) are in the $L$-packet $\Pi(\phi)$.

\begin{lemma}
  \label{l:cayleyscrosses}
  Suppose $\phi=\phi(\lambda,y)$ is an $L$-homomorphism in standard form. Let $\textup{\texttt{p}}=(x,\Lambda)$ be an \textup{\texttt{atlas}} parameter satisfying $\textup{\texttt{p}} \in \Pi(\phi)$, and  let $\alpha$ be a simple root of $\Psiint(\Lambda)$ satisfying~\mbox{$\langle d\Lambda, \ch{\alpha}\rangle = 0$}. Let~$\theta_x$ be the involution~$\theta_{x,H}$ of~$H$, and let $\ch{\theta}_y$ be the involution~$\mathrm{int}(y)$ of~$\ch{H}$.
  \begin{enumerate}
\item We have $s_\alpha\times \textup{\texttt{p}}\in \Pi(s_\alpha\times \phi)$.
\item Suppose $\alpha$ is $\theta_x$-real. Then it satisfies the parity condition if and only if $\alpha$ is $\ch{\theta}_y$-compact. 
\item Assume $\alpha$ in (2) does not satisfy the parity condition.  Then $c_\alpha(\textup{\texttt{p}})\subset \Pi(c_\alpha(\phi))$.
  \end{enumerate}
\end{lemma}


To prove 
Theorem~\ref{main_theorem_rgroups}, we will consider the lowest $K$-type algorithm of Section~\ref{sec:finalize}, and interpret the Cayley transforms and cross-actions steps there in terms of component groups~$\widetilde{\mathbb{S}}_\phi$. To achieve this, we will need to define operations on $L$-homomorphisms which mirror the operations of the lowest $K$-type algorithm. 

We first note the analogue, for $L$-homomorphisms, of deforming the continuous parameter.

\begin{lemma}
\label{l:deform}
Suppose $\phi(\lambda,y)$ is an $L$-homomorphism, and $\tau\in \h^*$ is fixed by $\int(y)$.
Let
$h=e^{\pi i\tau}\in \ch H$. Then $\phi'=\phi(\lambda+\tau,hy)$ is a valid $L$-homomorphism, satisfying $\phi'(j)=\phi(j)$.
\end{lemma}

\begin{proof}
  To see that $(\lambda+\tau,hy)$ defines a valid $L$-homomorphism, we only need to check that $(hy)^2=\exp(2\pi i(\lambda+\tau))$.
  By the assumption on $\tau$ we have $(hy)^2=h^2y^2=\exp(2\pi i\tau)\exp(2\pi i\lambda)$.
Also $\phi'(j)=\exp(-\pi i(\lambda+\tau))hy=\exp(-\pi i\lambda)y=\phi(j)$.
\end{proof}

Suppose $\phi$ is an $L$-homomorphism, and  $\textup{\texttt{p}}=(x,\Lambda)$ is an \texttt{atlas} parameter whose attached representation belongs to~$\Pi(\phi)$.
Write $\h=\t+\a$ for the Cartan
decomposition of $\h$ with respect to $\theta_x$.
Set 
$\nu=d\Lambda|_{\a}$. We want to deform $\nu$ to $0$, one root at a
time. For this we need a smaller root system $\Psires$.

Define $\Psi_q$ to be the set of roots orthogonal to all imaginary coroots \cite[Proposition 3.12]{VoganIC4}. 
The involution $\theta$ restricts to a quasisplit involution of $\Psi_q$, i.e. the corresponding real form of~$G$ is quasisplit, and 
$H$ is a maximally split Cartan subgroup. Let $\Psires$ be the restriction of the roots of $\Psi_q$ to $\a$.
This is a root system.

Suppose furthermore $\nu\ne 0$. 
Then we can find a simple root $\beta$ of $\Psires$ such that $\langle \nu,\ch\beta\rangle\ne 0$.
Let $\nu_\beta$ be the projection of $\nu$ on the $\beta$-root wall:
$$
\nu_\beta=\nu-\langle \nu,\ch\beta\rangle\omega_\beta
$$
where $\omega_\beta\in \a^*$ is the corresponding fundamental weight.
Deform $\Lambda$ to $\Lambda_\beta$ accordingly, i.e. let~$\Lambda_\beta$ be the genuine character which satisfies $(d\Lambda_{\beta})|_{\a^*}=d\Lambda|_{\a^*}-\langle\nu,\ch\beta\rangle\omega_\beta$ and coincides with $\Lambda$ on the compact part of the cover. 

Let $\textup{\texttt{p}}_\beta=(x,\Lambda_\beta)$. This is again an \texttt{atlas} parameter.

On the dual side, write $\phi=\phi(\lambda,y)$ for some $y$. Set $y_\beta=\exp(-\pi i\langle \nu,\ch\beta\rangle\omega_\beta)y$ and $\lambda_\beta=d\Lambda_\beta$,
and define $\phi_\beta=(\lambda_\beta,y_\beta)$ as in Lemma \ref{l:deform}.

We make similar definitions to dispense with the restriction of $\nu$ to $\mathfrak z$.
That is define, $\texttt{p}_0=(x,\Lambda_0)$ and $\phi_0=(\lambda_0,y_0)$ similarly, with the orthogonal complement of $\a\cap\mathfrak z$ playing the role of the kernel of $\beta$. 

By inspection of the definitions, we have: 

\begin{lemma}
\label{l:deform1}
Suppose $\phi$, $\textup{\texttt{p}}$ and $\beta$ are as above. Then $\textup{\texttt{p}}_\beta\in\Pi(\phi_\beta)$, and
$\textup{\texttt{p}}_0\in\Pi(\phi_0)$.
\end{lemma}

This leads to a special case of cross actions and Cayley transforms in the case of a single singular root.

\begin{lemma}
  \label{l:singular}
Let $\alpha$ be a simple root of $\Psiint(\Lambda)$ satisfying
$\langle d\Lambda,\ch\alpha\rangle=0$.
  
\noindent (1) Suppose $\alpha$ is $\ch{\theta}_y$-complex.
Then $\textup{\texttt{p}}=(x,\Lambda)$ is equivalent to $s_\alpha\times \textup{\texttt{p}}$, and $\phi$ is conjugate to~\mbox{$s_\alpha\times\phi$}.

\noindent (2) Suppose $\alpha$ is $\ch{\theta}_y$-noncompact imaginary. 
Then $\phi$ is conjugate to  $\phi(\lambda,\sigma_\alpha y)$, and~\mbox{$c_\alpha(\textup{\texttt{p}})\subset \Pi(c_\alpha(\phi))$}.
\end{lemma}

\begin{proof}
  For (1), choose $n$ representing $s_\alpha$. Then  
  \[ s_\alpha\times \phi(\lambda,y)=\phi(\lambda,nyn\inv)=\phi(\mathrm{Ad}(n)\lambda,nyn\inv)=n\phi(\lambda,y)n\inv.\]
  The statement about $\textup{\texttt{p}}$ follows from transport of structure.

  For (2), recall $c^\alpha(y)$ is $\ch G$-conjugate to $y$, by an element fixing $\lambda$  (see \S \ref{s:crosscayleykgb}). The conjugacy statement follows, and the second statement is already in Lemma   \ref{l:cayleyscrosses}.
\end{proof}

\subsection{The Inductive Step}
\label{s:step}

Now let us assume that $\phi$ is in standard form, and consider~$\phi_\beta$.  This may fail to be in standard form---the essential case is discussed in Section~\ref{sec:example_SL2}. More generally, it can happen that $\phi_\beta$ is not in standard form due to a
real root $\alpha$ which is not simple for~$\Psiint(\lambda_\beta)$. For this reason we need to adjust our operations on $L$-homomorphisms, using the following observation.

\begin{lemma}
Suppose $\phi=\phi(\lambda,y)$ is in standard form, let $\beta$ be a simple root of $\Psires$, and set $\phi_\beta=(\lambda_\beta,y_\beta)$ as above.
Then there exists $g_1\in \Cent_{\ch G}(\lambda_\beta)$ such that 
the $\ch{\theta}_{\int(g_1)y}$-imaginary roots are spanned by simple roots of $\Psiint(\lambda_\beta)$, and that we have the following alternative:  either
\begin{enumerate}
\item[\textup{(1)}] $\int(g_1)\phi_\beta$ is in standard form; or 
\item[\textup{(2)}] There is a simple root $\alpha$ of $\Psiint(\lambda_\beta)$, which is $\ch{\theta}_{\int(g_1)y}$-imaginary and noncompact,
and  $c_\alpha(\int(g_1)\phi_\beta)$ is in standard form.
\end{enumerate}
\end{lemma}

\begin{proof}
This follows from a repeated application of the  previous Lemma.

We say a simple root~$\alpha$ of~$\Psiint$ is {of type~$\textup{\texttt{C}}^-$} with respect to~$y$ if it is $\ch{\theta}_y$-complex and $\ch{\theta}_y(\alpha)$ is positive. (Notice that this is the opposite convention to that of Section~\ref{sec:cayley_7}, because we are working on the dual side.)

Apply the following procedure inductively.

Suppose $\alpha$ is a simple root of $\Psiint$. 
If  $\alpha$ is of type~$\texttt{C}^-$
with respect to $y$, replace $\phi$ with $s_\alpha\times
\phi=\phi(\lambda,s_\alpha\times y)$; which is conjugate to $\phi$ by
Lemma \ref{l:singular}.
Repeat this until there are no
simple roots of type $\texttt{C}^-$.
This gives an $L$-homomorphism $\phi'=\phi(\lambda,y')$ which is conjugate to $\phi$. 
By \cite[Lemma 8.6.2]{VoganGreenBook} the
$\ch{\theta}_{y'}$-imaginary roots are spanned by simple roots.
If these roots are all compact then $\phi'$ is in standard form.
Otherwise if $\alpha$ is simple and noncompact then 
$c_\alpha(\phi')$ is in standard form,
since 
the split rank of the most split Cartan in the centralizer can go up by at most $1$.
\end{proof}

Now suppose we are given an $L$-homomorphism $\phi$ in standard form and  $\beta$ is a simple root of~$\Psires$.
Define $\phi_{\beta}$ and $g_1$ by the Lemma, and define $\phi'$ by $\phi'=\int(g_1)\phi_\beta$ in case (1)
and $\phi'=c_\alpha(\int(g_1)\phi_\beta)$ in case (2).
Recall from \S\ref{s:crosscayleyL} that in case (2), the parameter $\int(g_1)\phi_\beta$ is conjugate to $c_\alpha(\int(g_1)\phi_\beta)$, 
say by an element $g_2$.  Set $g=g_1$ in case (1) and $g=g_2g_1$ in case (2). We have 
the following sequence of $L$-homomorphisms:
$$
\phi \rightarrow \phi_{\beta}\overset{\int(g)}\longrightarrow\phi'
$$
where $\phi$ is in standard form;  $\phi_{\beta}$ is a valid homomorphism but not necessarily in standard form;
and $\phi'$ is in standard form. 

We now define a canonical map 
$\widetilde{\mathbb S}_{\phi}\rightarrow \widetilde{\mathbb S}_{\phi'}$.
First of all $\phi(j)=\phi_\beta(j)$,  the image of $\phi_\beta$ is contained in the image of $\phi$, 
and this induces a map
$\iota_1 \colon \widetilde{\mathbb{S}}_{\phi} \to \widetilde{\mathbb{S}}_{\phi_\beta}$ 
as in Section \S \ref{sec:program_rgroups}.

Next, choose an inverse image $\tilde g$ of $g$ in $\chGalg$. Then $\int(\tilde g)$ induces an isomorphism \mbox{$\iota_{2,{\tilde{g}}} \colon \widetilde{\mathbb{S}}_{\phi_\beta} \to \widetilde{\mathbb{S}}_{\phi'}$}.
This isomorphism is independent of the choices of $g$ and $\tilde g$;
if we make different choices $g',\tilde g'$  then  ${\tilde{g}}'{\tilde{g}}^{-1}$ is in $\ch{G}_{\phi'}^{\mathrm{alg}}$;
the isomorphisms $\iota_{2, {\tilde{g}}}$ and $\iota_{2, {\tilde{g}}'}$ then differ by the corresponding inner automorphism of  $\widetilde{\mathbb{S}}_{\phi'}$,
which is trivial since this group is abelian.
Write $\iota_{2}\colon \widetilde{\mathbb{S}}_{\phi_\beta} \to \widetilde{\mathbb{S}}_{\phi'}$ for the common value, and
define $\iota \colon \widetilde{\mathbb{S}}_{\phi} \to \widetilde{\mathbb{S}}_{\phi'}$ to be the composition $\iota_2\circ \iota_1$:
\[ \iota\colon \widetilde{\mathbb{S}}_{\phi} \overset{\iota_1}{\longrightarrow} \widetilde{\mathbb{S}}_{\phi_\beta}\overset{\iota_2}{\longrightarrow}  \widetilde{\mathbb{S}}_{\phi'}.\]

Consider the involutions $\tau=\tau(\phi)$ and $\tau'=\tau(\phi')$ of $H$.  
Then $\alpha$ is $\tau'$-imaginary, and we may consider the Cayley transform $c^\alpha: \X_{\tau'}[\alpha]\rightarrow \X_{\tau}$.
Recall from Section~\ref{sec:dictionary_KGB_charcompgroup} we have canonical surjective maps
$p\colon \widetilde{\mathbb{S}}_{\ch{\tau}}\twoheadrightarrow \widetilde{\mathbb{S}}_{\phi}$ and $p'\colon \widetilde{\mathbb{S}}_{\ch{\tau'}} \twoheadrightarrow \widetilde{\mathbb{S}}_{\phi'}$.
Note that the root~$\alpha$ is $\ch{\tau}'$-real and $\phi'$-singular; therefore the element~$\ch{\overline{m}}_\alpha$ (defined in \S\ref{sec:S_phi_and_S_tau}) belongs to the kernel of $p'\colon \widetilde{\mathbb{S}}_{\ch{\tau'}} \twoheadrightarrow~\widetilde{\mathbb{S}}_{\phi'}$. In the notation of~\eqref{e:tildeSnew}, this means
 $p'$ factors to a map $ \overline{p}' \colon \widetilde{\mathbb{S}}_{\ch{\tau'}}^{\mathrm{quo}} \twoheadrightarrow \widetilde{\mathbb{S}}_{\phi'}$.

\begin{prop}
\label{prop:bij_char_KGB_param_7}
  \begin{enumerate}[(1)]
\item We have $\overline{p}' \circ \lambda^{\alpha} = \iota \circ p$; that is, the following diagram~commutes:
\begin{equation} \label{diagram_phi_tau}
\xymatrix{
\widetilde{\mathbb{S}}_{\ch{\tau}}\ar@{->>}[d]_{p}\ar@{^{(}->}[r]^{\lambda^\alpha}& \widetilde{\mathbb{S}}^{\mathrm{quo}}_{\ch{\tau'}}\ar@{->>}[d]^{\overline{p}'}   \\
\widetilde{\mathbb{S}}_{\phi}\ar[r]^\iota&\widetilde{\mathbb{S}}_{\phi'}.
}
\end{equation}
\item The map $\iota: \widetilde{\mathbb{S}}_{\phi} \to \widetilde{\mathbb{S}}_{\phi'}$ is injective.
\end{enumerate}
\end{prop}

We actually care about the dual of this diagram. Let $\rho$ be the dual of $\iota$.
\begin{coro}
\label{p:diagram} If $\E_{\phi}$, $\E_{\phi'}$ are the maps  \eqref{e:E},
 then the following diagram is commutative:
\begin{equation}
\label{diagram_cayley_final}
\xymatrix{
\Char(\widetilde{\mathbb{S}}_{\phi'})\ar@{->>}[r]^\rho\ar@{_{(}->}[d]_{\mathcal{E}_{\phi}}&\Char(\widetilde{\mathbb{S}}_{\phi})\ar@{_{(}->}[d]^{\mathcal{E}_{\phi}}\\
\X_{\tau'}[\alpha]\ar@{->>}[r]^{c^\alpha}&\X_{\tau}.
}
\end{equation}
\end{coro}

\begin{proof}[Proof of the Corollary, given the Proposition]
Let $\beta'\colon \Char(\widetilde{\mathbb{S}}_{\phi'})\hookrightarrow \Char(\widetilde{\mathbb{S}}_{\ch{\tau'}})$ be
the map~\eqref{inj_toruscomp} for~$\phi'$, and let~$\beta$ be the corresponding map for~$\phi$. 
By definition these are dual to the maps $p',p$ above. 
Let~$\PiSpecialprime$ be the group of characters of $\widetilde{\mathbb{S}}_{\ch{\tau}'}$ trivial on $\ch{\overline{m}_\alpha}$ (see the proof of Lemma~\ref{lemm:pispecial}).  By the discussion of $\phi'$-final characters in
 Section~\ref{sec:S_phi_and_S_tau},  the image of~$\beta'$ is contained in $\PiSpecialprime$, and the induced map $\Char(\widetilde{\mathbb{S}}_{\phi'})\hookrightarrow \PiSpecialprime$ is dual to the map~$\overline{p}'$ of Proposition~\ref{prop:bij_char_KGB_param}.
Therefore, in the diagram
$$
\xymatrix{
\Char(\widetilde{\mathbb{S}}_{\phi'})\ar[r]^\rho\ar@{_{(}->}[d]_{\beta'}&
\Char(\widetilde{\mathbb{S}}_{\phi})\ar@{_{(}->}[d]^{\beta}\\
\PiSpecialprime\ar@{<->}[d]_{\D'}\ar@{->>}[r]^{\lambda_\alpha}&
\Char(\wt{\mathbb S}_{\ch{\tau}})\ar@{<->}^{\D}[d]\\
\X_{\tau'}[\alpha]\ar@{->}[r]^{c^\alpha}&\X_{\tau},
}$$
the top square is dual to the diagram of Proposition~\ref{prop:bij_char_KGB_param}, so it is commutative. The bottom square is~\eqref{diagram_cayley_v1}, so it is commutative as well. By definition of~$\E_{\phi'}$ and~$\E_{\phi}$ 
the conclusion follows.
\end{proof}

\begin{proof}[Proof of the Proposition]
Write $\mathrm{proj}$ for the projection  $\chGalg \to \ch G$,  and fix~\mbox{${\tilde g}\in \mathrm{proj}\inv(g)$}.
We may attach to every~$\tilde u$ in $\ch{H}^{\ch\tau',\mathrm{alg}}$  the element
$\mathrm{int}({\tilde{g}\inv})(\tilde u)$; this is in~$\ch{G}^{\mathrm{alg}}_{\phi'}$.
Taking the image modulo the identity component  defines a map $p_{{\tilde{g}}} \colon \widetilde{\mathbb{S}}_{\ch\tau'}
 \to \widetilde{\mathbb{S}}_{\phi_\beta}$. By \cite[Lemma~12.10]{ABV} this map is surjective, and by construction $\iota_2 \circ p_{{\tilde{g}}} = p$ ; in particular, the element~$\ch{\overline{m}}_\alpha$ is in the kernel of $p_{{\tilde{g}}}$, which therefore induces a map 
$\overline{p}_{{\tilde{g}}}\colon\widetilde{\mathbb{S}}_{\ch\tau}^{\mathrm{quo}} \to \widetilde{\mathbb{S}}_{\phi_\beta}$.
To prove the first statement in the proposition, it is therefore enough to check the commutativity of the following diagram:
\begin{equation}
  \label{diagram_int}
\xymatrix{
\widetilde{\mathbb S}_{\ch\tau}\ar[r]^{\lambda^\alpha}\ar[d]_{p}&\widetilde{\mathbb S}^{\mathrm{quo}}_{\ch\tau'}\ar[d]_{\overline{p}_{\tilde g}}\ar[dr]^{\overline{p}}\\
\widetilde{\mathbb S}_{\phi}\ar[r]_{\iota_1}&\widetilde{\mathbb S}_{\phi_\beta}\ar[r]_{\iota_2}&\widetilde{\mathbb S}_{\phi'}.\\
}
\end{equation}
The right triangle commutes and it is the left square which we need to consider.
Let us begin with an element $\overline{u}$ of $\widetilde{\mathbb{S}}_{\ch{\tau}}$, and set~$\overline{u}' = \lambda^\alpha(\overline{u})$. 
By Lemma~\ref{lemm:technical_statements_dualside}(1), 
there exists an element~$\tilde u$ in the identity component of~$\ch{H}^{\alg, \ch{\tau}}$ determined by~$\overline{u}$ such that~$\tilde u$ also belongs to~$\ch{H}^{\alg, \ch{\tau}'}$.
As above
$\int(\ch{\tilde g}\inv)(\tilde u)\in \ch{G}_{\phi'}^{\mathrm{alg}}$; by Lemma~\ref{lemm:technical_statements_dualside}(2), we know that $\overline{p}_{{\tilde{g}}}(\lambda^{\alpha})(\overline{u})$ is the image of this element in the component group $\widetilde{\mathbb{S}}_{\phi'}$.
On the other hand, $\tilde u$ is in $\ch{G}_{\phi}^{\mathrm{alg}} \subset \ch{G}_{\phi'}^{\alg}$,
and $\iota_1(p(\overline{u}))$ is the image of this element in $\widetilde{\mathbb{S}}_{\phi'}$.

Thus what we have to show is that the elements $\int(\tilde g)\inv(\tilde u)$ and $\tilde u$ are contained in the same connected component of
$\ch{G}_{\phi'}^{\alg}$.
If ${\tilde{g}}$ and $\tilde u$ were both elements of  $\ch{G}_{\phi'}^{\alg}$, this would be immediate:
the component group  $\widetilde{\mathbb{S}}_{\phi'}$ is abelian,
and therefore the identity component contains the derived group.
(If $\Gamma$ is  a topological group  such that the component group $\Gamma/\Gamma_0$ is abelian,
then the morphism $\Gamma \to \Gamma/\Gamma_0$
factors through $\Gamma/[\Gamma,\Gamma]$, i.e. $[\Gamma,\Gamma]\subset \Gamma_0$.)

Although we know $\tilde u$ and $\int({\tilde g}\inv)(\tilde u)$ are in
$\ch{G}_{\phi'}^{\alg}=\mathrm{proj}^{-1}(\ch{L}_{\phi'}^{\ch{\tau}})$,
and $\ch g\in\ch L_{\phi'}$,
we have no guarantee that $\ch g$ is fixed by $\ch\tau$.
To handle this, use the Cartan decomposition
of $\mathrm{proj}\inv(\ch L_{\phi'})$ with respect to $\ch\tau$ to write
\begin{equation}
  \label{e:kX}
  \ch{\tilde g}=\ch{\tilde k}\exp(X)
\end{equation}
where the projection~$\ch{k}$ of~$\ch{\tilde{k}}$ in~$\ch{L}_{\phi'}$ satisfies $\ch\tau(\ch{k})=\ch{k}$, the element~$X$ of  $\ch{\mathfrak{l}}_{\phi'}$ satisfies \mbox{$\ch\tau(X)=-X$}, and we use the exponential map from~$\ch{\mathfrak{l}}_{\phi'}$ to $\mathrm{proj}\inv(\ch L_{\phi'})$.
Now set
$$
\ch{\tilde g}(t)=\ch{\tilde k}\exp(tX)
$$
for $t \in \R$, so that
$$
\ch{\tilde g}(0)=\ch{\tilde k},\quad \ch{\tilde g}(1)=\ch{\tilde g}.
$$
Now the argument of the preceding paragraph applies to prove that $\int(\ch{\tilde k})\tilde u$ and $\tilde u$ are in the same component
of $\ch{G}_{\phi'}^{\alg}$.

We claim $\int(\ch{\tilde g}(t))(\tilde u)$ is fixed by $\ch\tau$ for all $0\le t\le 1$,
so $\tilde k$ and $\ch{\tilde g}$ are in the same component of~$\ch G_{\phi'}^{\alg}$.
Therefore $\int(\tilde k)(\tilde u)$ and $\int(\ch{\tilde g})(\tilde u)$ are in the same component,
and putting these together we conclude the same holds for $\int(\ch{\tilde g})(\tilde u)$ and $\tilde u$, as required.

For the claim above, the condition that $\int(\ch{\tilde g}(t))(\tilde u)$ is fixed by $\ch\tau$ is equivalent to
$$
\ch{\tilde g}(t)\inv \ch{\tau}(\ch{\tilde g}(t))\in Z_{\ch G_{\phi'}^{\alg}}(\tilde u).
$$
Plugging in \eqref{e:kX} the left hand side is
$$
\exp(-2tX).
$$
The exponential map is injective when restricted to the $-1$-eigenspace of $\ch\tau$. Therefore this holds, independent of $t$,
if and only if $X$ is in the Lie algebra of the centralizer. This holds because this is the case at $t=1$: $\ch{\tilde g}(1)=\ch{\tilde g}$,
and $\int(\ch{\tilde g})(\tilde u)$ is fixed by $\ch\tau$.
This completes the proof of~(1).

Let us prove that $\iota$ is injective. Suppose~$x$ is an element of the kernel of $\iota$ and observe Diagram~\eqref{diagram_phi_tau}.
Any preimage of $x$ under the left vertical arrow must be in the kernel of $\overline{p}' \circ \lambda^\alpha$.
Recall from the discussion in \S \ref{sec:cayley_4} that the kernel of $p'$
is generated by the  the elements $\ch{\overline{m}}_\beta$ for those roots $\beta$ that are~$\ch{\tau}'$-real and~$\phi'$-singular.
If~$\beta$ is such a root, and is orthogonal to~$\alpha$, then it is~$\ch{\tau}$-real,
and $\ch{\overline{m}}_\beta \in \widetilde{\mathbb{S}}_{\ch{\tau}}$;
furthermore all~$\ch{\tau}$-real roots are obtained in this way, see~\cite[p.~200]{ABV}.
Thus the kernel of $\overline{p}' \circ \lambda^\alpha$ is generated by the $\ch{\overline{m}}_\beta$,
where~$\beta$ runs through the~$\ch{\tau}'$-real, $\phi'$-singular roots that are orthogonal to~$\alpha$.
Given $x \in \mathrm{Ker}(\iota)$, we deduce that $p^{-1}(x)$ must consist of products of such $\ch{\overline{m}}_\beta$.
But the corresponding roots~$\beta$ are~$\ch{\tau}$-real, and they must be $\phi$-singular
because of the relationship between the infinitesimal characters of $\phi$ and~$\phi'$. 
Therefore the  $\ch{\overline{m}}_\beta$ are already contained in $\mathrm{Ker}(p)$.
This shows that~$x$ must be the identity element of $\widetilde{\mathbb{S}}_{\phi}$.
\end{proof}

\subsection{The inductive algorithm}
\label{sec:finalize}
\label{sec:def_final}
\label{sec:description_finalize}

We can finally turn to the computation of lowest $K$-types. 

Let us begin with a complete Langlands parameter $(\phi, \chi)$.

After conjugating $\phi$ we may assume it is in standard form (Def.~\ref{d:stdform}),
and write \mbox{$\phi=\phi(\lambda,y)$}. Let  $\texttt{p}=(x,\Lambda)$
be the \texttt{atlas} parameter attached to~$(\phi, \chi)$ by the construction of Section~\ref{sec:pin_down_rep}. It is final (Definition~\eqref{e:domc}).
Let ~$\tau=\tau(\phi)$ be the involution of~$H$ attached to~$\phi$, so ~$x\in\X_\tau$.

Now, apply the following inductive procedure. 

We start with the set $\texttt{S}=\{\texttt{p}=(x,\Lambda)\}$, and update it as follows. ~\\

\begin{enumerate}[(1)]
\item[(0)] Apply the second case of Lemma \ref{l:deform1} to assume $\lambda|_{\a\cap\mathfrak z}=0$.
\item If every parameter in $\texttt{S}$ is tempiric the algorithm is finished. Otherwise go to step~(2).
\item Let $\textup{\texttt{p}}=(x,\Lambda)$ be the first non-tempiric element of $\texttt{S}$, and let $\Psires$ be the restricted roots of $H$ with respect to $\theta_x$.
Look for a simple root $\beta$ of $\Psires$
so that $\langle \lambda,\ch\beta\rangle\ne 0$. Replace $\textup{\texttt{p}}=(x,\Lambda)$ with $\textup{\texttt{p}}=(x,\Lambda_\beta)$ (cf. Lemma \ref{l:deform1}).
Also replace $\phi$ with $\phi(\lambda_\beta,y_\beta)$ as in Lemma~\ref{l:deform}. 
\item  Suppose there is a simple root $\alpha$ of $\Psiint(\lambda_\beta)$ of type $\texttt{C}^-$ for $\theta_{x,H}$. 
Replace $\textup{\texttt{p}}$ with $s_\alpha(\textup{\texttt{p}})$ and $\phi$ with $s_\alpha(\phi)$. Repeat this
 until there are no such roots, then go to step~(4). 
\item  Suppose $\alpha$ is a simple root for $\Psiint(\lambda_\beta)$ which is $\ch{\theta}_{y, \ch{H}}$-noncompact imaginary. 
Replace $\textup{\texttt{p}}$ with~$c_\alpha(\textup{\texttt{p}})$ (which may consist of two elements) and $\phi$ with $c^\alpha(\phi)$. Go to step (1).
\end{enumerate}

The fact that the algorithm terminates comes from the following
remarks.  First, the $\tau$-real roots form a root system $\Psi_{r,
  \tau}$, and the non-parity condition is a grading \cite{VoganIC4};
therefore if all the $\Psi_{r, \tau}$-simple roots fail the parity
condition, then all of the $\tau$-real roots must fail that
condition. Next, if every simple,  $\tau$-complex root is of type $\texttt{C}^+$
then every $\Psi_{r, \tau}$-simple root is simple \cite[Lemma 8.6.2]{VoganGreenBook}.
This implies that if
  $\texttt{p}_0$ is not final and all simple, $\tau$-complex singular roots
  are type~$\texttt{C}^+$, then there  exists a $\tau$-real simple
  root which does not satisfy the parity condition.
By Lemma \ref{l:independent} the parameter $s_\alpha(\texttt{p})$ is equivalent to $\texttt{p}$, and of course $s_\alpha(\phi)$ is conjugate to $\phi$. 
  Therefore, a
  finite number of steps of type (3) will lead to the situation of
  (4). Thus, after a finite number of iterations, there will remain no simple, singular, real or type 
  $\texttt{C}^-$ roots. This proves that the algorithm will terminate.

  The algorithm produces a set of parameters  with the following properties. These are of the form $\{(x_i,\Lambda')\mid i=1,\dots 2^N\}$,
where the $x_i$ are all conjugate to $x$ and contained in the same fiber~$\X_{\tau'}$, and $\Lambda'$ is a genuine
character of $\widetilde{H}_{\rho}(\mathbb{R}, \tau')$. These arise from the steps in the algorithm; for example each step produces
one or two {\tt KGB} elements which are conjugate to the preceding~ones.

Here is a more explicit description of the lowest $K$-types of $\pi$ in terms of the parameters produced by the algorithm. 
Let  $\xi$ be a strong involution representative of $x$,
and let~$X$ be a $(\mathfrak{g}, K_\xi)$-module such that $[\xi, X]=\pi(\phi, \chi)$.
For $i=1, \dots, r$, fix a strong involution representative~$\xi_{i,N}$ of~$x_i$.
The parameter $(x_i, \Lambda_N)$ determines a $(\mathfrak{g}, K_{\xi_{i,N}})$-module $Y_i'$,
as in~\S\ref{s:atlasparam}.
Since $\xi_{i,N}$ and~$\xi$ are $G$-conjugate, there is a canonical correspondence
between $(\mathfrak{g}, K_{\xi_{i,N}})$-modules and $(\mathfrak{g}, K_\xi)$-modules \cite[Prop.~3.1]{TwistedParameters}.
Therefore $(x_i, \Lambda_N)$  determines a unique $(\mathfrak{g}, K_\xi)$-module $Y_i$, which is  tempiric, and has a unique lowest $K_{\xi}$-type~$\mu_i$.
The set of lowest $K_\xi$-types of~$X$ is then precisely~$\{\mu_1, \dots, \mu_r\}$.

That the algorithm does produce the lowest $K_\xi$-types
of~$X$ follows from the following remarks. If~$X$ is the unique irreducible
quotient of
$\mathrm{Ind}_{M_{\mathbb{R}}A_{\mathbb{R}}N_{\mathbb{R}}}^{G_{\mathbb{R}}}(\sigma
\otimes \nu)$, where $\sigma$ is a limit of discrete series representation of $M_\R$ and $\nu$ is a unitary character of~$A_\R$, then the set of lowest $K$-types of~$X$ coincides with
that
of~$X_0=\mathrm{Ind}_{M_{\mathbb{R}}A_{\mathbb{R}}N_{\mathbb{R}}}^{G_{\mathbb{R}}}(\sigma
\otimes \mathbf{1})$: see \cite[Proposition 6.6]{HermitianFormsSMF}.
The representation~$X_0$ is a direct sum of tempiric ones,
and its irreducible constituents can be found in terms of Cayley transforms
by a repeated application of the Schmid character identities
(see the discussion in \cite[Section~6]{HermitianFormsSMF}).
The algorithm is designed so that the tempiric modules~$Y_i$ that it produces
are precisely the irreducible constituents of~$X_0$.

Here is a key property of the algorithm.
It produces a sequence $\phi=\phi_1,\phi_2,\dots, \phi'=\phi_n$ of $L$-homomorphisms in standard form, with the following properties.
For $i\le n-1$, there is an element $\ch g_i\in \ch L_{\phi_i}$ conjugating $\phi_i(j)$ to $\phi_{i+1}(j)$. 
Let $\ch g=\ch g_1\ch g_2\dots \ch g_{n-1}$. Then $\ch g\in \ch L_\phi$, and \mbox{$\int(\ch g)(\phi(j))=\phi'(j)$}. Furthermore $\phi_c$ and $\phi'$ have the same restriction to
the split part of $\ch\h$ with respect to $\ch{\tau}(\phi')$. Therefore $\int(\ch g)(\phi_c)=\phi'$. 

\begin{proof}[Proof of Theorem \ref{main_theorem_rgroups}]
  
  Let $\phi\colon \mathbf{W}_{\mathbb{R}}\to\lgr{G}$ be an $L$-homomorphism, and let $\phi_c$ be the homomorphism \eqref{e:phi_c}.
 Consider the morphism
 \begin{equation} \label{iota_phi_phic} \iota_{\phi, \phi_c}\colon \widetilde{\mathbb{S}}_\phi \to \widetilde{\mathbb{S}}_{\phi_c}\end{equation}
 induced by the inclusion $\phi_c(\mathbf{W}_{\mathbb{R}})\subset \phi(\mathbf{W}_\mathbb{R})$. Let 
  $\mathrm{Res}\colon \Char(\tSphic) \to \Char(\tSphi)$ be the map dual to~$\iota_{\phi, \phi_c}$. 
  
 After conjugating~$\phi$ we may assume it is in standard form. Fix a character $\chi$ of the group $\widetilde{\mathbb{S}}_\phi$
 and let $\texttt{p} = (x, \Lambda_\phi)$ be the \texttt{atlas} parameter attached to $(\phi, \chi)$ in Section~\ref{sec:pin_down_rep}.

 Apply the preceding algorithm. As discussed above this produces a sequence \mbox{$\phi=\phi_1,\dots, \phi_n=\phi'$} of $L$-homomorphisms,
and an element $\ch g$, such that $\int(\ch g)(\phi_c)=\phi'$. We have the following sequence of maps:
$$
\iota_{\phi,\phi'}\colon\widetilde{\mathbb{S}}_{\phi} \overset{\iota_{\phi,\phi_c}}\longrightarrow \widetilde{\mathbb{S}}_{\phi_c}\overset{\int(\ch g)}\longrightarrow \widetilde{\mathbb{S}}_{\phi'}.
$$
 where  $\iota_{\phi, \phi'}$ is the composition of the maps $\iota_{\phi_k, \phi_{k+1}}\colon \widetilde{\mathbb{S}}_{\phi_k} \to \widetilde{\mathbb{S}}_{\phi_{k+1}}$ induced at each step by composition with~$\ch{\tilde{g}}_k$.\footnote{In particular, by Proposition~\ref{prop:bij_char_KGB_param_7}(2), the map  $\iota_{\phi, \phi'}$ is injective, and therefore $\iota_{\phi, \phi_c}$ is also injective: this gives another proof of Proposition~\ref{prop:injectivity}.}

 Now let $\tau'=\tau(\phi')$, and let  $\varrho\colon \Char(\widetilde{\mathbb{S}}_{\phi'})\to \Char(\tSphi)$ be the map dual to $\iota_{\phi,\phi'}$.

  Each step in the algorithm gives rise to a diagram of the form~\eqref{diagram_cayley_final}, and
  composing all these diagrams horizontally, we get a commutative diagram
 \begin{equation}
\label{diagram_total}
\xymatrix{
\Char(\widetilde{\mathbb{S}}_{\phi'})\ar@{->>}[r]^{\varrho}\ar@{_{(}->}[d]_{\E_{\phi'}}&\Char(\widetilde{\mathbb{S}}_{\phi})\ar@{_{(}->}[d]^{\E_{\phi}}\\
\X_{\tau'}(*)\ar@{->}[r]^{}&\X_{\tau}
}
\end{equation}
where $\X_{\tau'}(*)$ is the subset of $\X_{\tau'}$ where all the appropriate diagrams can be composed (in other words, the subset on which all operations performed on \texttt{KGB} elements in the algorithm run can be inverted by imaginary Cayley transforms or complex cross actions).

By construction $\X_{\tau'}(*)$  contains all the elements $x_{i}$ from the output of the algorithm. Since the diagram commutes, the $x_i$ can be obtained by starting with the original \texttt{KGB} element~$x$ in~$\X_\tau$, considering its inverse image~$\chi$ in~$\Char(\tSphi)$, taking the fiber~$\varrho^{-1}(\{\chi\})$ in~$\Char(\widetilde{\mathbb{S}}_{\phi'})$, and pushing it down to $\X_{\tau'}(*)$ via $\mathcal{E}_{\phi'}$. Write $\chi'_{i}$ for the character of $\widetilde{\mathbb{S}}_{\phi'}$ corresponding to $x_i$ under $\mathcal{E}_{\phi'}$.
The representations corresponding to the complete Langlands parameters~\mbox{$(\phi', \chi'_{i})$} are precisely those that give the lowest $K$-types of the representation~$\pi(\phi, \chi)$, as discussed at the beginning of \S\ref{sec:finalize}.

Now, under conjugation by $\ch{\tilde{g}}$, the characters $\chi'_{i}$ correpond to characters $\chi_{i, c}$ of $\tSphic$.
Given the  above discussion of the relationship between~$\iota_{\phi, \phi'}$ and $\iota_{\phi, \phi_c}$,
the $\chi_{i,c}$ are precisely the characters of $\tSphic$ which map to $\chi$ under the restriction map
$\mathrm{Res}\colon \Char(\tSphic)\to \Char(\tSphi)$ dual to $\iota_{\phi, \phi_c}$, as in~\eqref{restriction_morphism}.
Thus the complete Langlands parameters $(\phi_c, \chi_{i,c})$, for $\chi_{i,c}$ in $\mathrm{Res}^{-1}(\{\chi\})$,
are precisely those which parametrize  to the lowest $K$-types of $\pi(\phi, \chi)$.  This completes the proof.
\end{proof}

\section*{Appendix. Whittaker data, generic representations, and atlas basepoints}\label{sec:appendix}

We work in the setting of Section \ref{sec:structure_theory}: we are given $G$, an inner class determined by $\gamma\in\Out(G)$, and we fix a pinning
$(T,B,\{X_\alpha\})$
for $G$.
Let us fix an involution $\theta$ of $G$ in the inner class and  let $K=G^\theta$. We say $\theta$, or  $(\g,K)$, is {\it quasisplit} if
the corresponding real form $G(\R)$ is quasisplit (see \cite{AV1} for other characterizations).
There is a unique conjugacy class of quasisplit involutions in each inner class \cite[Theorem 6.14]{AV1}.
A representative is $\theta=\int(\xi_0)$ where $\xi_0=e^{\pi i\ch\rho}\xi_{\gamma}$. See \S\ref{sec:basepoint}. 

It is most natural to work entirely in the algebraic setting, in which case the preferred representation in an $L$-packet
(corresponding to the trivial character) has an algebraic Whittaker model. We give the definitions and properties here,
and make the connection with classical Whittaker models in Section~\href{sec:real_whittaker}{A.2}.

\subsection*{A.1. Algebraic Whittaker models}\label{sec:alg_whittaker}

The involution $\theta$ acts on the Lie algebra $\g$ and its vector space dual $\g^*$;
write $\g^*=\k^*\oplus\s^*$ where $\s$ is the $(-1)$-eigenspace of $\theta$. 
Let $\sprin$ be the intersection of $\s^*$ with the principal nilpotent $G$-orbit. 
This is nonempty  if and only if $(\g,K)$ is quasisplit.
We define an {\it algebraic Whittaker datum}  for $(\g,K)$ to  be a $K$-orbit on $\sprin$.
If $\pi$ is a $(\g,K)$-module of finite length, then the associated variety $AV(\pi)$ is a union of $K$-orbits on~$\s^*$.
If $\O\subset \sprin$ is an algebraic Whittaker datum, we say $\pi$ has an {\it algebraic Whittaker model of type $\O$}
if $\O\subset AV(\pi)$. We say $\pi$ is {\it large} if it has an algebraic Whittaker model of type~$\O$ for some $\O \subset \sprin$.

Set $y=e^{\pi i\rho}\ch\delta$, where~$\ch{\delta} \in \lgr{G}$ is used to define the $L$-group in Section~\ref{sec:Lgroup}. 
Recall there is a special $L$-homomorphism $\phi_0=(\rho,y)$, defining an $L$-packet~$\Pi(\phi_0)$ of fundamental series (see Example~\ref{ex:fundamental}).
Furthermore $(\phi_0,\mathbf{1})$ determines a special fundamental series $\pi_0\in \Pi(\phi_0)$. In our setting (Section~\ref{sec:local_langlands}) this can be viewed canonically as an irreducible $(\g, K)$-module (compare \S\ref{sec:def_representations} and~\cite[Prop.~3.1]{TwistedParameters}), and arises as the unique irreducible quotient of a standard module $I(\phi_0, \mathbf{1})$.

\begin{enonce*}{Lemma A.1}
The $(\g, K)$-module $\pi_0$ is a large fundamental series for the quasisplit pair~\mbox{$(\g,K)$}.
The associated variety $AV(\pi_0)$ is the closure of a single $K$-orbit on $\sprin$.
\end{enonce*}

\begin{proof}
If $\alpha$ is a simple $\xi_0$-imaginary  root then  $\int(\xi_0)(X_\alpha)=\int(e^{\pi i\langle \alpha,\ch\rho\rangle})=-1$.
This says that $\alpha$ is $\xi_0$-noncompact. By \cite[Theorem 6.2(f)]{Vogan_GK}, the $(\g, K)$-module $\pi_0$ is large.
 The second statement follows from the fact that the associated variety of a fundamental series representation is the closure of
 a single $K$-orbit. See \cite[Proposition A.9]{AV1}.
\end{proof}

We use this to specify an  algebraic Whittaker model. Let us single out the $K$-orbit corresponding to~$\pi_0$ by the Lemma, and define
\begin{equation}
\label{d:O0}\tag{A.2}
\O_0=AV(\pi_0).
\end{equation}

\begin{enonce*}{Proposition~A.3}
\label{p:algwhittaker}
Let~$\phi$ be an $L$-homomorphism, and let~$\mathbf{1}$ be the trivial character of $\widetilde{\mathbb{S}}_\phi$. Then the standard module~\mbox{$I(\phi,\mathbf{1})$} has an algebraic Whittaker model 
of type $\O_0$.

In particular, if $\phi$ is tempered then $\pi(\phi,\mathbf{1})$ has an algebraic Whittaker model of type $\O_0$.
\end{enonce*}

Recall in our setting, we start with an \texttt{atlas} parameter $(x,\Lambda)$ to define a
$(\g,K_x)$-module. Any complete Langlands parameter $(\phi, \chi)$ gives rise to an atlas parameter $(x, \Lambda)$, and in case $\chi=\mathbf{1}$ the element~$x$ is the basepoint in its fiber.  In order to prove the Proposition we need to
relate the corresponding~$(\g, K_x)$-module to our fixed quasisplit pair $(\g,K)$, where
$\theta=\theta_{\xi_0}$ and $K=G^{\theta}$ as above, using conjugation by some element $g\in G$.

The main result we need is the following. For $x$ a {\tt KGB} element and $B$ a Borel subgroup containing $H$, let
$\Sigma(x,B)$ be the set of $\theta_{x,H}$-imaginary roots in $B$.

\begin{enonce*}{Lemma A.4}
\label{l:Q}
Let $\xi_0$ and $\theta=\theta_{\xi_0}$ and $K$ be as above.
Suppose $x$ is a {\tt KGB} element with trivial torus part. 
Choose a strong involution $\xi$ representing $x$. 

Let $L=\Cent_G((H^{\theta_\xi})^0)$ and choose a $\theta_\xi$-stable parabolic subgroup  $Q=LU$ such that 
\begin{equation}
\label{e:Sigma1}
\Sigma(x,B)\subset U.
\end{equation}
Then there exists $g \in G$ satisfying 
\begin{equation}
\label{e:Sigma2}
g\xi_0g\inv=\xi\text{ and  }\Sigma(x_0,B)\subset g\inv Qg.
\end{equation}
\end{enonce*}

\begin{proof}[Proof of the Proposition, given the Lemma]
First it is convenient to note, since we can modify~$g$ by an element of $K$ on the right, that \eqref{e:Sigma2} is equivalent to the statement: 
\begin{equation}
  \label{e:Sigma3}
  \tag{\ref{e:Sigma2}$'$}
  \exists g\in G\text{ satisfying: }g\xi_0g\inv=\xi\text{ and  }\Sigma(x_0,B)\text{ is $K$-conjugate to a subset of }g\inv Qg.
\end{equation}
  
Let~$(x, \Lambda)$ be the atlas parameter attached to $(\phi, \mathbf{1})$ as in Section~\ref{sec:local_langlands},
with corresponding standard module $I$.
Fix a strong involution $\xi$ representing $x$ and choose $g$ as in the Lemma. 
We have to show that, after conjugating by $g\inv$, the corresponding standard $(\g,K)$-module $I$ satisfies $\O_0\subset AV(I)$.
In the notation of \cite{AV1} the corresponding standard module $I$ is the $(\g,K_\xi)$-module $I(B,\Lambda)$.

Let $H'=g\inv Hg,B'=g\inv Bg$. 
Note that $\int(g\inv)$ takes the action of $\theta_\xi=\theta_{\xi,H}$ on $H$ to the action of $\theta=\theta_{\xi_0}$ on $H'$:
$$
g\inv \theta_\xi(h)g=\theta(g\inv hg)\quad (h\in H).
$$

Write $\Sigma(x,H,B)$ to emphasize the role of $H$. After conjugating by $g\inv$ in \eqref{e:Sigma1} the two conditions
of the Lemma become
$$
\begin{aligned}
\Sigma(x_0,H',B')&\subset g\inv Ug,\\
\Sigma(x_0,H,B)&\text{ is $K$-conjugate to a subset of }g\inv Qg.
\end{aligned}
$$
Now both statements are about $(g,K)$-modules, and the $(\g,K_\xi)$-module  $I(B,\Lambda)$ has been replaced  by the $(\g,K)$-module $I(B',g\inv\Lambda'g)$.
Since $Q$ is $\theta_\xi$-stable, $g\inv Qg$ is $\theta$-stable.

Thus we are in precisely the setting of 
\cite[Theorem A.10]{AV1}, with: $H'$ in place of $T$ in \cite{AV1};
$g\inv Qg$    in place of $Q$;  $H$ in place of $T_c$;
and $\Sigma(x_0,H,B)$ in place of $\sigma(T_c,\Sigma_c)$. 
Note that $B$ is a $\theta$-stable Borel subgroup, and
if $(x_0,\Lambda_0)$ is any parameter then $AV(I(x_0,\Lambda_0))=\O_0$ (see \cite[A.11]{AV1}).
By \cite[Theorem A.10]{AV1}
we have $\O_0\subset AV(I)$ as required.
\end{proof}

\begin{proof}[Proof of the Lemma]
  Let~$x_0$ and $x$ be the \texttt{KGB} elements corresponding to~$\xi_0$ and $\xi$, respectively. Each of these is the basepoint in its fiber. We can pass from~$x_0$ to~$x$ by series of simple complex cross actions, and simple noncompact imaginary Cayley transforms (see \cite[Section 14]{Algorithms}).
  Each of these preserves the property of the torus part being $0$, and it shows the existence of $g$ satisfying $gx_0g\inv =x$ (see
  \eqref{def_basepoint}).
  It is enough to show the Lemma holds at each step. 
  
So assume the Proposition holds for $x$, and write $Q=Q_x=L_xU_x$ as in the Lemma. Thus by the inductive hypothesis we can choose $g_x$ such that $g\xi_0g\inv = \xi$ and
\begin{equation}
  \label{e:step}
  \begin{aligned}
    \Sigma(x,B)&\subset U_x,\\
    \Sigma(x_0,B)&\subset g_x\inv Q_xg.
  \end{aligned}
\end{equation}
We need to show the same holds with $x$ replaced by $y=s_\alpha\times x$ for $\alpha$ simple, $\theta_x$-complex, or by $y=c^\alpha(x)$ with $\alpha$ simple, $\theta_x$-noncompact imaginary.

Consider the case of a simple cross action. Let $\eta$ be a strong involution representing~$y=s_\alpha \times x$. Let  $g_\alpha\in \Norm(H)$ representing $s_\alpha$ such that $\eta = g_\alpha\xi g_\alpha^{-1}$. 
Let $Q_y=g_\alpha Q_xg_\alpha\inv$. This is a $\theta_\eta$-stable parabolic sugroup.

  Now $s_\alpha$ takes the $\theta_x$-imaginary roots to the $\theta_y$-imaginary roots.
  Since $\alpha$ is simple it permutes the $B$-positive roots; since it is not imaginary takes $\theta_x$-imaginary positive roots to positive roots.
  Thus we have $s_\alpha(\Sigma(x,B))=\Sigma(y,B)$. Therefore by  \eqref{e:step},
  $$
  \Sigma(y,B)=s_\alpha(\Sigma(x,B))\subset s_\alpha U_x=U_y
  $$
  Set $g_y=g_\alpha g_x$. Then $g_y\xi_0g_y\inv =\eta$, and $g_y\inv Q_yg_y= (g_\alpha g_x)\inv(g_\alpha Q_xg_\alpha\inv)(g_\alpha g_x)=g_x\inv Qg_x$.
  Therefore
  $$
  \Sigma(x_0,B)\subset g_y\inv Q_yg_y.
  $$
  This proves that condition \eqref{e:step} holds for $y$ as required.

  Now suppose $\alpha$ is a simple, $x$-noncompact imaginary root and let $y=c^\alpha(x)$.
  The simplest proof of the Lemma in this case is representation theoretic.
  Let $(x,\Lambda)$ be any parameter (with given $x$) and let $(y,\Lambda')$ be its Cayley transform.

  We now apply the Hecht--Schmid identity \cite[Corollary 8.4.6]{VoganGreenBook}, which says that
  (on the level of Grothendieck groups)
  \begin{equation}
    \label{e:groth}
  I(y,\Lambda')=s_\alpha I(x,\Lambda) + I'
  \end{equation}
  where $I'$ is a certain standard module (it doesn't matter which) and $s_\alpha$ is the coherent continuation action.
  Then $\O_0\subset AV(I(x,\Lambda))\Rightarrow \O_0\subset AV(s_\alpha(I(x,\Lambda))\Rightarrow \O_0\subset AV(I(y,\Lambda'))$.
(To see the first implication, choose $\Lambda$ as in the proof of \cite[Theorem A.11]{AV1} so that the large composition factors of
  $I(x,\Lambda)$ are large fundamental series. Then  by  \cite[Theorem 7.3.16(b)]{AV1}
each of these large fundamental series occurs in $s_\alpha(I(x,\Lambda))$.)

Now apply \cite[Theorem A.10]{AV1} again, in the reverse direction, to conclude
  $\Sigma(y,B)\subset U$ and $\Sigma(x_0,B)$ is $K$-conjugate to a subset of $g\inv Qg$.
  Therefore \eqref{e:Sigma1} and \eqref{e:Sigma3} hold for~$y$, as required.
\end{proof}

\subsection*{A.2. Real Whittaker models}\label{sec:real_whittaker}

The notion of algebraic Whittaker model is equivalent, in a precise
sense, to the usual notion of Whittaker model.
Here we give the statements of the results in our setting. The proofs are mainly an issue of putting together
the references \cite[Chapter 14]{ABV}, \cite{kostant_whittaker} and \cite{Vogan_GK}; for the discrete series case, see~\cite{AA_Whittaker}.

Let $G(\R)$ be the quasisplit form of $G$.
A  \emph{real Whittaker datum for $G(\mathbb{R})$} is a pair $\w=(B, \psi)$,
where $B$ is a Borel subgroup of $G$ defined over $\mathbb{R}$
and $\psi$ is a non-degenerate character of $N(\mathbb{R})$. 
{\it Non-degenerate} means that the restriction of $\psi$ to  each root space ~$\mathfrak{g}_\alpha(\mathbb{R})$, $\alpha \in S$, is nontrivial.
We will use the term {\it Whittaker datum} for real Whittaker datum,  and always use  {\it algebraic Whittaker datum}  in the algebraic setting of the previous subsection.
Equivalence of (real) Whittaker data is given by conjugacy by $G(\R)$.

We refer to \cite[Section 3]{ABV} for the notion that a 
representation of $G(\R)$ has a Whittaker model of type $\mathfrak
w$.
We say that a representation is \emph{$\mathfrak w$-generic} if it
has a Whittaker model of type $\mathfrak w$, and {\it generic} if
it is $\mathfrak w$-generic for some $\mathfrak w$. The equivalence class of~$\mathfrak{w}$ is uniquely specified by the representation.

Let $\xi=\xi_0,\theta=\theta_{\xi_0}, K=G^\theta$ be as in the previous section, so   $(\g,K)$ is quasisplit.
Let $G(\R)$ be a corresponding real form of $G$. This means: $\sigma$ is an anti-holomorphic involution of $G$, commuting with $\theta$,
and $G(\R)=G^\sigma$. Any two such groups are conjugate by $K$.

Any irreducible or standard module $(\g,K)$-module $\pi$ determines an irreducible or standard
Hilbert space representation $\pi_\R$ of $G(\R)$.
If $\phi$ is a Langlands parameter, we write $\Pi(\phi)$ for the corresponding $L$-packet of $(\g,K)$-modules (this is the ``classical packet'' considered in Section~\ref{sec:disjointness}), and $\Pi_{\R}(\phi)$ for the corresponding finite set of representations of~$G(\R)$.

Let $G_{\mathrm{ad}}=G/Z(G)$ be the adjoint group.
This is the group of inner automorphisms of $G$.
It is defined over $\mathbb{R}$, and $G_{\mathrm{ad}}(\mathbb{R})$
is the group of inner automorphisms of $G$ which are defined over~$\mathbb{R}$.
This contains the subgroup $\mathrm{Ad}(G(\mathbb{R}))=G(\mathbb{R})/Z(G(\mathbb{R}))$
of inner automorphisms of $G(\mathbb{R})$.
We denote by $Q(G(\mathbb{R}))$ the quotient $G_{\mathrm{ad}}(\mathbb{R})/\mathrm{Ad}(G(\mathbb{R}))$.
This is a finite group.

\begin{enonce*}{Lemma A.5}\label{p:dictionary_whittaker}
Suppose $\phi$ is a fundamental series Langlands parameter. The above discussion establishes canonical bijections between:
\begin{enumerate}
\item[\textup{(1)}] The large $(\g,K)$-modules in  $\Pi(\phi)$;
\item[\textup{(2)}] The set of algebraic Whittaker data for $(\g,K)$;
\item[\textup{(3)}] The set of equivalence classes of Whittaker data for $G(\R)$;
\item[\textup{(4)}] The generic  representations in $\Pi_\R(\phi)$.
\end{enumerate}
The bijection \textup{(1)$\leftrightarrow$(2)} attaches to a large $(\g, K)$-module the corresponding algebraic Whittaker datum, and \textup{(4)$\leftrightarrow$(3)} attaches to a large $(\g, K)$-module the corresponding equivalence class of Whittaker data. The bijection \textup{(1)$\leftrightarrow$(4)} is induced by passage from~$\pi$ to~$\pi_{\R}$. The resulting bijection \textup{(2)$\leftrightarrow$(3)} is independent of the choice of $\phi$.
The group $Q(G(\R))$ has canonical simply transitive actions on all these sets, and those commute with the bijections.
\end{enonce*}

Write $\O\mapsto \w(\O)$ for the bijection between algebraic and real Whittaker data. When we move from fundamental series to finite-length $(\g, K)$-modules, we have a more general statement:

\begin{enonce*}{Lemma A.6}
Let~$\pi$ be a finite-length $(\g,K)$-module. Then $\pi$ has an algebraic Whittaker model of type $\O$ if and only if $\pi_\R$ has a Whittaker model of type $\w(\O)$.
This correspondence  commutes with the action of $Q(G(\R))$.
\end{enonce*}

In particular $\pi$ is large if and only if $\pi_\R$ is generic.

Recall we have fixed an algebraic Whittaker datum $\O_0$ for $(\g,K)$ by setting
$\O_0=AV(\pi_0)$ (see~\eqref{d:O0}). Let $\w_0=\w(\O_0)$ be the corresponding Whittaker datum for $G(\R)$.
The analogue of Proposition \href{p:algwhittaker}{A.3} is now clear, and goes as follows. Given an $L$-homomorphism~$\phi$, we have the $(\g, K)$-module $I(\phi,\mathbf{1})$ as in Section~\ref{sec:pin_down_rep}, and a corresponding representation~$I_{\R}(\phi, \mathbf{1})$ of $G(\R)$. Combining Proposition \href{p:algwhittaker}{A.3} and Lemma~\href{p:dictionary_whittaker}{A.5} we obtain:

\begin{enonce*}{Proposition A.7}
\label{p:realwhittaker}
Let~$\phi$ be an $L$-homomorphism, and let~$\mathbf{1}$ be the trivial character of $\widetilde{\mathbb{S}}_\phi$. Then the standard module $I_\R(\phi,\mathbf{1})$ has a Whittaker model of type $\w_0$.
If $\phi$ is tempered, then $\pi_\R(\phi,\mathbf{1})$ has a Whittaker model of type $\w_0$.
\end{enonce*}

\bibliographystyle{smfplain}
\bibliography{jda_aa}
\end{document}